\documentclass[a4paper,12pt]{article}
 \usepackage{graphicx}  
\usepackage{setspace}  
\usepackage{geometry}  
\usepackage{lipsum}    
\usepackage{epsf}
\usepackage{amsbsy,amsmath}
\usepackage{mathtools}
\usepackage{mathrsfs}
\usepackage{amsfonts}
\usepackage{amssymb}
\usepackage{enumitem}
\usepackage{eucal}
\usepackage{enumitem}
\usepackage{graphics,mathrsfs}
\usepackage{amsthm}
\usepackage{secdot}
\usepackage{xcolor}
\usepackage{bm}
\usepackage{titlesec}
\usepackage{setspace}
\usepackage{caption}  
\usepackage{algorithm}
\usepackage[noend]{algpseudocode}
\usepackage{tcolorbox}
\usepackage{mathrsfs}
\usepackage{extarrows}

\renewcommand{\figurename}{FIGURE}

\makeatletter
\renewcommand{\fnum@figure}{\figurename~\thefigure\hspace{1em}}
\makeatother

\usepackage{graphicx}
\usepackage{soul}
\usepackage{todonotes}
\usepackage{subcaption}

\graphicspath{{images/}}

\allowdisplaybreaks

\newtheorem{theorem}{Theorem}[section]
\newtheorem{lemma}[theorem]{Lemma}
\newtheorem{proposition}[theorem]{Proposition}
\newtheorem{corollary}[theorem]{Corollary}

\newtheorem{remark}[theorem]{Remark}
\newtheorem{defin}[theorem]{Definition}

\theoremstyle{definition}

 \numberwithin{equation}{section}

\numberwithin{equation}{section}

 \newsavebox{\savepar}

\usepackage{hyperref}
\hypersetup{
    colorlinks=true,
    linkcolor=blue,
    filecolor=magenta,      
    urlcolor=cyan,
    citecolor=purple,
    }

\title{\begin{center}
	A sparse spectral method on a class of domains bounded by planar algebraic curves
\end{center}}

\author{
Jiajie Yao\textsuperscript{a,}\thanks{Corresponding author: \url{jy219@leicester.ac.uk}. The first author is grateful for the financial support of a Future 50 PhD scholarship from the University of Leicester.
} , 
Marco Fasondini\textsuperscript{a}, and
Sheehan Olver\textsuperscript{b}
}

\date{
\small
\textsuperscript{a}\textit{School of Computing and Mathematical Sciences,
University of Leicester, UK}\\
\textsuperscript{b}\textit{Department of Mathematics,
Imperial College London, UK}
}

\begin{document}
\captionsetup{font=small}
\maketitle

\begin{abstract}
	We develop a sparse spectral method for solving partial differential equations on a  class of two-dimensional geometries bounded by algebraic curves.
	The numerical method uses generalised bivariate Koornwinder polynomials which form a complete orthogonal  basis, but one which is {\it not graded} in terms of polynomial degree.  The polynomials are built from new families of univariate semiclassical orthogonal polynomials whose associated operator matrices (Jacobi matrices, raising matrices and differentiation matrices) are computed with optimal linear complexity in the number of basis functions.
When used to discretise partial differential equations the resulting matrices are sparse enabling efficient numerical solution. Moreover, we develop fast transforms from values on a grid to expansion coefficients. 
	The efficiency and accuracy of the resulting spectral  method are illustrated through a series of numerical experiments on geometries whose boundaries are smooth and piecewise smooth including non-convex geometries. We observe algebraic convergence for geometries with corners, which accelerates to  exponentially fast (spectral) convergence when the boundary is smooth.
\end{abstract}

\section{Introduction}\label{sect:intro}
Univariate classical orthogonal polynomials (OPs) have been well studied, with their definitions, properties, computation and applications documented in, e.g.,~\cite[\href{https://dlmf.nist.gov/18}{Ch.~18}]{NIST}. 
Classical OPs are orthogonal with respect to a weight function $w(x)$  which satisfies the Pearson equation
\begin{align}
	\frac{\mathrm{d}}{\mathrm{d}x} (\sigma(x) w(x)) = \tau(x) w(x), \label{Pearson}
\end{align}
where $\sigma$ and $\tau$ are polynomials with $\deg \sigma \leq 2$ and $\deg \tau = 1$. 
Semiclassical orthogonal polynomials are defined through a weight function $w(x)$ that satisfies the Pearson equation (\ref{Pearson}), but without  restrictions on the degrees of $\sigma$ and $\tau$~\cite{Semiclassical}. 

Orthogonal polynomials can be extended to multiple variables, as comprehensively discussed in \cite{DunklXu}. One can construct bivariate OPs from univariate OPs via Koornwinder's formula \cite{Koornwinder}:

\begin{theorem}[\cite{DunklXu}, Section~2.6.1] \label{thm:koornwinderbasis}
	Let $w_1(x)$ and $w_2(x)$ be weight functions on $[\alpha,\beta]$ and $[\gamma,\delta]$, respectively, and let $\rho(x)$ be a positive function on $(\alpha,\beta)$, where either
	\begin{itemize}
		\item[ ]\textup{case (i)}: $\rho$ is a degree 1 polynomial, or
		\item[] \textup{case (ii)}: $\rho$ is the square root of a polynomial of degree $\leq 2$,  $\delta=-\gamma>0$ and $w_2$ is an even function.
	\end{itemize}
For $0 \leq k \leq n$ and $n \in \mathbb{N}_0 := \{0,1,2,\dots\}$, let $p_{n,k}$ be a univariate OP of degree $n$ with respect to the weight $\rho(x)^{2k+1}\,w_1(x)$ and let $q_{n}$ be a univariate OP of degree $n$ with respect to the weight $w_2(x)$, then the polynomials (referred to as Koornwinder polynomials) given by 
	\begin{align}\label{KoornwinderOPs}
		H_{n,k}(x,y) := p_{n-k,k}(x) \, \rho(x)^k \, q_{k}\left(\frac{y}{\rho(x)}\right), \quad 0 \leq k \leq n,
	\end{align}
	are bivariate OPs with respect to the weight $W(x,y) := w_1(x)\,w_2\left(\displaystyle{\frac{y}{\rho(x)}}\right)$  on the domain
	\begin{equation}
		\Omega := \{(x,y)\  | \ \gamma\rho(x) \leq y \leq \delta\rho(x), \, \alpha\leq x \leq \beta\},  \label{eq:Omega}
	\end{equation} 
referred to as a Koornwinder domain, hence
\begin{equation}
\left\langle H_{n,k}, H_{m,\ell}	  \right\rangle_W = \delta_{n,m}\delta_{k,\ell}	\| H_{n,k}	  \|_W^2 \label{eq:Koornip}
\end{equation}
for
\[
\left\langle f, g  \right\rangle_W := \iint_{\Omega}f(x,y) g(x,y)W(x,y)\mathrm{d}x\mathrm{d}y\quad \hbox{and}\quad   \|f	  \|_W :=  \sqrt{\left\langle f, f  \right\rangle_W}.
\]
\end{theorem}

The following are well-known special cases of Koornwinder domains: 
\begin{itemize}
	\item the unit square, i.e., $\Omega = [-1, 1]^2$, for which $\rho(x)=1$ and $[\alpha, \beta] = [-1, 1] = [\gamma, \delta]$ in (\ref{eq:Omega});  
	\item the unit disc, for which $\rho(x)=\sqrt{1-x^2}$ and $[\alpha, \beta] = [-1, 1] = [\gamma, \delta]$;  
	\item the triangle $\Omega = \lbrace (x,y) : 0 \leq x \leq 1,0 \leq  y \leq 1 - x \rbrace$, for which  $\rho(x)=1-x$ and $[\alpha, \beta] = [0, 1] = [\gamma, \delta]$.  
\end{itemize}

If $w_1$ and $w_2$ are classical weight functions\footnote{That is, $w_1(x) = (x-\alpha)^a(\beta - x)^b$ and $w_2(x) = (x-\gamma)^c(\delta - x)^d$, for  $a, b, c, d > -1$.} and all the roots of $\rho$ (in case (i)) or $\rho^2$ (in case (ii)) are at $\alpha$ or $\beta$ (the triangle and disc are examples of this), then the univariate OPs $\lbrace p_{n,k} \rbrace$ and $\lbrace q_n\rbrace$ defined  in Theorem~\ref{thm:koornwinderbasis} are classical OPs and thus one obtains bivariate OPs from products of (univariate) classical OPs.  
 Hence we refer to Koornwinder domains (\ref{eq:Omega}) such as the square, disc and triangle, for which only classical OPs can be used to construct bivariate OPs, as classical geometries or classical elements. A characteristic property of classical OPs is that their derivatives are also classical OPs~\cite[\href{https://dlmf.nist.gov/18.9}{Section 18.9}]{NIST}.
This is one of the key properties that enables the construction of sparse spectral methods for solving partial differential equations (PDEs) on classical geometries \cite{DiscOPs,TriangleOPs}.  

Sparse spectral methods are devised using sparse matrix representations of operators (such as the multiplication, change-of-basis and differentiation operators) acting on OPs.  These sparse matrix representations are referred to as operator matrices and will be discussed in Section~\ref{sect:classopmats}.
 Since bivariate Koornwinder polynomials are products of univariate OPs, the entries of the operator matrices of Koornwinder polynomials can be expressed in terms of the entries of operator matrices of univariate OPs.  Hence, on classical 2D geometries, all the operator matrices that are required for a sparse spectral method can be constructed from entries of operator matrices of classical OPs.

If $\rho$ (in case (i)) or $\rho^2$ (in case (ii)) has at least one root off the support of $w_1$, i.e., at least one root in $\mathbb{C}\setminus [ \alpha, \beta]$, then the weight $w_1 \rho^{2k+1}$ of the OPs $\lbrace p_{n,k}  \rbrace$ is not a classical weight but rather a semiclassical weight, as defined below (\ref{Pearson}), which we shall prove in Proposition~\ref{Moment_Recurrence_Relation}.  
  We refer to Koornwinder domains (\ref{eq:Omega}) for which semiclassical OPs are used to construct the Koornwinder polynomial basis, as semiclassical geometries or semiclassical elements.  In~\cite{SnowballOlver}, a sparse spectral method is presented for the following semiclassical geometries:
 \begin{itemize}
	\item disc slices, which have $\rho(x)=\sqrt{1-x^2}$, $[\alpha, \beta] \subset [-1,1]$ and $[\gamma, \delta] = [-1, 1]$ in (\ref{eq:Omega}); 
	\item trapeziums, defined by $\rho(x)=1-c x$, $0 < c < 1$ and $[\alpha, \beta] = [0, 1] = [\gamma, \delta]$  in (\ref{eq:Omega}).  
\end{itemize}

Compared to classical OPs, whose operator matrices are known explicitly, the computation of the operator matrices of semiclassical OPs poses greater challenges.    In~\cite{SnowballOlver}, the Lanczos algorithm, Christoffel--Darboux formula and quadrature are used to compute the operator matrices of the semiclassical OPs $p_{n,k}$ defined in Theorem~\ref{thm:koornwinderbasis}, which in turn are used to compute the operator matrices of the Koornwinder polynomial basis.  This has $\mathcal{O}\left(N^3\right)$ complexity if we let $0 \leq n \leq N$ in (\ref{KoornwinderOPs}), where the number of degrees of freedom are $(N+1)(N+2)/2 \sim N^2/2$.   
By contrast, the approach we shall use to compute the operator matrices has optimal (linear) complexity in the degrees of freedom, i.e., $\mathcal{O}(N^2)$ complexity.
  
One can express the semiclassical weight $w_1 \rho^{2k+1}$ of the OPs $p_{n,k}$ as a product of a classical weight and a \emph{modification function} $\phi$, i.e., $w_1 \rho^{2k+1} = (x - \alpha)^a(\beta - x)^b\phi(x)$, where the zeros or singularities of $\phi$ are in $\mathbb{C}\setminus [\alpha, \beta]$.  
In \cite{Gutleb}, OPs whose weights are products of classical weights and  polynomial or rational modification functions are related to classical OPs via \emph{connection matrices} which are computed in optimal complexity with matrix factorisations, as we shall discuss in Section~\ref{sect:semiclassconnmats}.  This matrix factorisation approach is a more stable alternative to the methods used in~\cite{SnowballOlver} for computing the operator matrices of semiclassical OPs and provides an efficient and quadrature-free way to construct sparse spectral methods on semiclassical geometries \cite{Papadopoulos1}.  
  For OPs with weight modifications that are neither polynomial nor rational, a moment-based algorithm is proposed in \cite{FastGram} to compute connection matrices.  We shall use both the matrix factorisation and moment-based algorithms to compute connection matrices and thereby compute operator matrices in optimal complexity.

In this paper, we extend the collection of 2D geometries bounded by algebraic curves on which one can use Koornwinder polynomials to construct a sparse spectral method.   This will be accomplished via \textit{generalised Koornwinder polynomials} defined on \textit{generalised Koornwinder domains}, as detailed in Section~\ref{sect:generalisedKoornwinder}.   Examples  of PDE solutions on generalised Koornwinder domains are shown in Figure~\ref{fig:approx_solutions_Poisson_Helmholtz}.  
  In Sections~\ref{sect:classopmats} and~\ref{sect:semiclassconnmats}, we describe the computation of operator matrices of the univariate OPs in Koornwinder's formula which are then used in Section \ref{Operator_matrices _Koornwinder_polynomials} to construct the multiplication, differentiation and conversion operator matrices for generalised Koornwinder polynomials and we deduce their sparsity structures. 
Then we consider fast forward and inverse transforms in Section~\ref{sec:fasttransforms}, 
and develop a sparse spectral method for solving PDEs in Section \ref{Koornwinder_Spectral_Methods}.  The \textsc{Julia} code for the sparse spectral method on generalised Koornwinder domains and for reproducing the numerical experiments in this paper is available at \cite{YJJCode}.
      
\begin{figure}[h]
    \centering
    \includegraphics[width=0.345\linewidth]{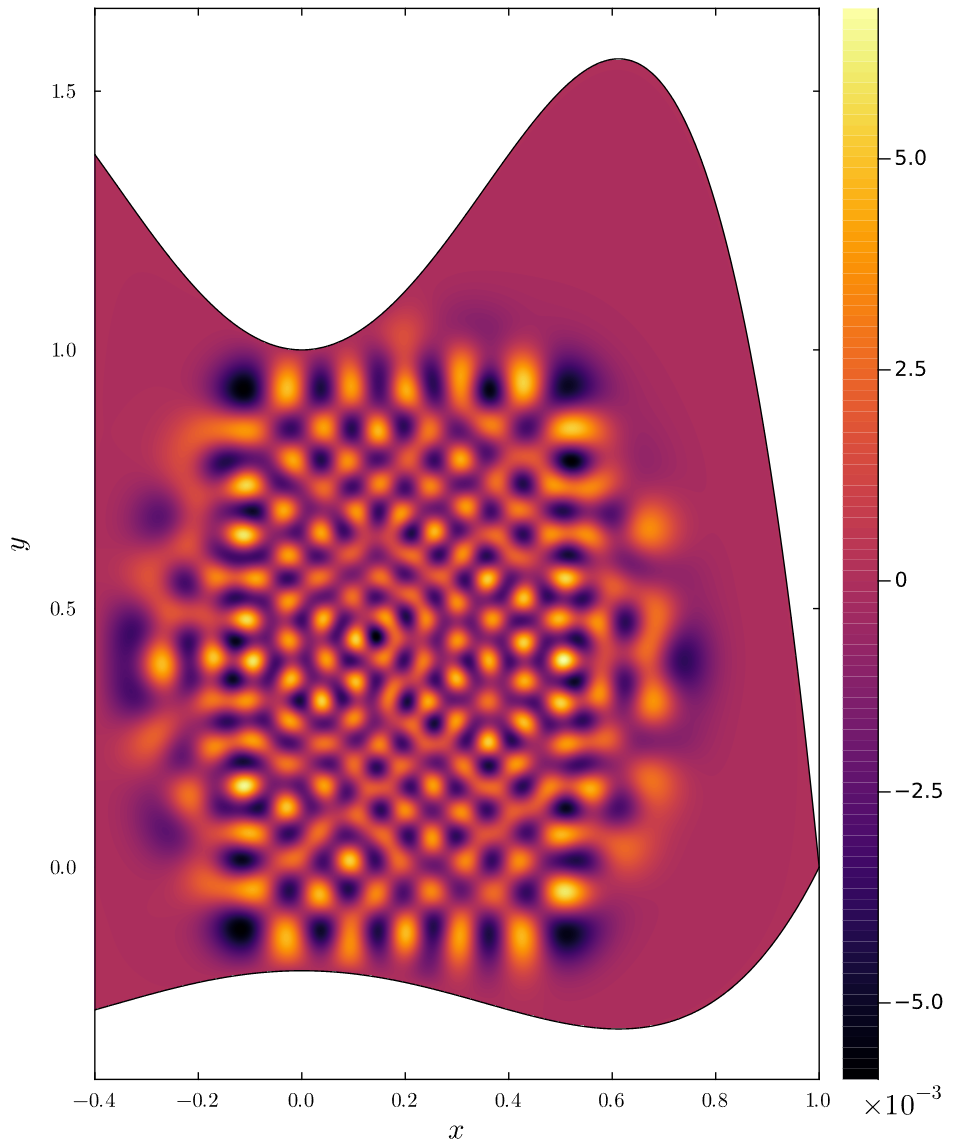}
    \includegraphics[width=0.645\linewidth]{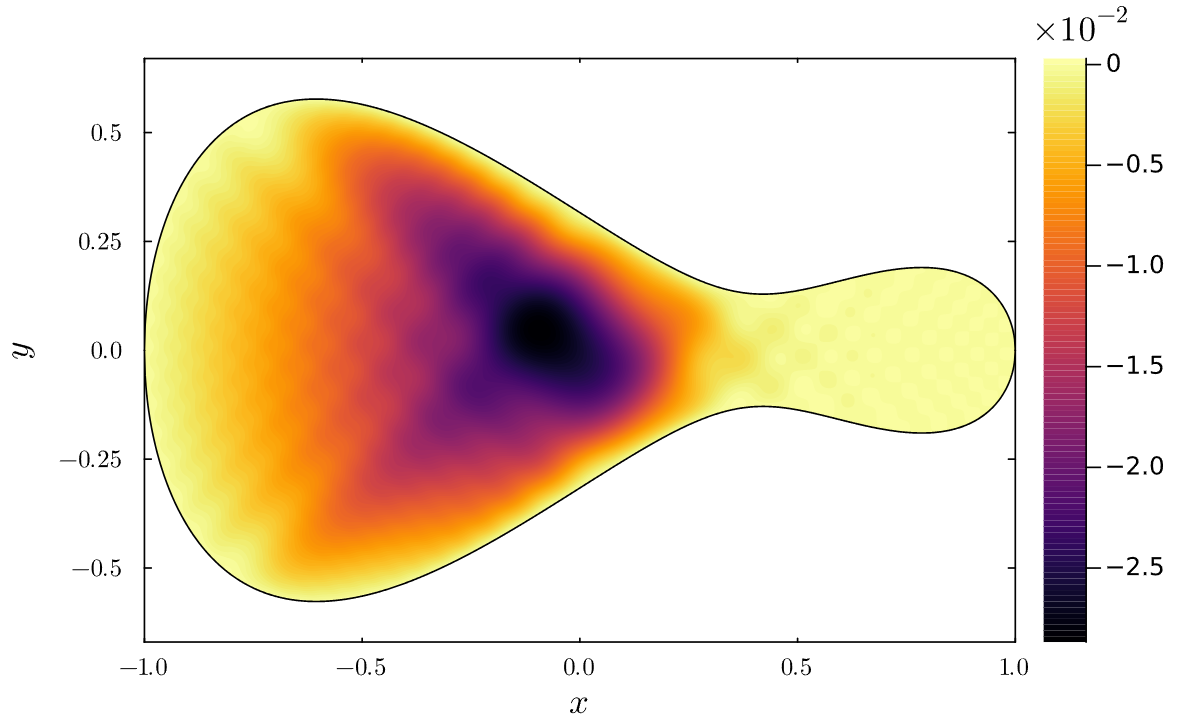}
	\caption{Numerical solutions to a variable-coefficient Helmholtz equation (left) and the Poisson equation (right) computed via the sparse spectral method using generalised Koornwinder polynomials.  These examples are discussed in more detail in Sections~\ref{Helmholtz_Example} and \ref{Poisson_Example_smoothdomain}, respectively.  The generalised Koornwinder domain on the left is defined by (\ref{eq:Omega}) with   \( (\alpha, \beta, \gamma, \delta) = (-0.4, 1, -0.2, 1) \) and \(\rho(x) = 1+3x^2-4x^4\) and the one on the right has \( (\alpha, \beta, \gamma, \delta) = (-1, 1, -1, 1) \) and \(\rho(x) = \sqrt{(1-x^2)(0.1 - 0.4x + 0.5x^2)} \). }
	\label{fig:approx_solutions_Poisson_Helmholtz}
\end{figure}

\begin{remark}
The properties of Koornwinder polynomials enable the natural extension of sparse spectral methods on the above Koornwinder domains to their corresponding three-dimensional geometries, for example, 
from triangles \cite{TriangleOPs} to tetrahedra \cite{TetrahedraOPs},
  from disc slices \cite{SnowballOlver} to spherical caps \cite{SphericalCapOPs}, and from 2D annuli \cite{Papadopoulos1} to their stretched 3D counterparts \cite{gyroscopic}.  The 3D analogues of generalised Koornwinder polynomials are left for future work.
 \end{remark}

 \section{Generalised Koornwinder polynomials and generalised Koornwinder domains}\label{sect:generalisedKoornwinder}

 \begin{defin}
Generalised Koornwinder polynomials are defined by (\ref{KoornwinderOPs})  but without restrictions on the  degree of $\rho$ (in \textup{case (i)}) or $\rho^2$ (in \textup{case (ii)}).  Hence, we let $\rho$ be positive on $(\alpha, \beta)$, where either
\begin{itemize}
		\item[ ]\textup{case (i)}: $\rho$ is a polynomial with $\deg \rho  =: d_0$, or
		\item[] \textup{case (ii)}: $\rho$ is the square root of a polynomial with $\deg \rho^2 =: d_1$, $\delta=-\gamma>0$ and $w_2$ is an even function.
	\end{itemize} 
The associated region $\Omega$, defined in (\ref{eq:Omega}), will be referred to as a generalised Koornwinder domain.
\end{defin}

\begin{theorem}\label{thm:genKoornwinder}
  Generalised Koornwinder polynomials have the following properties:
\begin{itemize}
\item[\textup{(i)}] They are mutually orthogonal, i.e., they satisfy (\ref{eq:Koornip}).
    \item[\textup{(ii)}] In \textup{case (i)}, there exist constants $c_0, \ldots, c_k$ such that 
\begin{equation}
H_{n,k} = c_0 x^{n-k}\rho^k + c_1 x^{n-k}\rho^{k-1}y + \cdots + c_k x^{n-k} y^k +\mathcal{O}\left(x^{\ell_1}\rho^{\ell_2}y^{\ell_3}  \right), \qquad \ell_1 + \ell_2 + \ell_3 < n,  \label{eq:case1koorntomono}
\end{equation}
where $0 \leq \ell_3 \leq k-1$,
and there are constants $c_{m,\ell}$ such that
\begin{equation}
x^{n-k}y^k = 
 \sum_{\ell = 0}^{k} \;\sum_{m=\ell}^{n + (k-\ell)(d_0 - 1)}   c_{m,\ell} H_{m,\ell}, \qquad 0 \leq k \leq n.  \label{eq:case1monotokoorn}
\end{equation}
For \textup{case (ii)}, if $k$ is even, there are constants $c_0, c_2, \ldots, c_k$ such that
\begin{equation}
H_{n,k} = c_0 x^{n-k}\rho^k + c_2 x^{n-k}\rho^{k-2}y^2 + \cdots + c_k x^{n-k} y^k +\mathcal{O}\left(x^{\ell_1}\rho^{\ell_2}y^{\ell_3}  \right), \qquad \ell_1 + \ell_2 + \ell_3 < n,  \label{eq:case2koorntomono}
\end{equation}
where $\ell_3$ is even and $0 \leq \ell_3 \leq k-2$;
if $k$ is odd, there are constants $c_1, c_3, \ldots, c_k$ such that
\begin{equation}
H_{n,k} = c_1 x^{n-k}\rho^{k-1}y + c_3 x^{n-k}\rho^{k-3}y^3 + \cdots + c_k x^{n-k} y^k +\mathcal{O}\left(x^{\ell_1}\rho^{\ell_2}y^{\ell_3}  \right), \qquad \ell_1 + \ell_2 + \ell_3 < n,  \label{eq:case2koorntomonokodd}
\end{equation}
where $\ell_3$ is odd and $1 \leq \ell_3 \leq k-2$.  In \textup{case (ii)}, there are constants $c_{m,\ell}$ such that
\begin{equation}
x^{n-k}y^k = \sum_{\substack{\ell = 0 \\ \mathrm{mod}(\ell,2)\, =\, \mathrm{mod}(k,2)}}^{k}  \sum_{m = \ell}^{n + (k-\ell)(d_1 - 2)/2}\:  c_{m,\ell} H_{m,\ell}, \qquad 0 \leq k \leq n. \label{eq:case2monotokoorn}
\end{equation}
\end{itemize}
\end{theorem}

\begin{proof}
The orthogonality of the generalised Koornwinder polynomials (property (i))  can be demonstrated as for the standard Koornwinder polynomials in Theorem~\ref{thm:koornwinderbasis}, see \cite[Proposition 2.6.1]{DunklXu}.

To prove property (ii), we observe that equations (\ref{eq:case1koorntomono}), (\ref{eq:case2koorntomono}) and (\ref{eq:case2koorntomonokodd}) follow immediately from the definition of the generalised Koornwinder polynomials  and the fact that the OPs $\lbrace q_k \rbrace$ have the same parity as $k$ in case (ii).  

To prove (\ref{eq:case1monotokoorn}) and (\ref{eq:case2monotokoorn}), we consider the expansion
\begin{equation}
\sum_{m \geq 0} \sum_{\ell=0}^{m} c_{m,\ell} H_{m,\ell}, \qquad c_{m,\ell} = \frac{\langle x^{n-k}y^{k}, H_{m,\ell}\rangle_{W}}{\langle H_{m,\ell}, H_{m,\ell}\rangle_{W}}. \label{eq:monoexpansion}
\end{equation}
 We have that
\begin{eqnarray}
\left\langle x^{n-k}y^k, H_{m,\ell}   \right\rangle_W &=& \int_{\gamma \rho}^{\delta\rho}\int_{\alpha}^{\beta} x^{n-k}y^k p_{m-\ell,\ell}(x) \rho^{\ell} q_{\ell}(y/\rho) w_1(x) w_2(y/\rho)\, \mathrm{d}x\mathrm{d}y  \nonumber \\
 &=& \int_{\gamma }^{\delta} t^k q_{\ell}(t) w_2(t)\, \mathrm{d}t \int_{\alpha}^{\beta} \rho^{k-\ell} x^{n-k} p_{m-\ell,\ell}(x) \rho^{2\ell+1}  w_1(x)\, \mathrm{d}x. \label{twointegrals}
\end{eqnarray}
By orthogonality, the first integral in~(\ref{twointegrals}) is zero (and therefore also $c_{m,\ell}$ in (\ref{eq:monoexpansion})) if $\ell > k$.  In case (ii), we have that $\gamma = -\delta$, $w_2$ is even  and $q_{\ell}$ has the same parity as $\ell$, hence the first integral vanishes if $\ell$ and $k$ have opposite parity.  Hence,  $0 \leq \ell \leq k$ in (\ref{eq:monoexpansion}) and (\ref{twointegrals}) and additionally in case (ii), $\ell$ and $k$ have the same parity.

The second integral in (\ref{twointegrals}) vanishes due to orthogonality if $m-\ell > \deg(\rho^{k-\ell}x^{n-k})$. In case (i), this implies that the second integral is zero if $m > n + (k-\ell)(d_0-1)$ and in case (ii), if $m >  n + (k-\ell)(d_1-2)/2$.  Hence, in (\ref{eq:monoexpansion}) we have $m \leq n + (k-\ell)(d_0-1)$ in case (i) and in case (ii), $m \leq  n + (k-\ell)(d_1-2)/2$.  Now it is straightforward to show, using the orthogonality of the generalised Koornwinder polynomials, that the expansion (\ref{eq:monoexpansion}) is equal to the monomial $x^{n-k}y^k$ and hence (\ref{eq:case1monotokoorn}) and (\ref{eq:case2monotokoorn}) have been proved.

               \end{proof}

\begin{corollary}\label{cor:Koornspandeg}
An immediate  consequence of (\ref{eq:case1monotokoorn}) and (\ref{eq:case2monotokoorn}) is that the generalised Koornwinder polynomials span the space of bivariate polynomials for any $d_0, d_1 \geq 0$.  Furthermore, it follows from (\ref{eq:case1koorntomono}) that in \textup{case (i)},
\begin{equation}
		\deg H_{n,k} = \begin{cases}
		n & \text{if } d_0 \leq 1 \\[0.5em]
		n + k(d_0-1) & \text{if } d_0 \geq 2
		\end{cases}, \label{eq:Koorndeg1}
	\end{equation}
and from (\ref{eq:case2koorntomono}) and (\ref{eq:case2koorntomonokodd}) we have in \textup{case (ii)} that
\begin{equation}
		\deg H_{n,k} = \begin{cases}
		n & \text{if } d_1 \leq 2 \\[0.5em]
		n +  \left\lfloor \frac{k}{2} \right\rfloor \left(d_1 -2 \right) & \text{if } d_1 \geq 3  
		\end{cases}.  \label{eq:Koorndeg2}
	\end{equation}
\end{corollary}

In order to simplify notation and unify the treatment of cases (i) and (ii), we introduce the indicator parameter $\eta$, which is defined to be 
\begin{equation*}
\eta = \begin{cases}
0 & \text{in case (i)}, \\
1 & \text{in case (ii)}
\end{cases}.
\end{equation*}
Hence, we have that $d_{\eta} = \deg \rho^{1 + \eta}$. 

Note from (\ref{eq:Koorndeg1}) and (\ref{eq:Koorndeg2}) that if $d_{\eta} \leq 1 + \eta$ (precisely the degree restrictions that apply in Theorem~\ref{thm:koornwinderbasis}), then $\deg H_{n,k} = n$ and in this case the Koornwinder polynomials are bivariate OPs.  If $d_{\eta} \geq 2 + \eta$, then we have from (\ref{eq:Koorndeg1}) and (\ref{eq:Koorndeg2}) that $\deg H_{n,k} > n$ for $1+\eta \leq k \leq n$ and therefore even though the generalised Koornwinder polynomials are orthogonal, they are not bivariate OPs because they are not degree-graded.
    
The methods in this paper are applicable for any degrees $d_0, d_1$.  However, the cases mentioned in Section~\ref{sect:intro} for which Koornwinder polynomials are bivariate OPs, for which $d_{\eta} \leq 1 + \eta$, have already been used in the literature to devise sparse spectral methods, using explicit formul\ae\ for recurrence relationships, see \cite{TriangleOPs,Acta}.  Thus, henceforth we assume $d_{\eta} \geq 2 + \eta$.

A special class of generalised Koornwinder domains is smooth domains (domains without corners) which is only possible  
in case (ii) if $\rho$ takes the form $\rho(x) = \sqrt{(x-\alpha)(\beta-x)\phi(x)}$, where $\phi$ is a polynomial whose roots are off the interval $[\alpha, \beta]$ (an example of which is the right frame of Figure~\ref{fig:approx_solutions_Poisson_Helmholtz}).  Smooth domains will be discussed separately in Appendix~\ref{Generalised_Koornwinder_polynomials_smooth_domains}.

\subsection{A four-parameter family of generalised Koornwinder polynomials} \label{Computations_with_univariate_semiclassical_OPs}

As will be shown in Section~\ref{Operator_matrices _Koornwinder_polynomials}, in order to construct a sparse spectral method on generalised Koornwinder domains $\Omega$, we need a four-parameter family of generalised Koornwinder polynomials that is orthogonal with respect to the weight
\begin{equation}
	W^{(a,b,c,d)}(x,y):=  (\beta-x)^a (x-\alpha)^b (y - \gamma\rho(x))^c (\delta \rho(x) - y)^d, \qquad (x,y) \in \Omega.    \label{eq:Wdef}
\end{equation}

First, we introduce notation for two families of univariate OPs:
let $w_R^{(a,b,c)}(x)$ and $w_P^{(a,b)}(x)$ be  weight functions supported on $[\alpha, \beta]$ and $ [\gamma, \delta]$, respectively, given by 
\begin{align}
	w_R^{(a,b,c)}(x):=(\beta-x)^a (x-\alpha)^b \rho(x)^c, \quad w_P^{(a,b)}(x):=(\delta - x)^{a} (x-\gamma)^{b}, \quad a, b, c > -1,  \label{eq:wRwPdef}
\end{align}
and define the associated inner products by
 \begin{align}
		\langle p, q \rangle_{w_{R}^{(a,b,c)}} := \int_{\alpha}^{\beta} p(x) q(x) w_{R}^{(a,b,c)}(x) \, \mathrm{d}x, \quad \langle p, q \rangle_{w_{P}^{(a,b)}} := \int_{\gamma}^{\delta} p(y) q(y) w_{P}^{(a,b)}(y) \, \mathrm{d}y. \label{inner_prod_PR}
\end{align}
We denote by \( \{R_n^{(a,b,c)}\} \) the semiclassical orthonormal polynomials
  on \( [\alpha, \beta] \), with respect to the first inner product in~(\ref{inner_prod_PR})\footnote{A well-known family of semiclassical OPs is the Jacobi semiclassical OPs (introduced in~\cite{magnus1995painleve} and extensively used in~\cite{Papadopoulos1}), with weight $x^b(1-x)^a(t-x)^c$ on $[0, 1]$, where $t > 1$, hence the semiclassical Jacobi polynomials correspond to $R_n^{(a,b,c)}$ with $[\alpha, \beta] = [0, 1]$ and $\rho = t - x$. }. 
Likewise, \( \{\widetilde{P}_n^{(a,b)}\} \) denotes the shifted-and-normalised Jacobi polynomials on \( [\gamma, \delta] \), with respect to the second inner product in~(\ref{inner_prod_PR}) (for the definition of the standard Jacobi polynomials, see \cite[\href{https://dlmf.nist.gov/18.3}{section 18.3}]{NIST}).

We introduce the following  notation  
 for a row-vector containing univariate OPs:
\begin{align*}
	\mathbf{R}^{(a,b,c)}&=\mathbf{R}^{(a,b,c)}(x)=\left(\begin{array}{c}R_{0}^{(a,b,c)}(x)\:\Big \vert\:  R_{1}^{(a,b,c)}(x)\:\Big \vert\: R_{2}^{(a,b,c)}(x) \:\Big \vert\: \cdots\end{array}\right),\\[1em]
	\mathbf{\widetilde{P}}^{(a,b)}&=\mathbf{\widetilde{P}}^{(a,b)}(x)=\left(\begin{array}{c}\widetilde{P}_{0}^{(a,b)}(x)\:\Big \vert\: \widetilde{P}_{1}^{(a,b)}(x) \:\Big \vert\:  \widetilde{P}_{2}^{(a,b)}(x)\:\Big \vert\: \cdots\end{array}\right).
\end{align*}
These can also be interpreted as quasi-matrices, generalisations of matrices where the first argument is continuous, and provide a convenient framework for linking recurrence relations to structured numerical linear algebra \cite{Acta}. 

Noting that 
$$W^{(a,b,c,d)}(x,y) = w^{(a,b,c+d)}_R(x) \,  w^{(d,c)}_P\left(\frac{y}{\rho(x)}\right), $$
hence $w^{(a,b,c+d)}_R$ corresponds to $w_1$ in Theorem~\ref{thm:koornwinderbasis} and $w^{(d,c)}_P$ corresponds to $w_2$, it follows from  Theorems~\ref{thm:koornwinderbasis} and~\ref{thm:genKoornwinder} that 
  	\begin{equation}\label{KoornwinderFomula}
		H_{n,k}^{(a,b,c,d)}(x,y) := R_{n-k}^{(a, b, c+d+2k+1)}(x) \, \rho(x)^k \, \widetilde{P}_k^{(d,c)}\left(\frac{y}{\rho(x)}\right), \quad 0 \leq k \leq n, \quad (x,y) \in \Omega, 
	\end{equation}
constitutes a four-parameter family of generalised Koornwinder polynomials that are mutually orthogonal (and orthonormal) with respect to the inner product
\begin{equation}
	\langle f, \, g \rangle_{W^{(a,b,c,d)}} := \iint_\Omega f(x,y) \, g(x,y) \, W^{(a,b,c,d)}(x,y) \, \mathrm{d}y \, \mathrm{d}x.  \label{eq:inproddef}
\end{equation}
In case (ii), to ensure $w_P^{(d,c)}(x)$ is an even function, we set $d = c$ and denote the weight as $w_P^{(c)}(x) := w_P^{(c,c)}(x)=(\delta^2 - x^2)^{c}$.

The block-quasi-matrix of the Koornwinder polynomials on $\Omega$ is given by
\begin{equation*}
	\mathbf{H}^{(a,b,c,d)}=\mathbf{H}^{(a,b,c,d)}(x,y)=\left(\begin{array}{c}\mathbb{H}_0^{(a,b,c,d)}(x, y)\: \Big \vert \: \mathbb{H}_1^{(a,b,c,d)}(x, y) \: \Big \vert \: \mathbb{H}_2^{(a,b,c,d)}(x,y) \: \Big \vert \: \cdots\end{array}\right)
\end{equation*}
where each block $\mathbb{H}_n^{(a,b,c,d)}$ is given by 
\begin{equation*}
	\mathbb{H}_n^{(a,b,c,d)}(x,y)=\left(H_{n,0}^{(a,b,c,d)}(x,y)\: \Big \vert \: H_{n,1}^{(a,b,c,d)}(x,y) \: \Big \vert \: \cdots \: \Big \vert \: H_{n,n}^{(a,b,c,d)}(x,y)\right).
\end{equation*}

\section{Operator matrices of classical OPs}\label{sect:classopmats}

The orthonormalised Jacobi polynomials $\{\widetilde{P}_n^{(d, c)}\}$ on $(\gamma, \delta)$ can be obtained via an affine transformation from the standard Jacobi polynomials on $(-1,1)$. The corresponding raising relation, three-term recurrence and differentiation formula follow from those of the standard Jacobi polynomials given in \cite[\href{https://dlmf.nist.gov/18.9}{section 18.9}]{NIST}:
\begin{align}
	&\widetilde{P}^{(d,c)}_{n}\left(y\right) = r^{(d+1,c+1)}_{n}\,\widetilde{P}^{(d+1,c+1)}_{n}\left(y\right) - t^{(d+1,c+1)}_{n-1}\,\widetilde{P}^{(d+1,c+1)}_{n-1}\left(y\right) -s^{(d+1,c+1)}_{n-2}\,\widetilde{P}^{(d+1,c+1)}_{n-2}\left(y\right), \label{ClassicalRaising}\\
	&y \widetilde{P}_n^{(d, c)}(y) = \delta_n^{(d,c)} \widetilde{P}_{n+1}^{(d, c)}(y) + \gamma_n^{(d,c)} \widetilde{P}_n^{(d, c)}(y) + \delta_{n-1}^{(d,c)} \widetilde{P}_{n-1}^{(d, c)}(y),  \label{eq:jacthreeterm}\\
	&\frac{\mathrm{d}}{\mathrm{d}y} \widetilde{P}_n^{(d, c)}(y) = d^{(d+1,c+1)}_{n-1}\, \widetilde{P}^{(d+1,c+1)}_{n-1}\left(y\right). \label{ClassicalDerivative} 
\end{align}
In case (ii), we have that $c=d$ and $\gamma_n^{(d,c)} = 0, n = 0,1,2,\ldots$.  Equations (\ref{ClassicalRaising})--(\ref{ClassicalDerivative}) can  be expressed in quasi-matrix notation as follows,
\begin{equation}
	\mathbf{\widetilde{P}}^{(d,c)} = \mathbf{\widetilde{P}}^{(d+1,c+1)} \, T(w_P^{(d,c)}), \quad y \,\mathbf{\widetilde{P}}^{(d,c)} = \mathbf{\widetilde{P}}^{(d,c)} \, J(w_P^{(d,c)}), \quad  \frac{\mathrm{d}}{\mathrm{d}y} \mathbf{\widetilde{P}}^{(d,c)} = \mathbf{\widetilde{P}}^{(d+1,c+1)} \, D(w_P^{(d,c)}),\label{eq:jacmats}
\end{equation}
where $T(w_P^{(d,c)})$, $J(w_P^{(d,c)})$ and $D(w_P^{(d,c)})$ are, respectively, the raising, Jacobi and differentiation matrices of $\{\widetilde{P}_n^{(d, c)}\}$, which are matrix representations of raising, multiplication and differentiation operators acting on $\mathbf{\widetilde{P}}^{(d,c)}$ and hence we shall refer to these matrices collectively as operator matrices.

In this paper, the element in the $(i+1)$-th row and $(j+1)$-th column of a matrix $A$ is denoted by $A_{[i,j]}$.
 For simplicity, we allow $i$ and $j$ to take negative values, in which case $A_{[i,j]}$ is defined to be zero.
We say that $A$ is banded and has bandwidths $(\sigma , \mu)$ if $A_{[i,j]}=0$ for $i-j>\sigma $ and $j-i>\mu$, and the same applies to block matrices. Hence, the tridiagonal Jacobi matrix $J(w_P^{(d,c)})$ has bandwidths $(1,1)$, $T(w_P^{(d,c)})$ has bandwidths $(0,2)$, and $D(w_P^{(d,c)})$ has bandwidths $(-1,1)$.  If the matrices $A$ and $B$ have bandwidths $(\sigma_1, \mu_1)$ and $(\sigma_2, \mu_2)$, respectively, then $AB$ has bandwidths at most $(\sigma_1 + \sigma_2, \mu_1 + \mu_2)$.

Note that all these operator matrices can be expressed as the inner product of quasi-matrices. For example, we write
\begin{equation}
	T(w_P^{(d,c)}) = \left\langle  \mathbf{\widetilde{P}}^{(d+1,c+1)}, \mathbf{\widetilde{P}}^{(d,c)}  \right\rangle_{w_P^{(d+1,c+1)}} = \int_{\gamma}^{\delta} {\mathbf{\widetilde{P}}^{(d+1,c+1)}(x)}^\top\, \mathbf{\widetilde{P}}^{(d,c)}(x)\, w_P^{(d+1,c+1)}(x)\, \mathrm{d}x,  \label{eq:jacraise}
\end{equation}
where ${\mathbf{\widetilde{P}}^{(d+1,c+1)}(x)}^\top\, \mathbf{\widetilde{P}}^{(d,c)}(x)$ is interpreted as an outer product of two (infinite) row vectors whose entries
are functions and integration is performed entry-wise, hence 
\begin{equation*}
\left(T(w_P^{(d,c)})\right)_{[i,j]} = \langle \widetilde{P}^{(d+1,c+1)}_i,  \widetilde{P}^{(d+1,c+1)}_j \rangle_{w_P^{(d+1,c+1)}}.
\end{equation*}
  The identity (\ref{eq:jacraise}) follows from taking the inner product of both sides of the first equation in (\ref{eq:jacmats}) with $\mathbf{\widetilde{P}}^{(d+1,c+1)}$ and using the fact that since the OPs $\{ \widetilde{P}_n^{(d+1,c+1)} \}$ are orthonormal, $\langle  \mathbf{\widetilde{P}}^{(d+1,c+1)}, \mathbf{\widetilde{P}}^{(d+1,c+1)}  \rangle_{w_P^{(d+1,c+1)}} = I$, where $I$ is the (infinite) identity matrix.

More generally, for any linear operator $\mathcal{L}$ mapping the columns of $\mathbf{\widetilde{P}}^{(d,c)}$ to $L_{w_P^{(\delta,\gamma)}}^2$, its matrix representation $M$ in the orthonormal basis $\mathbf{\widetilde{P}}^{(\delta,\gamma)}$ can be deduced by expanding  ${\cal L} \widetilde{P}{(d,c)}$ in the basis  corresponding to $\mathbf{\widetilde{P}}^{(d,c)}$, i.e.
\begin{equation}
\mathcal{L} \mathbf{\widetilde{P}}^{(d,c)} = \mathbf{\widetilde{P}}^{(\delta,\gamma)} M  \qquad \Rightarrow \qquad M = \left\langle  \mathbf{\widetilde{P}}^{(\delta,\gamma)}, \mathcal{L}\mathbf{\widetilde{P}}^{(d,c)}  \right\rangle_{w_P^{(\delta,\gamma)}}.  \label{eq:linopmat}
\end{equation}
Note any operator that maps polynomials to polynomials (e.g., differentiation) will satisfy this criteria.

\section{Operator matrices of semiclassical OPs via connection matrices}\label{sect:semiclassconnmats}

To compute the operator matrices associated with the semiclassical OPs $\mathbf{ R}^{(a,b,c+d+2k+1)}$ that are required to construct the generalised Koornwinder polynomials (\ref{KoornwinderFomula}), we first need to compute the connection matrix $\mathcal{C}^{(a,b,c+d+2k+1)}$ that relates the quasi-matrix of shifted-and-normalised Jacobi polynomials with weight $w^{(a,b)}_S(x)$ on $[\alpha, \beta]$, namely $\mathbf{S}^{(a,b)}$, to $\mathbf{ R}^{(a,b,c+d+2k+1)}$, i.e.,  
\begin{align}\label{ConnMatSR}
	\mathbf{S}^{(a,b)} = \mathbf{R}^{(a,b,c+d+2k+1)} \,\mathcal{C}^{(a,b,c+d+2k+1)}.
\end{align}
An equivalent expression of (\ref{ConnMatSR}) is
\begin{equation*}
{S}_n^{(a,b)} = \sum_{k = 0}^{n}  \mathcal{C}^{(a,b,c+d+2k+1)}_{[k,n]}  {R}_k^{(a,b,c+d+2k+1)}
\end{equation*}
since (\ref{ConnMatSR}) is a change of polynomial basis.  Hence, in general, a connection matrix that relates different OP families is upper triangular, in which case it has bandwidths $(0,\infty)$, however we shall find that the connection matrix in (\ref{ConnMatSR}) is banded.

We note that the weight $w^{(a,b,c+d+2k+1)}_R$ can be expressed as the product of a classical  weight $w^{(a,b)}_S$ and a measure modification function $\rho^{c+d+2k+1}$,
\begin{equation}
w^{(a,b,c+d+2k+1)}_R(x) =  
(\beta-x)^a (x-\alpha)^b \rho^{c+d+2k+1} = w^{(a,b)}_S(x) \rho^{c+d+2k+1}.  \label{eq:factorweight}
\end{equation}
 Without loss of generality, we assume that $\rho$ (in case (i)) or $\rho^2$ (in case (ii)) has no roots at the endpoints $\alpha$ and $\beta$, i.e., all the roots are in $\mathbb{C}\setminus [\alpha, \beta]$.  Otherwise, the root(s) at $\alpha$ or $\beta$ can be incorporated into the classical weight. 
              Since 
\begin{equation}
\mathcal{C}^{(a,b,c+d+2k+1)}_{[i,j]} = \left\langle   R_i^{(a,b,c+d+2k+1)}, S_j^{(a,b)} \right\rangle_{w_R^{(a,b,c+d+2k+1)}} =  \left\langle   \rho^{c+d+2k+1}R_i^{(a,b,c+d+2k+1)}, S_j^{(a,b)} \right\rangle_{w_S^{(a,b)}},  \label{eq:semiclassconmatentries}
\end{equation}
it follows from the definition of orthogonality that if the modification function $\rho^{c+d+2k+1}$ is a polynomial, then the connection matrix $\mathcal{C}^{(a,b,c+d+2k+1)}$ has bandwidths $(0,\deg \rho^{c+d+2k+1})$.  We shall let $a, b, c, d$ be non-negative integers in our numerical experiments, hence in case (i), the modification function is a polynomial whereas in case (ii), with $c = d$, the modification function $\rho^{2c + 2k + 1}$  is not a polynomial.  
 However, since $\rho^2$ is a positive polynomial on $[\alpha, \beta]$ in case (ii), $\rho$ is analytic on $[\alpha, \beta]$ and thus we shall use the fact that it has a polynomial approximation $q(x)$ in the OP basis $\lbrace S^{(a,b)}_n \rbrace$ with say $\deg q = m$ on $[\alpha, \beta]$ such that $\| \rho - q \|_{\infty} < \epsilon$, where we let $\epsilon$ be the value of the machine precision\footnote{For the numerical experiments in this paper, we shall use machine precision-accurate polynomial approximations to $\sqrt{0.1 - 0.4x +0.5x^2} $  on $[-1, 1]$ and $\sqrt{(0.5 + x)(1 + 6x^2 - 20x^4 + 15x^6)}$ on $ [0, 1]$ for the domains in Figures~\ref{fig:approx_solutions_Poisson_Helmholtz} (right)  and \ref{fig:Screened_Poisson_fish_solution}.  These approximations have degrees $138$ and $93$, respectively. 
    
 }.  Therefore, in case (i) the connection matrix $\mathcal{C}^{(a,b,c+d+2k+1)}$ has bandwidths $(0,d_0(c+d+2k+1))$ and in case (ii), the connection matrix will be approximated (up to machine precision) by a matrix with bandwidths $(0,d_1(c+k)+m)$ (since $\rho^{2c+2k + 1} \approx (\rho^2)^{c+k}q$ and $\deg[ (\rho^2)^{c+k}q] = d_1(c + k) + m$).

In the next two subsections we describe how we first compute the connection matrix (\ref{ConnMatSR}) for $k = 0$ using a moment-based algorithm and then recursively compute the connection matrices $\mathcal{C}^{(a,b,c+d+2\ell+1)}$, for $\ell = 1, 2, \ldots, k$ using a matrix factorisation method.

 \subsection{Connection matrix with $k = 0$}  \label{sect:kzero}

 The measure modification function for $k = 0$, $\rho^{c+d+1}$, has the expansion 
\begin{equation}
\rho^{c+d+1} = \sum_{n = 0}^{\infty} \mu_n^{(a,b,c,d)} S_n^{(a,b)}(x) = \mathbf{S}^{(a,b)}\bm{\mu}_0^{(a,b,c,d)},  \label{eq:modk0exp}
\end{equation}
where the entries of the vector $\bm{\mu}_0^{(a,b,c,d)} \in \mathbb{R}^{\infty \times 1}$
   are the expansion coefficients $\mu_n^{(a,b,c,d)}$ and thus
\begin{align}
	\bm{\mu}_0^{(a,b,c,d)} =  \left\langle \mathbf{S}^{(a,b)}, \rho^{c+d+1}  \right\rangle_{w_{S}^{(a,b)}} = \int_\alpha^\beta {\mathbf{S}^{(a,b)}}^\top  \rho^{c+d+1}  w^{(a,b)}_S {\rm\,d}x. \label{eq:mu0vec}
\end{align}
In case (i), $\bm{\mu}_0^{(a,b,c,d)}$ has at most $d_0(c + d + 1) + 1$ nonzero entries and in case (ii) $\bm{\mu}_0^{(a,b,c,d)}$ is approximated up to machine precision by a vector with at most $d_1 c + m + 1$ nonzero entries, where $m$ is the degree of the polynomial approximation to $\rho$.  

Using the moment-based Algorithm 2 in \cite{FastGram}, the $N \times N$ principal finite section of the connection matrix $\mathcal{C}^{(a,b,c+d+1)}$ can be computing using the first $2N-1$ entries of the vector $\bm{\mu}_0^{(a,b,c,d)}$ with linear complexity in $N$: $\mathcal{O}(d_0 N)$ in case (i) and $\mathcal{O}((d_1 + m) N)$ in case (ii).  

We shall use Gauss--Jacobi quadrature to compute the first $d_{\eta} + 1$ terms of (\ref{eq:mu0vec}).  The next $2N - 2 - d_{\eta}$ terms will be computed using the recurrence relation (\ref{eq:murecc}), which is derived from the following differential equation satisfied by the measure modification function.

\begin{lemma}
Suppose that no roots of $\rho$ (in case (i)) or $\rho^2$  (in case (ii)) lie in $[\alpha, \beta]$ and that $d_{\eta} \geq 1$. Then the weight $w_R^{(a,b,c+d+1)}$ is a semiclassical weight that satisfies the Pearson equation
\begin{align}
		\frac{\mathrm{d}}{\mathrm{d}x}\left(\sigma_{R} \, w^{(a,b,c+d+1)}_R \right) = \tau_R \, w^{(a,b,c+d+1)}_R, 
		 \label{PearsonEq}
	\end{align}
where 	
\begin{equation}
\sigma_{R} := \rho^{1+\eta} \sigma_S, \qquad \tau_R := (c+d+2+\eta) \rho^\eta \rho' \sigma_S + \rho^{1+\eta} \tau_S,  \label{eq:sigtauR}		
\end{equation}
and $\sigma_S = (x-\alpha)(\beta-x)$ and $\tau_S = (b+1)(\beta-x) - (a+1)(x-\alpha)$  are the polynomials associated with the classical Pearson equation: $(\sigma_{S} \, w_S^{(a,b)})' = \tau_S \, w_S^{(a,b)}$.  Furthermore, it follows from the semiclassical Pearson equation that the measure modification function for $k = 0$, $\rho^{c+d+1}$, satisfies the differential equation
\begin{equation}
-(c+d+1)f_{\eta,\rho} \,\rho^{c+d+1} + \rho^{1+\eta} \sigma_S\,\frac{\mathrm{d}}{\mathrm{d}x}\rho^{c+d+1} = 0, \qquad f_{\eta,\rho}= f_{\eta,\rho}(x) := {\rho^{\eta}} \rho' \sigma_S. \label{eq:Pearsonsimp}
\end{equation}
\end{lemma}
\begin{proof}
The identities (\ref{PearsonEq}) and (\ref{eq:Pearsonsimp}) can be verified by  substituting (\ref{eq:factorweight}) with $k=0$ and using the classical Pearson equation. The weight $w_R^{(a,b,c+d+1)}$ is semiclassical because $\deg \sigma_{R} > 2$ and $\deg \tau_{R} > 1$ (see the sentence below (\ref{Pearson})).  
\end{proof}

Just as in (\ref{eq:jacmats}), for the quasi-matrix $\mathbf{S}^{(a,b)}$ of shifted-and-normalised Jacobi polynomials with respect to $w^{(a,b)}_S(x)$ there exist  banded operator matrices such that
\begin{align}
	\mathbf{S}^{(a,b)} = \mathbf{S}^{(a+1,b+1)} \, T(w_S^{(a,b)}), \quad x \,\mathbf{S}^{(a,b)} = \mathbf{S}^{(a,b)} \, J(w_S^{(a,b)}), \quad \frac{\mathrm{d}}{\mathrm{d}x} \mathbf{S}^{(a,b)} = \mathbf{S}^{(a+1,b+1)} \, D(w_S^{(a,b)}).  \label{eq:Smats}
\end{align}
It follows from the second equation in (\ref{eq:Smats}) and (\ref{eq:linopmat}) that for a polynomial $p(x)$, 
\begin{equation}
p(x)\mathbf{S}^{(a,b)} = \mathbf{S}^{(a,b)} \, p\left(J(w_S^{(a,b)})\right) \quad \Rightarrow \quad p\left(J(w_S^{(a,b)})\right) = \left\langle \mathbf{S}^{(a,b)}, p(x) \mathbf{S}^{(a,b)}   \right\rangle_{w_{S}^{(a,b)}},  \label{eq:polymult}
\end{equation}
and since the Jacobi matrix $J(w_S^{(a,b)})$ has bandwidths $(1, 1)$, we know $p\left(J(w_S^{(a,b)})\right)$ has bandwidths $(\deg p, \deg p)$.

To derive the recurrence satisfied by the entries of $\bm{\mu}_0^{(a,b,c,d)}$, we require the operator matrix $M$ defined by  
\begin{equation}
\sigma_S \frac{\mathrm{d}}{\mathrm{d}x} \mathbf{S}^{(a,b)} =  \mathbf{S}^{(a,b)} M.  \label{eq:weightdiff}
\end{equation}
It follows from (\ref{eq:linopmat}) and (\ref{eq:Smats}) that
\begin{eqnarray}
M &=& \left\langle \mathbf{S}^{(a,b)},  \sigma_S \frac{\mathrm{d}}{\mathrm{d}x} \mathbf{S}^{(a,b)} \right\rangle_{w_S^{(a,b)}} = \left\langle \mathbf{S}^{(a,b)},   \frac{\mathrm{d}}{\mathrm{d}x} \mathbf{S}^{(a,b)} \right\rangle_{w_S^{(a+1,b+1)}} \notag \\
&=& \left\langle \mathbf{S}^{(a+1,b+1)} \, T(w_S^{(a,b)}),   \mathbf{S}^{(a+1,b+1)} \, D(w_S^{(a,b)}) \right\rangle_{w_S^{(a+1,b+1)}}  \notag \\
&=& T(w_S^{(a,b)})^{\top} D(w_S^{(a,b)}).  \label{eq:Mform}
\end{eqnarray}
Since the bandwidths of $T(w_S^{(a,b)})$ and  $D(w_S^{(a,b)})$ are, respectively, $(0, 2)$ and $(-1, 1)$, $M$ has bandwidths $(2, 0)+ (-1,1) = (1, 1)$.
               
\begin{proposition} \label{Moment_Recurrence_Relation}
	The entries of the vector $\bm{\mu}_0^{(a,b,c,d)}$ defined in (\ref{eq:mu0vec}) satisfy a linear recurrence relation with at most $2 d_{\eta} +3$ terms, given by
	\begin{align}
		\left(-(c+d+1)  f_{\eta, \rho}\left(J(w_S^{(a,b)}) \right)  +  \rho^{1+\eta}\left(J(w_S^{(a,b)}) \right)\,{T(w_S^{(a,b)})}^\top \,D(w_S^{(a,b)}) \right) \, \bm{\mu}_0^{(a,b,c,d)} = \bm{0}.  \label{eq:murecc}
	\end{align}
 \end{proposition}

\begin{proof}
Substituting the measure modification function expressed in the $\mathbf{S}^{(a,b)}$ basis, (\ref{eq:modk0exp}), into the differential equation (\ref{eq:Pearsonsimp}) and using  (\ref{eq:polymult}), (\ref{eq:weightdiff}), (\ref{eq:Mform}) and the fact that $f_{\eta,\rho}$ and $\rho^{1 + \eta}$ are polynomials, (\ref{eq:murecc}) follows.  The matrix multiplying  $\bm{\mu}_0^{(a,b,c,d)}$ in (\ref{eq:murecc}) has bandwidths $(d_{\eta} + 1, d_{\eta} + 1)$, hence the entries of  $\bm{\mu}_0^{(a,b,c,d)}$ satisfy a linear recurrence relation with at most $2(d_{\eta} + 1) + 1 = 2 d_{\eta} +3$ terms.
              \end{proof}

\subsection{Connection matrices with $k > 0$}\label{sect:connectkg0}
            To compute the connection matrices (\ref{ConnMatSR}) with $k > 0$, we require the following raising matrices (which can also be viewed as connection matrices between different families of semiclassical OPs) 
\begin{equation}
\mathbf{R}^{(a,b,c+d+2k+1)} = \mathbf{R}^{(a,b,c+d+2k+3)} \,\mathcal{R}_{k+1}^{(a,b,c,d)},  \label{eq:semiclassraise}	
\end{equation}
because then it follows from (\ref{ConnMatSR}) with $k=0$ and (\ref{eq:semiclassraise}) that
\begin{eqnarray}
	\mathbf{S}^{(a,b)} &=& \mathbf{R}^{(a,b,c+d+1)} \,\mathcal{C}^{(a,b,c+d+1)}  = \mathbf{R}^{(a,b,c+d+3)} \mathcal{R}_{1}^{(a,b,c,d)}	\mathcal{C}^{(a,b,c+d+1)} =\cdots \notag \\
 &=&  \mathbf{R}^{(a,b,c+d+2k+1)}\mathcal{R}_{k}^{(a,b,c,d)}\mathcal{R}_{k-1}^{(a,b,c,d)}\cdots \mathcal{R}_{1}^{(a,b,c,d)}	\mathcal{C}^{(a,b,c+d+1)},  \label{eq:raiseconnect}
\end{eqnarray}
and thus, comparing (\ref{eq:raiseconnect}) and (\ref{ConnMatSR}), we conclude that
\begin{equation*}
\mathcal{C}^{(a,b,c+d+2k+1)} = \mathcal{R}_{k}^{(a,b,c,d)}\mathcal{R}_{k-1}^{(a,b,c,d)}\cdots \mathcal{R}_{1}^{(a,b,c,d)}	\mathcal{C}^{(a,b,c+d+1)}.
 \end{equation*}
Hence, connection matrices with $k > 0$ will be computed recursively  using the raising matrices, 
\begin{equation}
\mathcal{C}^{(a,b,c+d+2\ell+1)} = \mathcal{R}_{\ell}^{(a,b,c,d)}\mathcal{C}^{(a,b,c+d+2\ell-1)}, \qquad \ell = 1, \ldots, k.  \label{eq:conmatrec}
\end{equation}
The measure modification function between the semiclassical OP families $\mathbf{R}^{(a,b,c+d+2\ell-1)}$ and $\mathbf{R}^{(a,b,c+d+2\ell+1)}$ is $\rho^2$ because
\begin{equation*}
	w_R^{(a,b,c+d+2\ell+1)} = \rho^2 w_R^{(a,b,c+d+2\ell-1)}, \quad 1 \leq \ell \leq k.
\end{equation*}
The associated measure modification matrix between the semiclassical OP families is defined by
\begin{equation*}
	\rho^2\left( J(w_{R}^{(a,b,c+d+2\ell-1)}) \right) = \left\langle \mathbf{R}^{(a,b,c+d+2\ell-1)}, \rho^2\mathbf{R}^{(a,b,c+d+2\ell-1)} \right\rangle_{w_{R}^{(a,b,c+d+2\ell-1)}},
\end{equation*}
which follows from (\ref{eq:polymult}) and the fact that $\rho^2$ is a polynomial; here $J(w^{(a,b,c+d+2\ell-1)}_R)$ denotes the Jacobi matrix of $\mathbf{R}^{(a,b,c+d+2\ell-1)}$, the computation of which will be discussed in the next section.  We note that the modification matrix is symmetric, has bandwidths $(\deg \rho^2, \deg \rho^2)$ and is positive definite since $\rho^2$ is positive on $[\alpha, \beta]$.  It follows from~\cite[Lemma 2.8]{Gutleb} that $\mathcal{R}_\ell^{(a,b,c,d)}$ is the Cholesky factor of the associated modification matrix, i.e., 
\begin{equation}
\rho^2\left( J(w_{R}^{(a,b,c+d+2\ell-1)}) \right) = {\mathcal{R}_{\ell}^{(a,b,c,d)}}^{\top} \mathcal{R}_{\ell}^{(a,b,c,d)},  \label{eq:raisechol}
\end{equation}
and $\mathcal{R}_{\ell}^{(a,b,c,d)}$ has bandwidths $(0, \deg \rho^2)$.    We compute the $N \times N$ principal finite section of $\mathcal{R}_{\ell}^{(a,b,c,d)}$ by computing the principal finite section of $J(w_{R}^{(a,b,c+d+2\ell-1)})$, evaluating the modification matrix $\rho^2\left( J(w_{R}^{(a,b,c+d+2\ell-1)}) \right) $ via the Clenshaw algorithm \cite{Clenshaw1,Clenshaw2} and then computing its Cholesky factor, which has complexity $\mathcal{O}(d_{\eta}^2 N)$. 
Computing finite sections of $\mathcal{R}_{\ell}^{(a,b,c,d)}$ via Cholesky factorisations for $\ell = 1, \ldots, k$ (and having computed the connection matrix for $k = 0$ using the moment-based approach as described in the previous section), we obtain (finite sections of) connection matrices with $k > 0$ via (\ref{eq:conmatrec})\footnote{To compute a single $N\times N$ connection matrix of bandwidths $(0,\mu)$, the moment-based method has complexity $\mathcal{O}(\mu N)$, while the Cholesky factorisation method has $\mathcal{O}(\mu^2 N)$ complexity.  The reason we do not use the moment-based method for connection matrices with $k >0$ is because in order to compute an $N\times N$   
  principal finite section of the connection matrix $\mathcal{C}^{(a,b,c+d+2k+1)}$  with this method, one needs to compute an $M\times M$ principal finite section of the $k = 0$ connection matrix $\mathcal{C}^{(a,b,c+d+1)}$, where $M = \mathcal{O}(2^kN)$, 
 which becomes prohibitively expensive for large $k$ ($k$ ranges between $0$ and $N$).  Hence, we use the moment-based method only for $k = 0$. }.

\subsection{Jacobi matrices of semiclassical OPs} \label{SemiclassicalJacobi}
The Jacobi matrices of  the  families of normalised semiclassical OPs are symmetric, tridiagonal and satisfy 
\begin{align}
	x \,&\mathbf{R}^{(a,b,c+d+2k+1)} = \mathbf{R}^{(a,b,c+d+2k+1)} \, J(w_R^{(a,b,c+d+2k+1)}). \label{eq:semiclassjacdef}
\end{align}
For $k = 0$, we use  (\ref{ConnMatSR}) to deduce that
\begin{align}
	x \,\mathbf{R}^{(a,b,c+d+1)} &= x \, \mathbf{S}^{(a,b)} \,\left(\mathcal{C}^{(a,b,c+d+1)} \right)^{-1} = \mathbf{S}^{(a,b)} \, J(w_S^{(a,b)}) \, \left(\mathcal{C}^{(a,b,c+d+1)} \right)^{-1} \notag\\
	&= \mathbf{R}^{(a,b,c+d+1)} \,\mathcal{C}^{(a,b,c+d+1)} \, J(w_S^{(a,b)}) \, \left(\mathcal{C}^{(a,b,c+d+1)} \right)^{-1}.  \label{eq:semiclassjac}
\end{align}
Thus, comparing (\ref{eq:semiclassjacdef}) and (\ref{eq:semiclassjac}) with $k = 0$, there is an infinite-dimensional similarity transformation between $J(w_R^{(a,b,c+d+1)})$ 
and the (known) Jacobi matrix $J(w^{(a,b)}_S)$ of the classical OPs:
\begin{equation}\label{semiclassicalJACOBI}
	J(w_R^{(a,b,c+d+1)}) = \mathcal{C}^{(a,b,c+d+1)} \, J(w_S^{(a,b)}) \, \left(\mathcal{C}^{(a,b,c+d+1)} \right)^{-1}.
\end{equation}
The Jacobi matrices for $k >0$ are derived just as for the case $k = 0$, but using (\ref{eq:semiclassraise}), and then we obtain the recursive formula 
\begin{equation}
J(w_R^{(a,b,c+d+2\ell+1)}) = \mathcal{R}_{\ell}^{(a,b,c,d)} \, J(w_R^{(a,b,c+d+2\ell-1)}) \, \left(\mathcal{R}_{\ell}^{(a,b,c,d)} \right)^{-1}, \qquad \ell = 1, \ldots, k.  \label{eq:semiclassjackg0}
\end{equation}
                        To compute the Jacobi matrices (\ref{semiclassicalJACOBI}), (\ref{eq:semiclassjackg0}) and also the differentiation matrices for semiclassical OPs (to be discussed in the next subsection), we shall use the following result, which is a generalisation  of the approach used to compute Jacobi matrices of semiclassical OPs in~\cite{Gutleb} and differentiation matrices of semiclassical OPs in~\cite{Papadopoulos1} to banded matrices of arbitrary finite bandwidths.  
             \begin{lemma} \label{thm:MatrixProduct}
Let $A$, $B$ and $C$ be infinite-dimensional matrices satisfying
 $A = CB^{-1}$, where 
	$A$ is an unknown banded matrix with bandwidths $(\lambda, \mu)$, $B$ is an upper-triangular matrix with nonzero diagonal entries,
	and $C$ is an arbitrary (possibly unstructured) matrix. 
	Then the $N\times N$ principal finite section of $A$ can be computed by
  	solving the following sequence of lower-triangular systems of dimension at most $(\lambda + \mu + 1)\times (\lambda + \mu + 1)$: 
	\begin{align*}
		&\begin{pmatrix}
		B_{[i+j_0,i+j_0]} & 0  & \cdots  & 0 \\
		B_{[i+j_0,i+j_0+1]} & B_{[i+j_0+1,i+j_0+1]}  & \cdots  & 0 \\
		\vdots & \vdots  & \ddots  & \vdots \\
		B_{[i+j_0,i+j_e ]} & B_{[i+j_0+1, i+j_e]}  & \cdots & B_{[i+j_e,i+j_e]} 
		\end{pmatrix}
		\begin{pmatrix}
		A_{[i,i+j_0]} \\
		A_{[i,i+j_0+1]} \\
		\vdots \\
		A_{[i,i+j_e]} 
		\end{pmatrix}
		=
		\begin{pmatrix}
		C_{[i,i+j_0]} \\
		C_{[i,i+j_0+1]} \\
		\vdots \\
		C_{[i,i+j_e]} 
		\end{pmatrix},
	\end{align*}
	where  $\max\{ 0, -\mu \} \leq i \leq \min\{ N-1, N-1+\lambda\}$, $j_0 = \max\{ -\lambda, -i \}$ and $j_e = \min\{ \mu, N-1-i \}$.  Hence, $A$ is computed with $\mathcal{O}\left((\lambda + \mu)^2 N\right)$ complexity, i.e., linear complexity in $N$. 
 \end{lemma}
 Hence it follows from Lemma~\ref{thm:MatrixProduct} that the Jacobi matrices (\ref{semiclassicalJACOBI})--(\ref{eq:semiclassjackg0}) of the semiclassical OPs can be computed in linear complexity in $N$ by solving a sequence of lower-triangular systems of size at most $3\times 3$.

\subsection{Differentiation matrices of semiclassical OPs}
The differentiation matrix for the semiclassical OPs in (\ref{KoornwinderFomula}) is defined as
\begin{align}\label{DMdef}
	\frac{\mathrm{d}}{\mathrm{d}x} \mathbf{R}^{(a,b,c+d+2k+1)} = \mathbf{R}^{(a+1,b+1,c+d+2k+2+\eta)}\,D_\eta(w^{(a,b,c+d+2k+1)}_R).
\end{align}
The incrementing and decrementing of parameters as seen here is analogous to other well-known univariate OP families' derivatives \cite{NIST,SnowballOlver,TriangleOPs}.  The proof of the following result is given in Appendix~\ref{AppendixA}.
\begin{proposition}\label{bandwidths_diff}
	$D_\eta(w^{(a,b,c+d+2k+1)}_R)$ is banded with bandwidths $(-1, 1 + d_{\eta})$.
\end{proposition}
  To compute the differentiation matrix for $k = 0$, we use the connection matrix (\ref{ConnMatSR}) and the (known) differentiation matrix of classical OPs:
 \begin{eqnarray*}
\frac{\mathrm{d}}{\mathrm{d}x} \mathbf{R}^{(a,b,c+d+1)} &=& \frac{\mathrm{d}}{\mathrm{d}x} \mathbf{S}^{(a,b)} \, \left( \mathcal{C}^{(a,b,c+d+1)} \right)^{-1} = \mathbf{S}^{(a+1,b+1)} \, D(w_S^{(a,b)}) \, \left( \mathcal{C}^{(a,b,c+d+1)} \right)^{-1}\\
	&=& \mathbf{R}^{(a+1,b+1,c+d+2+\eta)} \, \mathcal{C}^{(a+1,b+1,c+d+2+\eta)} \, D(w_S^{(a,b)}) \, \left( \mathcal{C}^{(a,b,c+d+1)} \right)^{-1},
\end{eqnarray*}
hence
\begin{align}
	D_\eta(w^{(a,b,c+d+1)}_R) = \mathcal{C}^{(a+1,b+1,c+d+2+\eta)} \, D(w_S^{(a,b)}) \, \left( \mathcal{C}^{(a,b,c+d+1)} \right)^{-1},  \label{eq:semiclassDMkzero}
\end{align}
where the connection matrices in (\ref{eq:semiclassDMkzero}) are computed using the moment-based method described in Section~\ref{sect:kzero}.

The differentiation matrices for $k > 0$ can be derived similarly, but using (\ref{eq:semiclassraise}) and (\ref{DMdef}), we obtain the recursive formula
\begin{equation}
D_\eta(w_R^{(a,b,c+d+2\ell+1)}) = \mathcal{R}_\ell^{(a+1,b+1,c+1,d+\eta)} \, D_\eta(w_R^{(a,b,c+d+2\ell-1)}) \, \left(\mathcal{R}_\ell^{(a,b,c,d)} \right)^{-1}, \qquad \ell = 1, \ldots, k,  \label{eq:semiclassDMskg0}
\end{equation}
where the raising matrices in (\ref{eq:semiclassDMskg0}) are computed using the Cholesky factorisation approach in Section~\ref{sect:connectkg0}.

It follows from Lemma~\ref{thm:MatrixProduct} and Proposition~\ref{bandwidths_diff} that
 a semiclassical differentiation matrix can be computed in $\mathcal{O}\left(d_{\eta}^2 N  \right)$ complexity.

\subsection{Raising matrices of semiclassical OPs} \label{SemiclassicalRaising}

      The final type of operator matrices of univariate OPs we shall use in the next section is raising matrices satisfying
 \begin{align}\label{SemiclassicalRaisingfomula}
	\mathbf{R}^{(a,b,c+d+2k+1)} = \mathbf{R}^{(a+\lambda_1,b+\lambda_2,c+d+\lambda_3+\lambda_4+2k+1)} \, \mathcal{T}_{(a,b,c+d+2k+1)}^{(a+\lambda_1,b+\lambda_2,c+d+\lambda_3+\lambda_4+2k+1)},
 \end{align}
where $\lambda_1, \lambda_2, \lambda_3, \lambda_4 \in \mathbb{N}_0$  
  and in case (ii), $\lambda_3 = \lambda_4$.  These matrices are computed using the approach used to compute the raising matrices  (\ref{eq:semiclassraise}): since the weight modification function between the semiclassical OPs in (\ref{SemiclassicalRaisingfomula}) is $w_R^{(\lambda_1,\lambda_2, \lambda_3+\lambda_4)}$, it follows  from~\cite[Lemma 2.8]{Gutleb} that the raising matrix in (\ref{SemiclassicalRaisingfomula}) is the Cholesky factor of the modification matrix $w_R^{(\lambda_1,\lambda_2, \lambda_3+\lambda_4)}\left(J(w^{(a,b,c+d+2k+1)}_R)\right)$, which has bandwiths $(\ell, \ell)$, where $\ell = \lambda_1+\lambda_2+  (\lambda_3+\lambda_4)\deg \rho^2/2$ and thus the raising matrix (\ref{SemiclassicalRaisingfomula}) has bandwidths $(0, \ell)$ and computing its principal finite section in this manner has $\mathcal{O}(\ell^2N)$ complexity.

\section{Sparse operator matrices of generalised Koornwinder polynomials} \label{Operator_matrices _Koornwinder_polynomials}
In this section, we construct sparse operator matrices for the (bivariate) generalised Koornwinder polynomials $\mathbf{H}^{(a,b,c,d)}$ using entries of operator matrices of univariate OPs derived in sections~\ref{sect:classopmats} and~\ref{sect:semiclassconnmats}. 
    
\subsection{Multiplication matrices of generalised Koornwinder polynomials}
The matrices $J_{x}^{(a,b,c,d)}$ and $J_{y}^{(a,b,c,d)}$ represent multiplication of $\mathbf{H}^{(a,b,c,d)}$ by $x$ and $y$, respectively:
\begin{equation*}
	x\mathbf{H}^{(a,b,c,d)}=\mathbf{H}^{(a,b,c,d)}J_{x}^{(a,b,c,d)},\quad y\mathbf{H}^{(a,b,c,d)}=\mathbf{H}^{(a,b,c,d)}J_{y}^{(a,b,c,d)}.
\end{equation*}
\begin{proposition}
The generalised Koornwinder polynomials satisfy
  	\begin{align}
		x H_{n,k}^{(a,b,c,d)}(x,y) = \sum_{m=n-1}^{n+1}  l_{m,k}^{\,x, (n,k)} \, H_{m, k}^{(a,b,c,d)}(x, y),  \label{eq:multx}
	\end{align}
	where the coefficients $l_{m,k}^{\,x, (n,k)}$ are given by 
 	\begin{align}
		l_{m,k}^{\,x, (n,k)} = \left( J(w_R^{(a,b,c+d+2k+1)}) \right)_{[n-k,m-k]}, \quad n-1\leq m \leq n+1,  \label{eq:multxcoeffs}
	\end{align}
	Therefore, $J_{x}^{(a,b,c,d)}$ is symmetric, block-tridiagonal (it has block-bandwidths $(1, 1)$) with diagonal blocks (each block has bandwidths $(0, 0)$).
\end{proposition}
Equations (\ref{eq:multx}) and (\ref{eq:multxcoeffs}) can readily be  derived from the three-term recurrence satisfied by the OPs $\lbrace R_n^{(a,b,c+d+2k+1)}\rbrace$, see (\ref{eq:semiclassjacdef}). It follows from (\ref{eq:multx}) and (\ref{eq:multxcoeffs}) that $J_{x}^{(a,b,c,d)}$  is given by
 \begin{align*}
	J_{x}^{(a,b,c,d)} &= \begin{pmatrix}
                B^{x}_0 & A^{x}_0 & & & & \\
			\left(A^{x}_0\right)^{\top} & B^{x}_1 & A^{x}_1 & & & \\
			& \left(A^{x}_1\right)^{\top} & B^{x}_2 & A^{x}_2  & & & \\
			& & \left(A^{x}_2\right)^{\top} & \ddots & \ddots & \\
			& & & \ddots & \ddots & \ddots \\
	\end{pmatrix}
\end{align*}
where 
\begin{align*}
	A^x_n &:= \begin{pmatrix}
		l_{n,0}^{\,x, (n+1,0)} & 0 & \dots & 0 \\
			& \ddots & & \vdots \\
			& & l_{n,n}^{\,x, (n+1,n)} & 0 \\
			\end{pmatrix} \in \mathbb{R}^{(n+1)\times(n+2)}, \quad n = 0,1,2,\dots\\
	B^x_n &:= \begin{pmatrix}
		l_{n,0}^{\,x, (n,0)} & \\
			& \ddots & \\
			& & l_{n,n}^{\,x, (n,n)} \\
			\end{pmatrix} \in \mathbb{R}^{(n+1)\times(n+1)}, \quad n = 0,1,2,\dots.
  \end{align*}

Whereas multivariate orthogonal polynomials also have block-tridiagonal multiplication-by-$y$ operators, our non-graded bases only guarantees that they are block-banded:

\begin{proposition}\label{M_y}
  	The generalised Koornwinder polynomials satisfy
	\begin{align*} \label{M_y_r}
		y H_{n,k}^{(a,b,c,d)}(x,y) &= \sum_{m=n-\deg\rho^2 + 1}^{n+1}  l_{m,k+1}^{\,y, (n,k)} \, H_{m, k+1}^{(a,b,c,d)}(x, y)
		+  \sum_{m=n-1}^{n+\deg\rho^2-1}  l_{m,k-1}^{\,y, (n,k)} \, H_{m, k-1}^{(a,b,c,d)}(x, y)  \\ 
		& + \sum_{m=n-d_0}^{n+d_0}  l_{m,k}^{\,y, (n,k)} \, H_{m, k}^{(a,b,c,d)}(x, y),
	\end{align*}
	where 
 	\begin{align*}
		& \, l_{m,k+1}^{\,y, (n,k)} = \delta_{k}^{(d,c)} \, \left(\mathcal{T}_{(a,b,c+d+2k+1)}^{(a,b,c+d+2k+3)} \right)_{[m-k-1,n-k]}, \quad n-\deg\rho^2+1 \leq m \leq n+1,\\\\
		& \, l_{m,k-1}^{\,y, (n,k)} = \delta_{k-1}^{(d,c)} \, \left(\mathcal{T}_{(a,b,c+d+2k-1)}^{(a,b,c+d+2k+1)} \right)_{[n-k,m-k+1]} , \quad n-1 \leq m \leq n+\deg\rho^2-1,\\\\
		& \, l_{m,k}^{\,y, (n,k)} = \begin{cases}
			\textnormal{case (i):}\quad \gamma_{k}^{(d,c)} \times \left( \rho \left[J(w_R^{(a,b,c+d+2k+1)})\right] \right)_{[n-k,m-k]} , \quad n-d_0\leq m \leq n+d_0\\
			\textnormal{case (ii):}\quad 0
		\end{cases}.
	\end{align*}
Consequently, $J_{y}^{(a,b,c,d)}$ has block-bandwidths $(\deg \rho^2 -1, \deg\rho^2 - 1)$ and each block has bandwidths $(1, 1)$ and moreover, $J_{y}^{(a,b,c,d)}$ is symmetric.
 \end{proposition}
\begin{proof}
Since generalised Koornwinder polynomials span the space of bivariate polynomials (see Corollary~\ref{cor:Koornspandeg}), there exists a finite expansion such that 
 	\begin{equation*}
		y H_{n,k}^{(a,b,c,d)}(x,y) = \sum_{m \geq 0} \sum_{j = 0}^{m} l_{m,j}^{\,y, (n,k)} \, H_{m,j}^{(a,b,c,d)}(x,y).  
	\end{equation*}
Since the generalised Koornwinder polynomials are orthonormal, using (\ref{KoornwinderFomula})--(\ref{eq:inproddef}), (\ref{eq:jacthreeterm})  and the change of variable $t = y/\rho(x)$, we have that
 	\begin{align}
		l_{m,j}^{\,y, (n,k)} &= \langle y H_{n,k}^{(a,b,c,d)},  H_{m,j}^{(a,b,c,d)}\rangle_{W^{(a,b,c,d)}}
		=  \iint_\Omega H_{n,k}^{(a,b,c,d)}(x,y) \, H_{m,j}^{(a,b,c,d)}(x,y) \, y \, W^{(a,b,c,d)}(x,y) \, \mathrm{d}y \, \mathrm{d}x \nonumber \\
		&= \Big( \int^\beta_\alpha R_{n-k}^{(a, b, c+d+2k+1)}(x) \, R_{m-j}^{(a, b, c+d+2j+1)}(x) \, \rho(x)^{k+j+2} \, w_R^{(a, b, c+d)}(x) \, \mathrm{d}x \Big) \nonumber \\
		&\quad \quad \quad \quad \quad\cdot \, \Big( \int^{\delta}_{\gamma} \widetilde{P}_k^{(d,c)}(t) \, \widetilde{P}_j^{(d,c)}(t) \, t \, w_P^{(d,c)}(t) \, \mathrm{d}t \Big) \nonumber \\
		&= \Big( \int^\beta_\alpha R_{n-k}^{(a, b, c+d+2k+1)}(x) \, R_{m-j}^{(a, b, c+d+2j+1)}(x) \, w_R^{(a, b, c+d+k+j+2)}(x) \, \mathrm{d}x \Big) \nonumber \\
		&\quad \quad \quad \quad \quad\cdot \, \Big( \int^{\delta}_{\gamma} \widetilde{P}_k^{(d,c)}(t) \, \widetilde{P}_j^{(d,c)}(t) \, t \, w_P^{(d,c)}(t) \, \mathrm{d}t \Big) \nonumber \\
		&=
		\begin{cases}
			\delta_k^{(d,c)} \, \langle R_{n-k}^{(a, b, c+d+2k+1)}, R_{m-k-1}^{(a, b, c+d+2k+3)}\rangle_{w_R^{(a, b, c+d+2k+3)}}  \quad& \text{if } j = k+1 \\
			\delta_{k-1}^{(d,c)} \, \langle R_{n-k}^{(a, b, c+d+2k+1)}, R_{m-k+1}^{(a, b, c+d+2k-1)}\rangle_{w_R^{(a, b, c+d+2k+1)}} \quad& \text{if } j = k-1 \\
			\gamma_k^{(d,c)} \, \langle R_{n-k}^{(a, b, c+d+2k+1)}, \rho(x) \,R_{m-k}^{(a, b, c+d+2k+1)}\rangle_{w_R^{(a, b, c+d+2k+1)}} \quad& \text{if } j = k \\
			0 & \text{otherwise}
		\end{cases}  \label{eq:multycases}
	\end{align}
         	The first two inner products in (\ref{eq:multycases}) are entries of the raising matrices $\mathcal{T}_{(a,b,c+d+2k+1)}^{(a,b,c+d+2k+3)}$ and $\mathcal{T}_{(a,b,c+d+2k-1)}^{(a,b,c+d+2k+1)}$, respectively (see Section~\ref{SemiclassicalRaising}).  Since these raising matrices have bandwidths $(0, \deg\rho^2)$, for $j=k+1$ the coefficients $l_{m,j}^{\,y, (n,k)}$ can only be nonzero for $n-\deg\rho^2+1\leq m \leq n+1$.  For $j= k-1$, the coefficients can be nonzero only for $n-1\leq m \leq n+\deg\rho^2-1$.
   	In case (i), the third inner product are the entries of $\rho \left(J(w_R^{(a,b,c+d+2k+1)})\right)$, which has bandwidths $(d_0, d_0)$; hence the coefficients $l_{m,k}^{\,y, (n,k)}$ can only be nonzero  for $n-d_0 \leq m \leq n+d_0$.
	In case (ii), since $c=d$ and $\gamma_k^{(d,c)} = 0$, it follows that $l_{m,k}^{\,y, (n,k)} = 0$.
	
	The symmetry of $J_{y}^{(a,b,c,d)}$ follows from the orthonormality of the generalised Koornwinder polynomials and the fact that multiplication-by-$y$ is self-adjoint, i.e., $\langle y H_{n,k}^{(a,b,c,d)},  H_{m,j}^{(a,b,c,d)}\rangle_{W^{(a,b,c,d)}} = \langle  H_{n,k}^{(a,b,c,d)},  y H_{m,j}^{(a,b,c,d)}\rangle_{W^{(a,b,c,d)}}$.
\end{proof}

\subsection{Differentiation matrices of generalised Koornwinder polynomials}\label{sec:Kdiffmats}
We denote the weighted generalised Koornwinder polynomials by 
\begin{align}
	\mathbf{W}^{(a,b,c,d)}(x,y) := W^{(a,b,c,d)}(x,y) \, \mathbf{H}^{(a,b,c,d)}(x,y),  \label{eq:WKoorn}
\end{align}
 and	define the differentiation matrices $D_{x}^{(a,b,c,d)},  D_y^{(a,b,c,d)},  W_{x}^{(a,b,c,d)}$ and $W_y^{(a,b,c,d)}$ by
	\begin{align}
		{\partial  \over \partial x}\mathbf{H}^{(a,b,c,d)}(x,y) &= \mathbf{H}^{(a+1,b+1,c+1,d+1)}(x,y) \: D_{x}^{(a,b,c,d)}, \label{eq:PartialXdef}\\
		{\partial  \over \partial y}\mathbf{H}^{(a,b,c,d)}(x,y) &= \mathbf{H}^{(a,b, c+1,d+1)}(x,y) \: D_y^{(a,b,c,d)}, \\
		{\partial \over \partial x}\mathbf{W}^{(a,b,c,d)}(x,y) &= \mathbf{W}^{(a-1,b-1,c-1,d-1)}(x,y) \: W_{x}^{(a,b,c,d)}, \\
		{\partial \over \partial y}\mathbf{W}^{(a,b,c,d)}(x,y) &= \mathbf{W}^{(a,b,c-1,d-1)}(x,y) \: W_y^{(a,b,c,d)}. \label{eq:Partialdef}
	\end{align}
 The incrementing and decrementing of parameters here is analogous to other well-known bivariate OP families' derivatives \cite{SnowballOlver,TriangleOPs}.
\begin{proposition}\label{thm:partialy}
 The generalised Koornwinder polynomials satisfy
	\begin{equation} \label{eq:partyformula}
		{\partial  \over \partial y}H_{n,k}^{(a,b,c,d)}(x,y) = d^{(d+1,c+1)}_{k-1}\,H_{n-1,k-1}^{(a,b,c+1,d+1)}(x,y),  
	\end{equation}
 where	 $d^{(d+1,c+1)}_{k-1}$ is given in (\ref{ClassicalDerivative}), and thus $D_y^{(a,b,c,d)}$ has block-bandwidths $(-1, 1)$ and each block has bandwidths $(-1, 1)$.
\end{proposition}
\begin{proof}
Similar to the proof of Proposition~\ref{M_y}, the result follows from calculating the expansion coefficients $c_{m,j}^{\,y, (n,k)}$, where
\begin{equation}
		{\partial  \over \partial y} H_{n,k}^{(a,b,c,d)} = \sum_{m \geq 0}\sum_{j=0}^m c_{m,j}^{\,y, (n,k)} H_{m,j}^{(a,b,c+1,d+1)}, \quad c_{m,j}^{\,y, (n,k)} = \left\langle {\partial  \over \partial y} H_{n,k}^{(a,b,c,d)},  H_{m,j}^{(a, b, c+1,d+1)} \right\rangle_{W^{(a, b, c+1,d+1)}},  
		\label{eq:partyexp}
	\end{equation}
and using (\ref{ClassicalDerivative}).
                \end{proof}

\begin{proposition}
 Let $a, b, c, d > 0$, then we have that
	\begin{align*}
		{\partial  \over \partial y}[W^{(a,b,c,d)}(x,y)\,H_{n,k}^{(a,b,c,d)}(x,y)] = - d^{(d,c)}_{k}\,W^{(a,b,c-1,d-1)}(x,y)\,H_{n+1,k+1}^{(a,b,c-1,d-1)}(x,y),
	\end{align*}
and thus $W_y^{(a,b,c,d)}$ has block-bandwidths $(1, -1)$ and each block has bandwidths $(1, -1)$.
\end{proposition}

\begin{proof}
We first note that 
   \begin{align*}
		\langle \,W^{(a,b,c,d)} \, H_{n,k}^{(a,b,c,d)} , \, W^{(a,b,c,d)} \, H_{m,j}^{(a,b,c,d)} \,\rangle_{W^{(-a,-b,-c,-d)}} = \langle \,H_{n,k}^{(a,b,c,d)} , \, H_{m,j}^{(a,b,c,d)} \,\rangle_{W^{(a,b,c,d)}} = \delta_{n,m}\,\delta_{k,j},
   \end{align*}
   which shows that $\{W^{(a,b,c,d)}\,H_{n,k}^{(a,b,c,d)}\}$ are mutually orthonormal with respect to $W^{(-a,-b,-c,-d)}$. 
 Therefore, the weighted derivative has the expansion  
    \begin{equation*}
	{\partial  \over \partial y}[W^{(a,b,c,d)}(x,y)\,H_{n,k}^{(a,b,c,d)}(x,y)] = \sum_{m \geq 0}\sum_{j = 0}^{m} \bar{c}_{m,j}^{\,y, (n,k)} \,W^{(a,b,c-1,d-1)}(x,y)\,H_{m,j}^{(a,b,c-1,d-1)}(x,y).
   \end{equation*}
where, using integration-by-parts, (\ref{eq:partyformula}) and (\ref{eq:partyexp}),
    \begin{align*}
		\bar{c}_{m,j}^{\,y, (n,k)} &= \left\langle \,{\partial  \over \partial y}(W^{(a,b,c,d)}\,H_{n,k}^{(a,b,c,d)}) , \, W^{(a,b,c-1,d-1)}\,H_{m,j}^{(a,b,c-1,d-1)} \,\right\rangle_{W^{(-a,-b,1-c,1-d)}} \\
	   &=  \int_\alpha^\beta \int_{\gamma\rho(x)}^{\delta\rho(x)}\: {\partial  \over \partial y}[W^{(a,b,c,d)}(x,y)\,H_{n,k}^{(a,b,c,d)}(x,y)] \: H_{m,j}^{(a,b,c-1,d-1)}(x,y) \: \mathrm{d}y \: \mathrm{d}x  \\
	   &=  - \int_\alpha^\beta \int_{\gamma\rho(x)}^{\delta\rho(x)}\: {\partial  \over \partial y}H_{m,j}^{(a,b,c-1,d-1)}(x,y) \: [H_{n,k}^{(a,b,c,d)}(x,y)\:W^{(a,b,c,d)}(x,y)]  \: \mathrm{d}y \: \mathrm{d}x  \\
	   &= - \left\langle \,{\partial  \over \partial y}H_{m,j}^{(a,b,c-1,d-1)}, \, H_{n,k}^{(a,b,c,d)} \, \right\rangle_{W^{(a,b,c,d)}} = - d^{(d,c)}_{k}\,\delta_{n,m-1}\,\delta_{k,j-1}, 
   \end{align*}	
   and thus $\bar{c}_{m,j}^{\,y, (n,k)}$ can only be nonzero when $m=n+1$ and $j=k+1$.
  \end{proof}

The proof of the following two theorems is presented in Appendix~\ref{AppendixB}.

\begin{proposition} \label{Diff_matrix}
The generalised Koornwinder polynomials satisfy
 	\begin{align*}
		{\partial  \over \partial x}&H_{n,k}^{(a,b,c,d)}(x,y) = \sum_{m=n-\deg\rho^2-1}^{n-1}  c_{m,k}^{\,x, (n,k)} \, H_{m, k}^{(a+1,b+1,c+1,d+1)}(x, y) \\ 
		&+  \sum_{m=n-3}^{n+\deg\rho^2-3}  c_{m,k-2}^{\,x, (n,k)} \, H_{m, k-2}^{(a+1,b+1,c+1,d+1)}(x, y)  + \sum_{m=n-d_0-2}^{n+d_0-2}  c_{m,k-1}^{\,x, (n,k)} \, H_{m, k-1}^{(a+1,b+1,c+1,d+1)}(x, y).
	\end{align*}
The expressions for $c_{m,k}^{\,x, (n,k)}$ and $c_{m,k-2}^{\,x, (n,k)}$  are given in, respectively, (\ref{eq:partialxcmk}) and (\ref{remark}); in case (i), $c_{m,k-1}^{\,x, (n,k)}$ is given by (\ref{eq:partialxcmkm1}) and in case (ii), these coefficients are zero, $c_{m,k-1}^{\,x, (n,k)} = 0$.
 Therefore, $D_x^{(a,b,c,d)}$ has block-bandwidths $(\deg\rho^2 - 3, \deg\rho^2 + 1)$ and each block has bandwidths $(0,2)$.
\end{proposition}

\begin{proposition} \label{weighted_Diff_matrix}
	Let $a, b, c, d > 0$, then we have that
 	\begin{align*} 
		&{\partial  \over \partial x}[W^{(a,b,c,d)}(x,y)\,H_{n,k}^{(a,b,c,d)}(x,y)] = \sum_{m=n+1}^{n+\deg\rho^2+1} \bar{c}_{m,k}^{\,x, (n,k)}\,W^{(a-1,b-1,c-1,d-1)}(x,y)\,H_{m,k}^{(a-1,b-1,c-1,d-1)}(x,y) \nonumber \\
		& \qquad \qquad \qquad  + \sum_{m=n-\deg\rho^2+3}^{n+3} \bar{c}_{m,k+2}^{\,x, (n,k)}\,W^{(a-1,b-1,c-1,d-1)}(x,y)\,H_{m,k+2}^{(a-1,b-1,c-1,d-1)}(x,y) \nonumber \\
		& \qquad \qquad \qquad  + \sum_{m=n-d_0+2}^{n+d_0+2} \bar{c}_{m,k+1}^{\,x, (n,k)}\,W^{(a-1,b-1,c-1,d-1)}(x,y)\,H_{m,k+1}^{(a-1,b-1,c-1,d-1)}(x,y).
	\end{align*}
Expressions for the coefficients $\bar{c}_{m,k}^{\,x, (n,k)}, \bar{c}_{m,k+2}^{\,x, (n,k)}$ and $\bar{c}_{m,k+1}^{\,x, (n,k)}$ are given in (\ref{eq:partialxwcmk})--(\ref{eq:partialxwcmkp1}).
 Hence, $W_x^{(a,b,c,d)}$ has block-bandwidths $(\deg \rho^2 + 1, \deg\rho^2 - 3)$ and each block has bandwidths $(2,0)$.
\end{proposition}

\subsection{Conversion matrices of generalised Koornwinder polynomials}\label{sec:Kconversion}
    To construct sparse operator matrices for the partial differential operators in the next section, we require the following conversion (raising and lowering) matrices for incrementing or decrementing  parameters of non-weighted or weighted generalised Koornwinder polynomials:
\begin{align}
	\mathbf{H}^{(a,b,c,d)}(x,y) &= \mathbf{H}^{(a+1,b+1,c,d)}(x,y) \: T_{(a,b,c,d)}^{(a+1,b+1,c,d)}, \label{eq:koornconv1} \\
	\mathbf{H}^{(a,b,c,d)}(x,y) &= \mathbf{H}^{(a,b,c+1,d+1)}(x,y) \: T_{(a,b,c,d)}^{(a,b,c+1,d+1)} ,\label{eq:koornconv2}\\
 	\mathbf{W}^{(a,b,c,d)}(x,y) &= \mathbf{W}^{(a-1,b-1,c,d)}(x,y) \: T_{W,(a,b,c,d)}^{(a-1,b-1,c,d)},\label{eq:wkoornconv1}\\
	\mathbf{W}^{(a,b,c,d)}(x,y) &= \mathbf{W}^{(a,b,c-1,d-1)}(x,y) \: T_{W,(a,b,c,d)}^{(a,b,c-1,d-1)}. \label{eq:wkoornconv2}
 \end{align}
By composing (\ref{eq:koornconv1}) and (\ref{eq:koornconv2}), we obtain $T_{(a,b,c,d)}^{(a+1,b+1,c+1,d+1)}$, which maps $\mathbf{H}^{(a,b,c,d)}$ to $\mathbf{H}^{(a+1,b+1,c+1,d+1)}$,
\begin{equation}
 T_{(a,b,c,d)}^{(a+1,b+1,c+1,d+1)} = T_{(a+1,b+1,c,d)}^{(a+1,b+1,c+1,d+1)}T_{(a,b,c,d)}^{(a+1,b+1,c,d)},  \label{eq:Tinc1}
\end{equation}
and similarly we obtain $T_{W,(a,b,c,d)}^{(a-1,b-1,c-1,d-1)}$ by composing (\ref{eq:wkoornconv1}) and (\ref{eq:wkoornconv2}),
\begin{equation}
T_{W,(a,b,c,d)}^{(a-1,b-1,c-1,d-1)} =    T_{W,(a-1,b-1,c,d)}^{(a-1,b-1,c-1,d-1)}T_{W,(a,b,c,d)}^{(a-1,b-1,c,d)}. \label{eq:TWdec1}
\end{equation}
  \begin{proposition} \label{structure_conversion_mat}
	The conversion matrices are sparse, with banded-block-banded structure. More specifically:
	\begin{itemize}
		\item $T_{(a,b,c,d)}^{(a+1,b+1,c,d)}$ has block-bandwidths $(0,2)$ and each block has bandwidths $(0,0)$.
		\item $T_{(a,b,c,d)}^{(a,b,c+1,d+1)}$ has block-bandwidths $(\deg \rho^2 - 2,\deg \rho^2)$ and each block has bandwidths $(0,2)$.
 		\item $T_{W,(a,b,c,d)}^{(a-1,b-1,c,d)}$ has block-bandwidths $(2,0)$ and each block has bandwidths $(0,0)$.
		\item $T_{W,(a,b,c,d)}^{(a,b,c-1,d-1)}$ has block-bandwidths $(\deg\rho^2, \deg\rho^2 - 2)$ and each block has bandwidths $(2,0)$.
 	\end{itemize}
\end{proposition}
\noindent An outline of the proof of Proposition~\ref{structure_conversion_mat} is given in Appendix~\ref{AppendixC}.
 
\subsection{Sparse representations of differential operators}\label{sec:Laplace}

By composing the above multiplication, partial differentiation and conversion matrices associated with generalised Koornwinder polynomials, one can construct sparse matrix representations for linear differential operators.  

In our numerical experiments, we shall use the following sparse matrix representations of the Laplacian ($\Delta$) and biharmonic $(\Delta^2)$ operators, denoted by, respectively, $\Delta_{W,(1,1,1,1)}^{(1,1,1,1)}$ and ${{}_2}\Delta_{W,(2,2,2,2)}^{(2,2,2,2)}$:
\begin{equation*}
\Delta \mathbf{W}^{(1,1,1,1)} = \mathbf{H}^{(1,1,1,1)}\Delta_{W,(1,1,1,1)}^{(1,1,1,1)}, \qquad \Delta^2\, \mathbf{W}^{(2,2,2,2)} = \mathbf{H}^{(2,2,2,2)}\, {{}_2}\Delta_{W,(2,2,2,2)}^{(2,2,2,2)}.
\end{equation*}
We let $\Delta$ and $\Delta^2$ act on the weighted generalised Koornwinder polynomials $\mathbf{W}^{(1,1,1,1)}$ and $\mathbf{W}^{(2,2,2,2)}$, respectively, since these basis functions ensure the imposition of zero Dirichlet and Neumann (in the case of the biharmonic equation) boundary conditions on $\partial\Omega$, which for simplicity are the only types of boundary conditions we shall impose in our numerical experiments.  The following proposition is an immediate consequence of the results in sections~\ref{sec:Kdiffmats} and~\ref{sec:Kconversion}:
\begin{proposition}
The matrix Laplacian 
\begin{align} \label{Laplacian_Operator}
\Delta_{W,(1,1,1,1)}^{(1,1,1,1)} = D_x^{(0,0,0,0)}\,W_x^{(1,1,1,1)} 
+ T_{(0,0,1,1)}^{(1,1,1,1)} \,D_y^{(0,0,0,0)}\,T_{W,(1,1,0,0)}^{(0,0,0,0)}\,W_y^{(1,1,1,1)},
\end{align}
has block-bandwidths $(2\deg\rho^2 - 2, 2\deg\rho^2 - 2)$ and each block has bandwidths $(2,2)$.  The sparse matrix representation of the biharmonic operator is
\begin{align} \label{Biharmonic_Operator}
_2\Delta_{W, (2,2,2,2)}^{(2,2,2,2)}= \Delta_{(0,0,0,0)}^{(2,2,2,2)} \, \Delta_{W, (2,2,2,2)}^{(0,0,0,0)},
\end{align}
where
\begin{eqnarray*}
\Delta_{W, (2,2,2,2)}^{(0,0,0,0)}&=& W_x^{(1,1,1,1)}\, W_x^{(2,2,2,2)} + T_{W, (1,1,0,0)}^{(0,0,0,0)} \, W_y^{(1,1,1,1)} \, T_{W, (2,2,1,1)}^{(1,1,1,1)} \, W_y^{(2,2,2,2)}, \\
   \Delta_{(0,0,0,0)}^{(2,2,2,2)}&=& D_x^{(1,1,1,1)}\, D_x^{(0,0,0,0)} + T_{(1,1,2,2)}^{(2,2,2,2)}\, D_y^{(1,1,1,1)}\, T_{(0,0,1,1)}^{(1,1,1,1)}\, D_y^{(0,0,0,0)},
\end{eqnarray*}
 and $_2\Delta_{W, (2,2,2,2)}^{(2,2,2,2)}$ has block-bandwidths $(4\deg\rho^2 - 4, 4\deg\rho^2 - 4)$ with each block having bandwidths $(4,4)$.
\end{proposition}

We shall also need a sparse matrix representation, denoted by $L_W$, of the following variable-coefficient Helmholtz operator 
\begin{equation*}
\mathcal{L} \mathbf{W}^{(1,1,1,1)} = \mathbf{H}^{(1,1,1,1)} L_W, \qquad \mathcal{L} = \Delta  + k^2 v(x,y), \quad k \in \mathbb{R},
\end{equation*}
where
\begin{equation}
L_W  =	\Delta_{W,(1,1,1,1)}^{(1,1,1,1)} + k^2 \, T_{(0,0,0,0)}^{(1,1,1,1)} \, v(J_x^{(0,0,0,0)}, J_y^{(0,0,0,0)}) \, T_{W,(1,1,1,1)}^{(0,0,0,0)},  \label{eq:varHelmsyst}
\end{equation}
and \( v(J_x^{(0,0,0,0)}, J_y^{(0,0,0,0)}) \) is computed via the bivariate Clenshaw algorithm with matrix inputs~\cite{TriangleOPs}.  In the numerical experiment in Section~\ref{Helmholtz_Example}, $v(x,y) =  1 - \left[3(x-0.2)^2+2(y-0.4)^2\right]$; using Propositions~\ref{M_y} and~\ref{structure_conversion_mat} and (\ref{eq:Tinc1})--(\ref{eq:TWdec1}) to add up the matrix bandwidths, we find that (\ref{eq:varHelmsyst}) has  block-bandwidths $(4\deg\rho^2-2, 4\deg\rho^2-2)$ and each block has bandwidths $(4,4)$.

Figure~\ref{fig:Laplacian_Biharmonic_sparsity} shows examples of the sparse matrix representations of differential operators derived in this section. While many entries are non-zero in these particular truncations, the number of non-zero entries only grows linearly with the truncation size ($\mathcal{O}(N^2)$ non-zero entries for an $\mathcal{O}(N^2) \times \mathcal{O}(N^2)$ truncation) and hence as $N$ becomes large these will be sparse matrices.

\begin{figure}[htbp]
  \centering
   \begin{minipage}[t]{0.31\textwidth}
    \centering
    \makebox[\linewidth][c]{\parbox[c][1.8em][c]{\linewidth}{\centering\scriptsize Laplacian}}
    \vspace{0.3ex}
    \includegraphics[width=\linewidth]{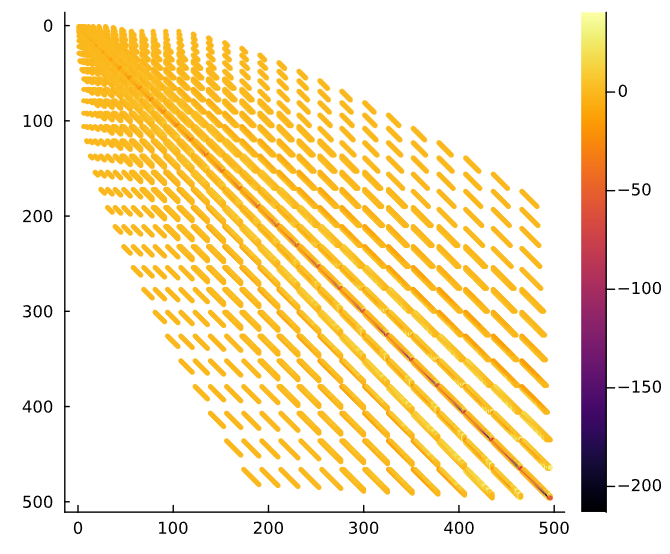}
  \end{minipage}
 \begin{minipage}[t]{0.33\textwidth}
    \centering
    \makebox[\linewidth][c]{\parbox[c][1.8em][c]{\linewidth}{\centering\scriptsize Variable-coefficient Helmholtz}}
    \vspace{0.3ex}
    \includegraphics[width=\linewidth]{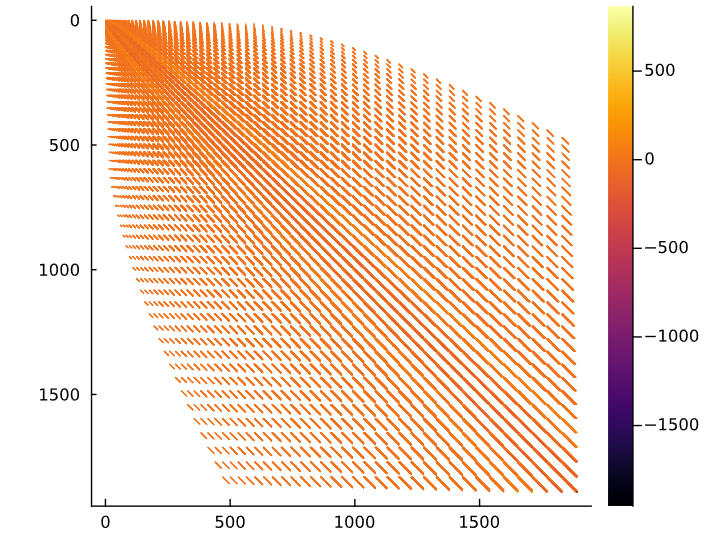}
\end{minipage}
 \begin{minipage}[t]{0.33\textwidth}
    \centering
    \makebox[\linewidth][c]{\parbox[c][1.8em][c]{\linewidth}{\centering\scriptsize Biharmonic}}
    \vspace{0.3ex}
    \includegraphics[width=\linewidth]{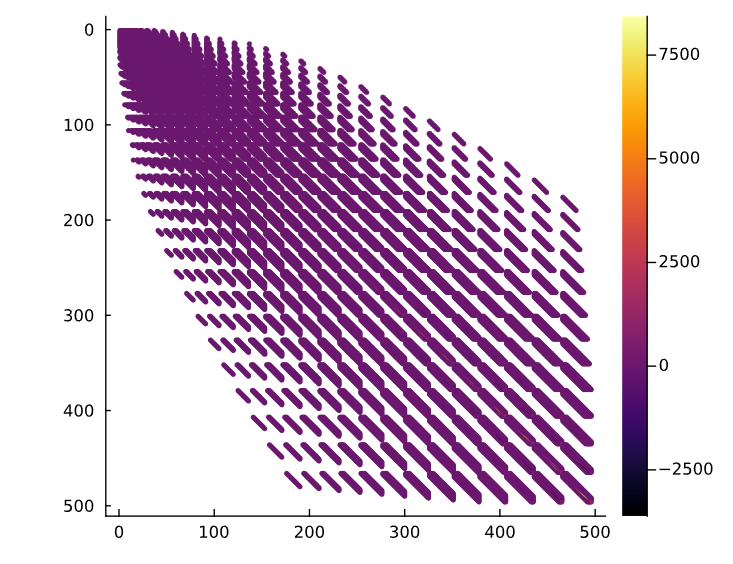}
  \end{minipage}
  \caption{Sparsity pattern of the differential operator matrices.  
   Left: the Laplace operator matrix~(\ref{Laplacian_Operator}) corresponding to the domain described in Section~\ref{Poisson_Example}; it has block-bandwidths $(2\deg\rho^2-2, 2\deg\rho^2-2)=(12,12)$ and each block has bandwidths $(2,2)$. 
  Centre: the variable-coefficient Helmholtz operator matrix~(\ref{eq:varHelmsyst}) corresponding to the domain described in Section~\ref{Helmholtz_Example}; it has block-bandwidths $(4\deg\rho^2-2, 4\deg\rho^2-2)=(30,30)$ and each block has bandwidths $(4,4)$.
  Right: the biharmonic  operator matrix~(\ref{Biharmonic_Operator}) corresponding to the domain described in Section~\ref{Biharmonic_Example}; it has block-bandwidths $(4\deg\rho^2-4, 4\deg\rho^2-4)=(12,12)$ and each block has bandwidths $(4,4)$. }
  \label{fig:Laplacian_Biharmonic_sparsity}
\end{figure}

  	We compute the  multiplication, partial differentiation and conversion matrices for generalised Koornwinder polynomials with optimal (linear) complexity in the number of Koornwinder basis functions. 
	Specifically, the $\frac{(N+1)(N+2)}{2} \times \frac{(N+1)(N+2)}{2}$ principal finite sections of these operator matrices associated with the Koornwinder basis functions $\left\lbrace H_{n,k} : 0 \leq n \leq N,  0 \leq k \leq n \right\rbrace$  are computed in $\mathcal{O}((\deg\rho^2) N^2 )$ complexity and the Laplace, biharmonic and variable-coefficient Helmholtz operator matrices discussed in this section are computed in $\mathcal{O}\left((\deg\rho^2)^2 N^2 \right)$ complexity.

\section{Fast transforms for generalised Koornwinder polynomials}\label{sec:fasttransforms}

 To implement the spectral method in the next section,  we need to approximate the expansion coefficients of a given function $f(x,y)$ in the generalised Koornwinder basis, given by
\begin{equation}
\mathbf{f_H} := \iint_{\Omega}  {\mathbf{H}^{(a,b,c,d)}}^\top f(x,y)  W^{(a,b,c,d)} \mathrm{d}y \,\mathrm{d}x := \Big(f_{0,0} \;\Big|\; f_{1,0} \:\: f_{1,1} \;\Big|\; \cdots \;\Big|\; f_{n,0}\: \cdots\: f_{n,n} \;\Big|\; \cdots \Big)^\top.  \label{eq:fHdef}
\end{equation}
One approach is to approximate these  coefficients via quadrature by extending the method in \cite{SnowballOlver} to generalised Koornwinder polynomials.  However, to compute via quadrature the coefficients in the first $N$ blocks of (\ref{eq:fHdef}), which has $\mathcal{O}(N^2)$ coefficients, has $\mathcal{O}(N^4)$ complexity.  Instead, here we show that by expressing the generalised Koornwinder polynomials in a tensor product Chebyshev basis, we can use fast transforms to map function values on a tensor product grid to approximate expansion coefficients and vice versa in $\mathcal{O}(N^3)$ complexity\footnote{The computational complexity of the map from function values to approximate expansion coefficients and the inverse map can be further reduced to quasi-optimal $\mathcal{O}(N^2 \log N)$ complexity, i.e., almost linear complexity in the number of basis functions,
by performing Givens rotations and using the Butterfly algorithm~\cite{spherical_harmonic}.}.

\subsection{Maps between generalised Koornwinder and tensor product Chebyshev expansion coefficients}

We define the tensor product Chebyshev basis as follows, with $t = y/\rho$,
       \begin{equation*}
\mathbf{T}(x) \otimes \mathbf{T}(t) = \Big( T_0(x) \mathbf{T}(t) \: \Big\vert\: T_1(x) \mathbf{T}(t) \: \Big\vert \: \cdots \Big)
\end{equation*}
where
 $\mathbf{T}(x)$ and $\mathbf{T}(t)$ 
are quasi-matrices of shifted Chebyshev polynomials on the intervals $[\alpha, \beta]$ and $[\gamma, \delta]$, respectively, and $\otimes$ denotes the Kronecker product. 
We have the following relations:
\begin{align}
	\mathbf{S}^{(a,b)} = \mathbf{T}(x)  R_{T,x}, \qquad \mathbf{\widetilde{P}}^{(d,c)} = \mathbf{T}(t) R_{T,t},
\end{align}
where $R_{T,x}$ and $R_{T,t}$ are the upper-triangular conversion matrices as computed via the fast Chebyshev--Jacobi transform \cite{Jac_Cheb}.  

We shall require the following operator matrix.
 \begin{lemma}
The operator matrix $Q_{\rho}^{(a,b,c+d+2k+1)}$ defined by
\begin{align*}
	&\rho(x)^k \, \mathbf{R}^{(a,b,c+d+2k+1)} = \mathbf{S}^{(a,b)} \, Q_{\rho}^{(a,b,c+d+2k+1)},   
 \end{align*}
can be expressed as
\begin{equation}
Q_{\rho}^{(a,b,c+d+2k+1)} = \begin{cases}
\left(\mathcal{C}^{(a,b,c+d+1)}\right)^{-1} \, Q_0 \cdots Q_{k-4} Q_{k-2}, & k \text{ even},\\
{\mathcal{C}^{(a,b,1)}}^\top \mathcal{C}^{(a,b,1)} \left(\mathcal{C}^{(a,b,c+d+3)}\right)^{-1} \, Q_1 \cdots Q_{k-4} Q_{k-2},  & k \text{ odd},
\end{cases}  \label{eq:Qrho}
\end{equation}
where $Q_{\ell}$ denotes the orthogonal matrix obtained from the QR factorisation of the operator matrix $\rho^2\left(   J(w_{R}^{(a,b,c+d+2\ell+1)})\right)$.
\end{lemma}
\begin{proof}
 From~\cite[Lemma 2.8]{Gutleb}, it follows that the raising matrix $\mathcal{T}_{(a,b,c+d+2k - 3)}^{(a,b,c+d+2k+1)}$ is the upper-triangular factor of the QR factorisation of $\rho^2\left(   J(w_{R}^{(a,b,c+d+2k-3)})\right)$, i.e., 
\begin{equation*}
\rho^2\left( J(w_{R}^{(a,b,c+d+2k-3)}) \right) = Q_{k-2}\,\mathcal{T}_{(a,b,c+d+2k - 3)}^{(a,b,c+d+2k+1)}.
\end{equation*}
Hence, for $k$ even, 
\begin{eqnarray*}
\rho^k\, \mathbf{R}^{(a,b,c+d+2k+1)} &=& \rho^k\,\mathbf{R}^{(a,b,c+d+2k-3)}\left(\mathcal{T}_{(a,b,c+d+2k - 3)}^{(a,b,c+d+2k+1)}\right)^{-1} \\
& =& \rho^{k-2}\,\mathbf{R}^{(a,b,c+d+2k-3)}\rho^2\left( J(w_{R}^{(a,b,c+d+2k-3)}) \right)\left(\mathcal{T}_{(a,b,c+d+2k - 3)}^{(a,b,c+d+2k+1)}\right)^{-1} \\
&=& \rho^{k-2}\,\mathbf{R}^{(a,b,c+d+2k-3)}Q_{k-2} \\
&=& \rho^{k-4}\,\mathbf{R}^{(a,b,c+d+2k-7)}Q_{k-4} Q_{k-2} \\
&\vdots &  \\
&=& \mathbf{R}^{(a,b,c+d+1)}Q_{0} \cdots Q_{k-4} Q_{k-2} \\
&=& \mathbf{S}^{(a,b)} \left(\mathcal{C}^{(a,b,c+d+1)}\right)^{-1} Q_{0} \cdots Q_{k-4} Q_{k-2}.
\end{eqnarray*}
For $k$ odd, the expression for $Q_{\rho}^{(a,b,c+d+2k+1)}$ can be derived similarly and by also using the fact that $\rho\, \mathbf{S}^{(a,b)} =\mathbf{S}^{(a,b)}\, {\mathcal{C}^{(a,b,1)}}^\top \mathcal{C}^{(a,b,1)} $, which follows from (\ref{eq:linopmat}) and (\ref{ConnMatSR}). 
\end{proof}

The coefficients in $\mathbf{f_H}$ can be rearranged into the following matrix:
\[
  \begin{pmatrix}
f_{0,0} & f_{1,1} & f_{2,2} & \cdots & f_{k,k} & \cdots\\
f_{1,0} & f_{2,1} & f_{3,2} & \cdots &  f_{k+1,k} & \cdots\\
 \vdots  & \vdots  &  \vdots & \ddots & \vdots & \ddots \\
f_{k,0} & f_{k+1,1} & f_{k+2,2} & \cdots&  f_{2k,k} & \cdots\\
\vdots  & \vdots  & \vdots  & \ddots & \vdots & \ddots 
\end{pmatrix} = \Big( \mathbf{f}_0 \; \; \mathbf{f}_1 \;  \cdots  \; \mathbf{f}_k \; \;\cdots \Big).
\]

Since $\widetilde{P}_k(t) = \widetilde{\mathbf{P}}(t)\mathbf{e}_k = \mathbf{T}(t)R_{T,t}\,\mathbf{e}_k$, where $\mathbf{e}_k$ is the infinite column vector with $1$ as its $(k+1)$-st entry and zeros elsewhere,
 \begin{eqnarray*}
\widetilde{P}_k \,\mathbf{T}(x) &=&  \Big( \widetilde{P}_k(t) T_0(x) \: \Big\vert\: \widetilde{P}_k(t) T_1(x) \: \Big\vert \: \cdots \Big) \\
&=& \Big(  T_0(x)\mathbf{T}(t)R_{T,t}\,\mathbf{e}_k \: \Big\vert\:  T_1(x)\mathbf{T}(t)R_{T,t}\,\mathbf{e}_k \: \Big\vert \: \cdots \Big) \\
&=& \Big(  \mathbf{T}(x) \otimes \mathbf{T}(t) \Big)\Big(  I \otimes R_{T,t}\,\mathbf{e}_k  \Big).
  \end{eqnarray*}

 The following shows how the expansion coefficients of a function in the generalised Koornwinder  basis can be used to obtain the tensor product Chebyshev coefficients,
\begin{eqnarray}
f(x,y)&=& {\mathbf{H}^{(a,b,c,d)}} \mathbf{f_H}  = \sum_{k = 0}^{\infty}\sum_{n = k}^{\infty} f_{n,k} H_{n,k}^{(a,b,c,d)} = \sum_{k = 0}^{\infty}\sum_{n = k}^{\infty} f_{n,k}R_{n-k}^{(a,b,c+d+2k+1)}\rho^k\widetilde{P}_k \nonumber \\
& =& \sum_{k = 0}^{\infty}\sum_{m = 0}^{\infty} \widetilde{P}_k \rho^k\, f_{m+k,k}R_{m}^{(a,b,c+d+2k+1)} = \sum_{k = 0}^{\infty} \widetilde{P}_k \rho^k\,\mathbf{R}^{(a,b,c+d+2k+1)}\mathbf{f_k}\nonumber \\
&=& \sum_{k = 0}^{\infty} \widetilde{P}_k \,\mathbf{T}(x)R_{T,x} \, Q_{\rho}^{(a,b,c+d+2k+1)}\mathbf{f_k} \nonumber \\
&=& \Big(  \mathbf{T}(x) \otimes \mathbf{T}(t) \Big) \underbrace{\sum_{k = 0}^{\infty} \Big(  I \otimes R_{T,t}\,\mathbf{e}_k  \Big) R_{T,x}\,Q_{\rho}^{(a,b,c+d+2k+1)}\mathbf{f_k}}_{=\mathbf{f}_{\mathbf{T}^2}} \label{Kronecker1} \\
&=& \Big(  \mathbf{T}(x) \otimes \mathbf{T}(t) \Big)\,\mathbf{f}_{\mathbf{T}^2}. \nonumber
\end{eqnarray}
Conversely, given $\mathbf{f}_{\mathbf{T}^2}$, the expansion coefficients in the tensor product Chebyshev basis, the Koornwinder coefficients can be obtained via
             \begin{align}
	\mathbf{f}_k = { Q_{\rho}^{(a,b,c+d+2k+1)} }^{-1}   R_{T,x}^{-1} \Big( I \otimes {\mathbf{e}^\top_k} {R_{T,t}^{-1}} \Big) \mathbf{f}_{\mathbf{T}^2}, \qquad k = 0, 1, \ldots. \label{Kronecker2}
\end{align}

The finite-dimensional versions of (\ref{Kronecker1}) and (\ref{Kronecker2}) are realised as follows. We let the Koornwinder coefficients $f_{n,k}$, $0 \leq n \leq N, 0 \leq k \leq n$ be the entries of the column vectors $\mathbf{f}_k \in \mathbb{R}^{N+1}$, for $k = 0, \ldots, N$, where coefficients outside of the index set $\lbrace (n, k) : 0 \leq n \leq N, 0 \leq k \leq n  \rbrace$ are set equal to zero.  Then applying $Q_{\rho}^{(a,b,c+d+2k+1)}$, as defined in (\ref{eq:Qrho}), and its inverse to vectors of length $N+1$ (via Householder reflectors to multiply by the orthogonal matrices and LU factorisation to invert the connection matrices) can be accomplished in $\mathcal{O}((\deg\rho^2)^2 N^2)$ complexity in case(i) and $\mathcal{O}(((\deg\rho^2)^2 +m) N^2)$ complexity in case (ii).  Each summand in (\ref{Kronecker1}) takes the form
\begin{equation*}
\left(I \otimes \mathbf{v}\right) \mathbf{y} = \left(
\begin{array}{c}
y_0 \mathbf{v} \\
\hline 
y_1 \mathbf{v} \\
\hline 
\vdots \\
\hline
y_N \mathbf{v}
\end{array}
\right),
\end{equation*}
where $\mathbf{v, y} \in \mathbb{R}^{N+1}$ and $y_0, \ldots, y_N$ are the entries of $\mathbf{y}$ and thus $\left(I \otimes \mathbf{v}\right) \mathbf{y}$ and  $\mathbf{f}_{\mathbf{T}^2} \in \mathbb{R}^{(N+1)^2}$.  The finite-dimensional version of (\ref{Kronecker2}) requires computing products of the form $\left( I \otimes \mathbf{w}^{\top}\right) \mathbf{z} = \left( \mathbf{w}^{\top} \mathbf{z}_0, \: \mathbf{w}^{\top} \mathbf{z}_1,\: \cdots, \: \mathbf{w}^{\top} \mathbf{z}_N \right)^{\top}$, where $\mathbf{w} \in \mathbb{R}^{N+1}$ and $\mathbf{z}_0, \ldots, \mathbf{z}_N \in \mathbb{R}^{N+1}$ are the (block) entries of $\mathbf{z} \in \mathbb{R}^{(N+1)^2}$, hence $\left( I \otimes \mathbf{w}^{\top} \right) \mathbf{z}$ and  $\mathbf{f}_k \in \mathbb{R}^{N+1}$.  Hence, computing the finite-dimensional versions of (\ref{Kronecker1}) (with $N+1$ terms in the summation) and (\ref{Kronecker2}) (for $\mathbf{f}_k$, $k = 0, \ldots, N$) has $\mathcal{O}\left(N^3\right)$ complexity.

Recall from Section~\ref{sec:Laplace} that to impose zero Dirichlet and, in the case of the biharmonic equation, Neumann boundary conditions, the solution is expanded in weighted bases $\mathbf{W}^{(a,b,c,d)} = W^{(a,b,c,d)}\mathbf{H}^{(a,b,c,d)}$, with $W^{(a,b,c,d)}$ defined in (\ref{eq:Wdef}).  Similar to the derivation of (\ref{Kronecker1}), it can be shown that for a function $f$ with an expansion in the weighted basis $f(x,y) =  \mathbf{W}^{(a,b,c,d)} \mathbf{f_{H}}$, its tensor product Chebyshev coefficients are given by 
\begin{equation}
	\mathbf{f}_{\mathbf{T}^2} = \sum_{k=0}^{\infty}  \left[ I \otimes \left(R_{T,t}\, w_P^{(d,c)}\left(J(w_P^{(d,c)})\right) \mathbf{e}_k\right) \right] R_{T,x}\, w_R^{(a,b,c+d)}\left(J(w_S^{(a,b)}) \right) Q_{\rho}^{(a,b,c+d+2k+1)} \, \mathbf{f}_k,  \label{eq:kron3}
\end{equation}
where $w_P^{(d,c)}$ and $w_R^{(a,b,c+d)}$ are defined in (\ref{eq:wRwPdef}).  Similar to the case for (\ref{Kronecker1}), the finite-dimensional version of (\ref{eq:kron3}) for computing $(N+1)^2$ tensor product Chebyshev coefficients has $\mathcal{O}\left(N^3 \right)$ complexity.

\subsection{Maps between tensor product Chebyshev coefficients and function values on a tensor product grid}

The function values $f(x_m, t_n)$, $0 \leq m, n \leq N$  in the $(x, t)$-plane on the tensor product Chebyshev grid $(x_m, t_n)$, where
\begin{equation*}
x_m = \frac{\beta + \alpha}{2} + \frac{\beta - \alpha}{2} \cos\left(\frac{(2m+1)}{(N+1)}\frac{\pi}{2}  \right), \qquad
t_n = \frac{ \delta + \gamma }{2} + \frac{\delta - \gamma}{2} \cos\left(\frac{(2n+1)}{(N+1)}\frac{\pi}{2}  \right),
\end{equation*}
which correspond in the $(x,y)$-plane to the function values $f(x_m,t_n\,\rho(x_m))$, can be mapped to $(N+1)^2$  entries of the vector of approximate tensor product Chebyshev coefficients $\mathbf{f_{T^2}}$ in $\mathcal{O}\left(N^2\log N\right)$ complexity via the discrete cosine transform (DCT) and the inverse map can be computed using the inverse discrete cosine transform (iDCT).   
          Figure~\ref{fig:sample_grid} shows the mapped tensor product Chebyshev grids \((x_m, t_n\,\rho(x_m))\) in the $(x,y)$-plane  with \(N=50\) on the domains shown in Figure~\ref{fig:approx_solutions_Poisson_Helmholtz}.

\begin{figure}[h]
  \centering
  
   \begin{minipage}[b]{0.4\textwidth}
    \centering
    \includegraphics[width=\linewidth]{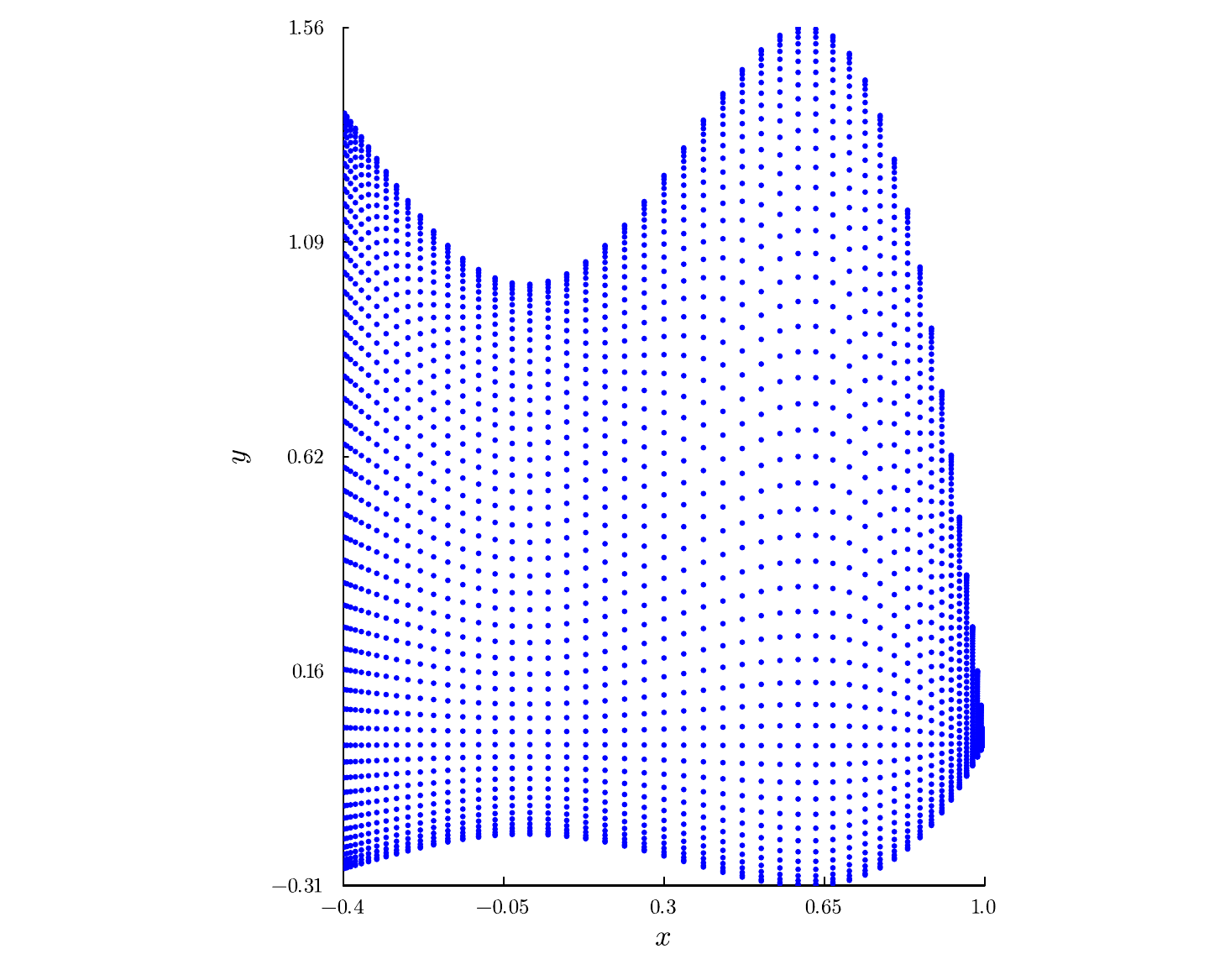}
  \end{minipage}
   \begin{minipage}[b]{0.5\textwidth}
    \centering
	\raisebox{-1cm}{
     \includegraphics[width=\linewidth]{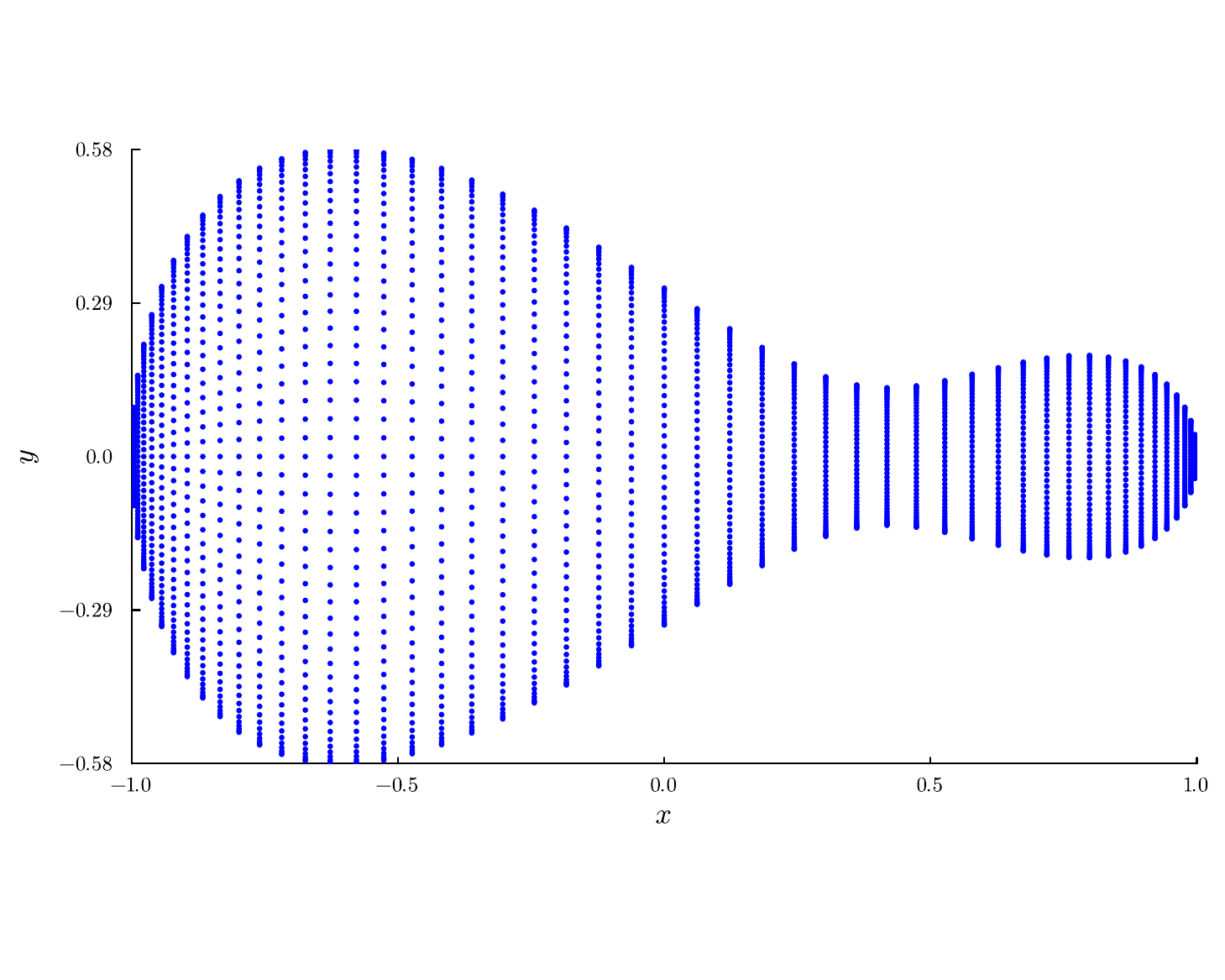} }
  \end{minipage}
  \caption{Tensor product Chebyshev grids \((x_m, t_n\,\rho(x_m))\) with \(N=50\) on the domains shown in Figure~\ref{fig:approx_solutions_Poisson_Helmholtz}.}
  \label{fig:sample_grid}
\end{figure}

To summarise, in the spectral method in the next section, we shall map between function values and Koornwinder coefficients and as follows:
\begin{equation}
\left\lbrace f(x_m,t_n\,\rho(x_m)) :  0 \leq m, n \leq N  \right\rbrace \quad  \xLeftrightarrow[\text{iDCT}]{\text{DCT}}\quad \mathbf{f_{T^2}} \quad  \xLeftrightarrow[\text{(\ref{eq:kron3})}]{\text{(\ref{Kronecker2})}}\quad \mathbf{f_{H}}.  \label{eq:maps}
\end{equation}

\section{Sparse spectral method for solving PDEs}\label{Koornwinder_Spectral_Methods}
           We now have all the building blocks for solving PDEs on generalised Koornwinder domains
 with zero Dirichlet conditions:
\begin{align*}
    \mathcal{L} u(x,y) = f(x,y) \:\:\text{in} \:\:\Omega, \quad  u(x,y) = 0 \:\:\text{on} \:\:\partial \Omega,
\end{align*}
where $\mathcal{L}$ is a linear differential operator,  $\Omega$ is the generalised Koornwinder domain defined in (\ref{eq:Omega})  
 and the solution is expanded in a weighted basis, defined in (\ref{eq:WKoorn}),
\begin{equation}
u(x,y) = \mathbf{W}^{(a,b,c,d)}(x,y)\mathbf{u}, \qquad \mathbf{u} = \left( \begin{array}{c}
\mathbf{u}_0 \\
\hline 
\mathbf{u}_1 \\
\hline
\mathbf{u}_2 \\
\hline
\vdots
\end{array}
\right), \qquad \mathbf{u}_n =  \left( \begin{array}{c}
{u}_{n,0} \\
{u}_{n,1} \\
\vdots \\
{u}_{n,n} 
\end{array}
\right) \in \mathbb{R}^{n+1}.  \label{eq:Kcoeffsblocks}
\end{equation} 
The  steps of the sparse spectral method for approximating the solution $u$  
 are provided in Algorithm~\ref{alg:Spectral methods}.
We consider the Poisson equation and a variant (the screened Poisson equation), an inhomogeneous variable-coefficient Helmholtz equation and the biharmonic equation, thereby demonstrating the versatility of the proposed approach.

\begin{algorithm}
\caption{Sparse spectral method with generalised Koornwinder polynomial basis functions $H^{(a,b,c,d)}_{n,k}$, for $0 \leq n \leq N$, $0 \leq k \leq n$}
\label{alg:Spectral methods}
\begin{algorithmic}[1]
\State \textbf{Fast transforms.} Obtain the Koornwinder coefficients of $f$, $\mathbf{f_H}$, using the transforms in (\ref{eq:maps}).  This step has complexity  $\mathcal{O}(N^3)$ and can be theoretically reduced to $\mathcal{O}(N^2 \log(N))$. 
 \State \textbf{Discretisation.}  Construct the sparse matrix representation $L_{W}$ of the differential operator $\mathcal{L}$ applied to $u(x,y)$ in $\mathcal{O}\left( (\deg\rho^2)^2 N^2\right)$ complexity.  

\State \textbf{Solve.} Perform LU factorisation combined with nested dissection to solve the sparse $M\times M$ linear system $L_{W} \,\mathbf{u} = \mathbf{f_H}$, where $M = (N+1)(N+2)/2$, in $\mathcal{O}\left( (\deg\rho^2)^3 N^3\right)$ complexity
 (using the UMFPACK library via \textsc{Julia}).

\State \textbf{Fast inverse transforms.} Map the solution's Koornwinder coefficients $\mathbf{u}$ to solution values $u(x_m, t_n \rho(x_m))$ using the transforms in (\ref{eq:maps}).    
The computational complexity is $\mathcal{O}(N^3)$ and can be theoretically reduced to $\mathcal{O}(N^2 \log(N))$.
     \end{algorithmic}
\end{algorithm}

\subsection{Poisson equation} \label{Poisson_Example_smoothdomain}
We consider the Poisson problem
\begin{align*}
    \Delta  u(x,y) = f(x,y) \:\:\text{in} \:\:\Omega, \quad  u(x,y) = 0 \:\:\text{on} \:\:\partial \Omega,
\end{align*}
where $\Omega$ is the non-convex smooth domain shown in the right frame of Figure~\ref{fig:approx_solutions_Poisson_Helmholtz}, defined by (\ref{eq:Omega}) with \( (\alpha, \beta, \gamma, \delta) = (-1, 1, -1, 1) \) and \(\rho(x) = \sqrt{(1-x^2)(0.1 - 0.4x + 0.5x^2)} \), hence $\Omega$ is the region $y^2 \leq \rho^2$.  As shown in Appendix~\ref{Generalised_Koornwinder_polynomials_smooth_domains}, on smooth domains, we use the one-parameter family of generalised Koornwinder polynomials $\mathbf{H}^{(a)}(x,y):= \mathbf{H}^{(0,0,a,a)}(x,y)$ and weighted polynomials $\mathbf{W}^{(a)}(x,y):= \mathbf{W}^{(0,0,a,a)}(x,y)$, where $W^{(a)}(x,y) = (\rho^2 - y^2)^a$.  We expand the solution in the weighted basis $\mathbf{W}^{(1)}$ to satisfy the zero Dirichlet boundary condition, hence setting $u = \mathbf{W}^{(1)}\mathbf{u}$, we have from (\ref{Laplacian_Operator_degenerate}) that $\Delta u = \Delta \mathbf{W}^{(1)}\mathbf{u} = \mathbf{H}^{(1)}\Delta_{W,(1)}^{(1)} \mathbf{u}$, where $\Delta_{W,(1)}^{(1)}$ is a sparse matrix Laplacian with block-bandwidths $(4,4)$ and each block has bandwiths $(2, 2)$ (see Proposition~\ref{prop:smoothoperator}). Since $\Delta u$ is expressed in the $\mathbf{H}^{(1)}$-basis, we express $f(x,y)$ in the same basis as  $f = \mathbf{H}^{(1)} \mathbf{f_H}$ and hence the Poisson problem is equivalent to the infinite linear system
\[
	\Delta_{W,(1)}^{(1)} \,\mathbf{u} = \mathbf{f_H},
\]
whose solution we approximate  by following the steps in Algorithm~\ref{alg:Spectral methods}.

To test the spectral method, and as an indication of its accuracy, we consider two Poisson problems with the following exact solutions: $u_1 := W^{(1)}(x,y) \,e^{-100(x^2+y^2)}$ and $u_2:= u_1\cos\left(120(x+y)^2\right)$, for which the right-hand side (RHS) functions are $f(x,y) := \Delta u_1$ and $\Delta u_2$.
   To estimate the error, we generate a uniform $750\times 750$ grid on a rectangle containing $\Omega$ and then compute the error at the grid points inside $\Omega$ for the approximate solutions obtained via Algorithm~\ref{alg:Spectral methods}.  On this grid, the $\infty$-norm absolute and relative errors for $u_1$ are, respectively, $\approx 1.4\cdot 10^{-12}$ and $1.4\cdot10^{-11}$ and for $u_2$, they are $\approx 1.8\cdot 10^{-11}$ and $1.8\cdot10^{-10}$. 
  
In Figure~\ref{fig:Poisson_BowlingPin_solution}, we plot the norms $\| \mathbf{u}_n \|_{\infty}$ of the generalised Koornwinder coefficients (see (\ref{eq:Kcoeffsblocks})) of $u_1$ (orange), $u_2$ (purple) and solutions to Poisson problems for which we do not have closed-form solutions for $0 \leq n \leq N$ with $N = 1000$, which are obtained by solving sparse $M \times M$ linear systems with $M = (N+1)(N+2)/2 = 501,501$. 

Since the solutions $u_1$ and $u_2$ are rapidly decaying entire functions that vanish to all orders below the level of machine precision on $\partial \Omega$ (the boundary $\partial\Omega$ is essentially invisible to these solutions), we observe super-exponential decay of their coefficients in Figure~\ref{fig:Poisson_BowlingPin_solution}.  We find that the coefficients of the other solutions in Figure~\ref{fig:Poisson_BowlingPin_solution} decay exponentially at roughly the rate\footnote{For Poisson problems on smooth domains with entire RHS functions, we have found that the rate of exponential decay of the solution coefficients is determined by $\partial\Omega$, however the precise relationship between the rate of decay and $\partial\Omega$ is not known to us.} $\mathcal{O}(e^{-0.13n})$.  Figure~\ref{fig:Poisson_BowlingPin_solution} shows that for RHS functions that are small or vanish on $\partial\Omega$ (red, orange, purple), the decay of the coefficients plateau at levels below machine precision (i.e., roughly $10^{-16}$), while for RHS functions that do not vanish on $\Omega$ (green, blue), the coefficients do not plateau.  Presumably this is due to rounding errors.

\begin{figure}[h]
  \centering
\makebox[\linewidth][c]{\parbox[c][1.8em][c]{\linewidth}{\centering\footnotesize Decay of expansion coefficients of solutions to the Poisson equation}}
               \includegraphics[width=0.85\linewidth]{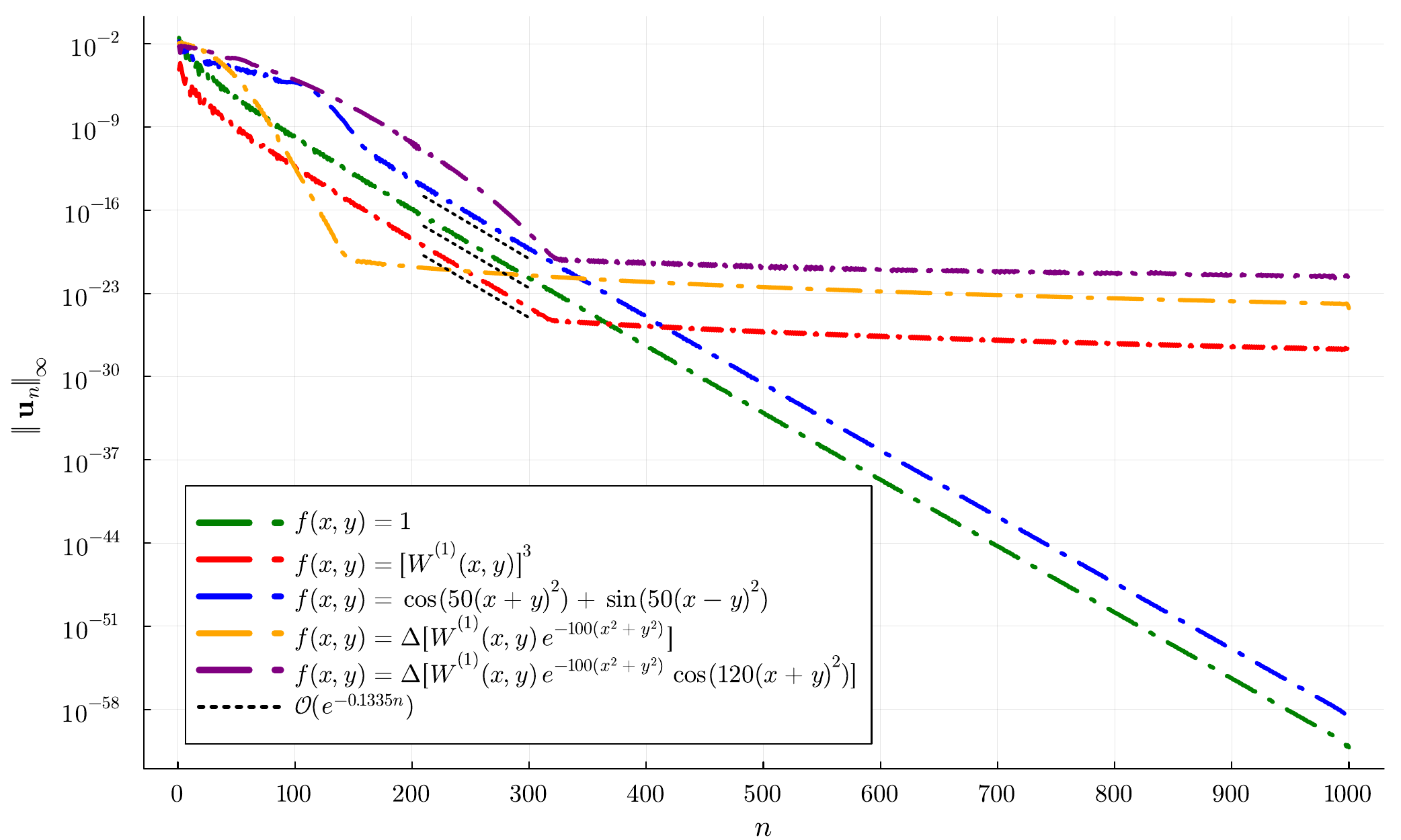}
     \caption{ 
 	The norms of each block of calculated coefficients of the approximate solution to the Poisson equation in section \ref{Poisson_Example_smoothdomain} for various RHS functions.  
 	}
  \label{fig:Poisson_BowlingPin_solution}
   \end{figure}

\subsection{Screened Poisson equation} \label{Poisson_Example}
We consider the screened Poisson equation
\begin{align*}
   \left( \Delta  - k^2\right) u(x,y) = f(x,y) \:\:\text{in} \:\:\Omega, \quad  u(x,y) = 0 \:\:\text{on} \:\:\partial \Omega,
\end{align*}
where $k=100$, $\Omega$ is the piecewise smooth region shown on the left in Figure~\ref{fig:Screened_Poisson_fish_solution} and is defined by (\ref{eq:Omega}) with  \( (\alpha, \beta, \gamma, \delta) = (0, 1, -0.3, 0.3) \) and \(\rho(x) = \sqrt{(0.5 + x)(1 + 6x^2 - 20x^4 + 15x^6)}\).  
 Expanding the solution in the weighted basis, $u = \mathbf{W}^{(1,1,1,1)}\mathbf{u}$, the operator $\mathcal{L} = \Delta - k^2$   applied to $u$ is expressed in the $\mathbf{H}^{(1,1,1,1)}$-basis, $\mathcal{L}u = \mathcal{L} \mathbf{W}^{(1,1,1,1)}\mathbf{u} = \mathbf{H}^{(1,1,1,1)}L_W \mathbf{u}$, hence $f$ is expanded as $f = \mathbf{H}^{(1,1,1,1)} \mathbf{f_H}$ and thus we solve $L_W \mathbf{u} = \mathbf{f_{H}}$, where
\begin{equation}
L_W \mathbf{u} :=	\left(\Delta_{W,(1,1,1,1)}^{(1,1,1,1)} - k^2 \, T_{(0,0,0,0)}^{(1,1,1,1)}  \, T_{W,(1,1,1,1)}^{(0,0,0,0)}\right) \,\mathbf{u} = \mathbf{f_H},  \label{eq:ScreenedPossionsyst}
\end{equation}
and the operator matrices on the left are defined in (\ref{Laplacian_Operator}), (\ref{eq:Tinc1}) and (\ref{eq:TWdec1}).
The left frame of Figure~\ref{fig:Laplacian_Biharmonic_sparsity} shows the sparse structure of the Laplace operator matrix $\Delta_{W,(1,1,1,1)}^{(1,1,1,1)}$ for this problem,
and the computed solution corresponding to the RHS function $f(x,y) = W^{(1,1,1,1)}(x,y)\, e^x$ is plotted in the left of Figure~\ref{fig:Screened_Poisson_fish_solution}.
           
As in the previous section, we estimate the error on a uniform grid for two known solutions to the screened Poisson equation, namely $u=v_1 := W^{(1,1,1,1)}(x,y) \,e^{-100(x^2+y^2)}$ and $u=v_2 := v_1\cos\left(300(x+y)^2  \right)$.  We find that the infinity-norm absolute and relative errors on the grid for $v_1$ are $\approx 1.7\cdot 10^{-17}$ and $8.2\cdot 10^{-15}$, respectively and for $v_2$, they are $\approx 6.6\cdot 10^{-15}$ and $3.3\cdot 10^{-12}$.

Due to the presence of corner singularities~\cite{grisvard2011elliptic},  in the right frame of Figure~\ref{fig:Screened_Poisson_fish_solution}  the coefficients of the solutions to the screened Poisson equation decay algebraically, at a rate which we determined numerically to be approximately $\mathcal{O}(n^{-4.5})$ as $n \to \infty$ (see the dotted line in Figure~\ref{fig:Screened_Poisson_fish_solution}).   

\begin{figure}[h]
  \centering
  
   \begin{minipage}[b]{0.28\textwidth}
    \centering
    \includegraphics[width=\linewidth]{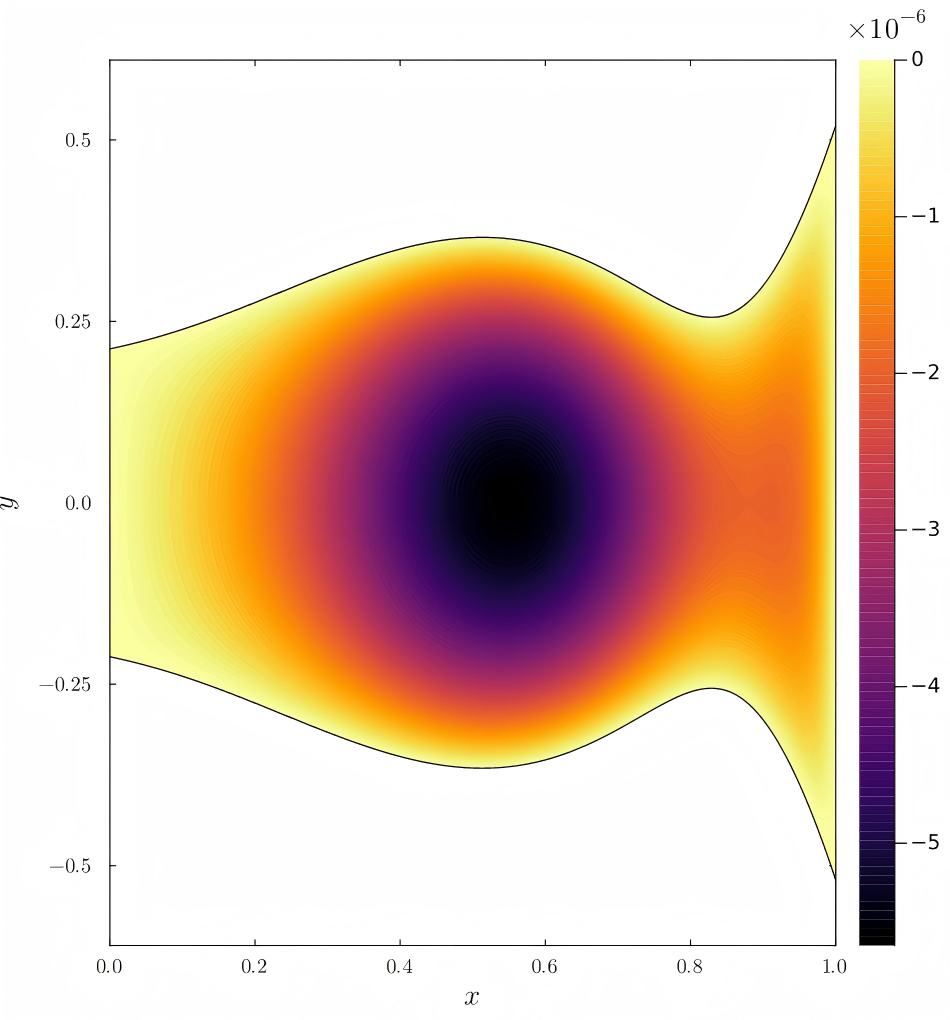}
  \end{minipage}
  \hspace{0.005\textwidth}
  \begin{minipage}[b]{0.7\textwidth}
    \centering
	\makebox[\linewidth][c]{\parbox[c][1.8em][c]{\linewidth}{\centering\scriptsize Decay of expansion coefficients of solutions to the screened Poisson equation}}
    \includegraphics[width=\linewidth]{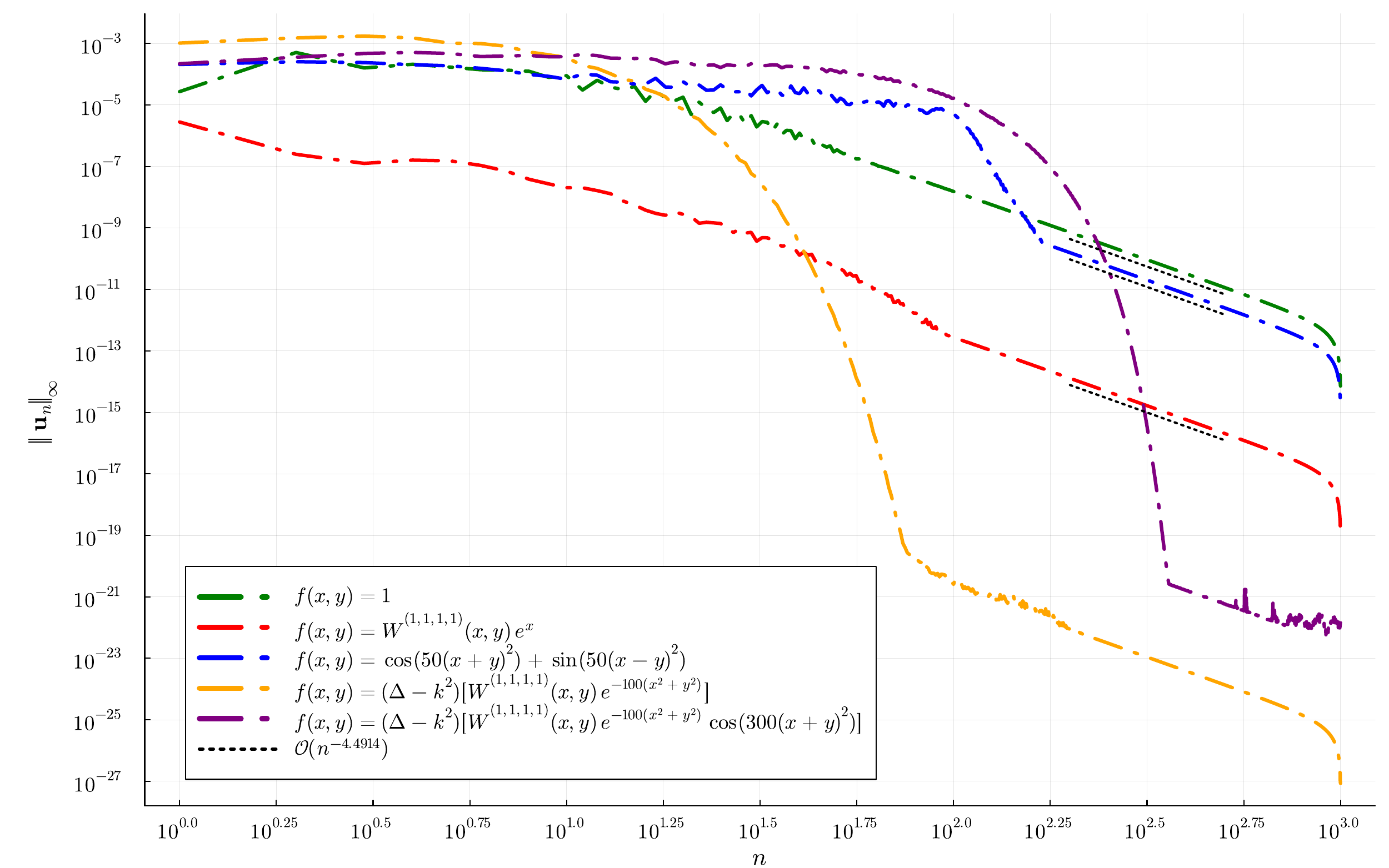}
  \end{minipage}

  \caption{Left: The approximate solution to \((\Delta - k^2)u = f\), subject to zero Dirichlet conditions, with \(f(x,y) = W^{(1,1,1,1)}(x,y)\, e^x\). Right: The norms of each block of calculated coefficients of the approximate solution to the screened Poisson equation for various RHS functions.  Note that Figure~\ref{fig:Poisson_BowlingPin_solution} is on a log-linear set of axes whereas this plot is on a log-log scale. }  
  \label{fig:Screened_Poisson_fish_solution}
\end{figure}

\subsection{Inhomogeneous variable-coefficient Helmholtz equation} \label{Helmholtz_Example}
Now we consider a variable-coefficient Helmholtz equation with zero Dirichlet conditions, i.e.,
\begin{align}
   \left( \Delta   + k^2 v(x,y)\right) u(x,y)= f(x,y) \:\:\text{in} \:\:\Omega, \quad  u(x,y) = 0 \:\:\text{on} \:\:\partial \Omega, \label{eq:vHelm}
\end{align}
where \(k=100\), $v(x,y) =  1 - \left[3(x-0.2)^2+2(y-0.4)^2\right]$, $\Omega$ is shown in the left frame of Figure~\ref{fig:approx_solutions_Poisson_Helmholtz}  and is defined by (\ref{eq:Omega}) with \( (\alpha, \beta, \gamma, \delta) = (-0.4, 1, -0.2, 1) \) and \(\rho(x) = 1+3x^2-4x^4\).  Expanding the solution in the weighted basis, $u = \mathbf{W}^{(1,1,1,1)}\mathbf{u}$, the operator $\mathcal{L} = \Delta + k^2v(x,y)$   applied to $u$ is expressed in the $\mathbf{H}^{(1,1,1,1)}$-basis, $\mathcal{L}u = \mathcal{L} \mathbf{W}^{(1,1,1,1)}\mathbf{u} = \mathbf{H}^{(1,1,1,1)}L_W \mathbf{u}$, hence $f$ is expanded as $f = \mathbf{H}^{(1,1,1,1)} \mathbf{f_H}$ and thus we solve $L_W \mathbf{u} = \mathbf{f_{H}}$, where
   and the operator matrices on the left are defined in (\ref{Laplacian_Operator}), (\ref{eq:Tinc1}), (\ref{eq:TWdec1}) and \( v(J_x^{(0,0,0,0)}, J_y^{(0,0,0,0)}) \) can be obtained via the bivariate Clenshaw algorithm with matrix inputs~\cite{TriangleOPs}.  A finite section of the operator matrix $L_W$ is shown in the middle frame of Figure~\ref{fig:Laplacian_Biharmonic_sparsity}, 
and the computed solution corresponding to the RHS function $f(x,y) = W^{(1,1,1,1)}(x,y)\,e^{x}$ is plotted on the left in Figure~\ref{fig:approx_solutions_Poisson_Helmholtz}.

As in the previous sections, we estimate the error on a uniform grid for two known solutions to the (\ref{eq:vHelm}), namely $u=w_1 := W^{(1,1,1,1)}(x,y) \,e^{-100(x^2+y^2)}$ and $u=w_2 := w_1\cos\left(100(x+y)^2  \right)$.
  We find that the infinity-norm absolute and relative errors on the grid for $w_1$ are $\approx 1.1\cdot 10^{-13}$ and $1.4\cdot 10^{-12}$, respectively and for $w_2$, they are $\approx 2.7\cdot 10^{-12}$ and $3.2\cdot 10^{-11}$.  Figure~\ref{fig:Helmholtz_solution} shows the norms of the coefficients of $w_1$ (black), $w_2$ (purple) and for other solutions for which we do not have exact solutions available.  

\begin{figure}[h]
  \centering
  \makebox[\linewidth][c]{\parbox[c][1.8em][c]{\linewidth}{\centering\scriptsize Decay of expansion coefficients of solutions to the variable-coefficient Helmholtz equation}}
       \includegraphics[width=0.85\linewidth]{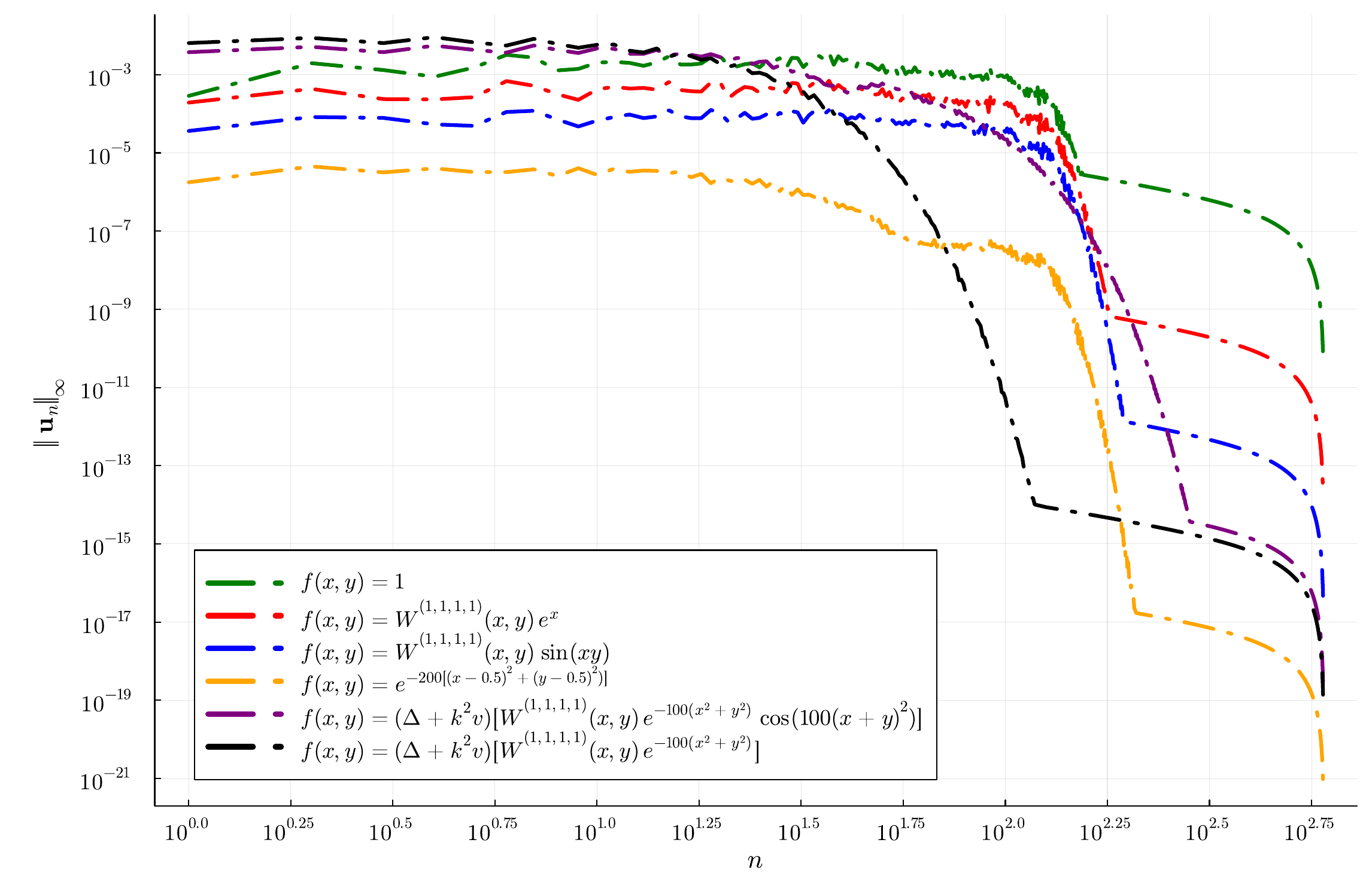}
             \caption{ 
 	The norms of each block of calculated coefficients of the approximate solution to the variable-coefficient Helmholtz equation in section \ref{Helmholtz_Example} for various RHS functions.}
  \label{fig:Helmholtz_solution}
\end{figure}

\subsection{Biharmonic equation} \label{Biharmonic_Example}
We consider the biharmonic equation with zero Dirichlet and Neumann conditions,
\begin{align*}
      \Delta^2 u = f \:\:\text{in} \:\:\Omega, \qquad  u = 0, {\partial u  \over \partial n}=0\:\:\text{on} \:\:\partial \Omega,
\end{align*}
where $\Omega$ is shown in the left frame of Figure~\ref{fig:biharmonic_EX} and is defined by (\ref{eq:Omega}) with \( (\alpha, \beta, \gamma, \delta) = (0, 0.8, 0, 1) \) and \(\rho(x) = 0.6x^2 + 0.2\).  Setting $u = \mathbf{W}^{(2,2,2,2)}\mathbf{u}$, we have from  (\ref{Biharmonic_Operator}) that $\Delta^2 u = \Delta^2 \mathbf{W}^{(2,2,2,2)}\mathbf{u} =  \mathbf{H}^{(2,2,2,2)}{} _2\Delta_{W, (2,2,2,2)}^{(2,2,2,2)} \mathbf{u}$.  Setting $f = \mathbf{H}^{(2,2,2,2)}\mathbf{f_{H}}$, we need to solve 
\begin{equation}
	_2\Delta_{W, (2,2,2,2)}^{(2,2,2,2)} \,\mathbf{u} = \mathbf{f_H}.  \label{eq:biharmsyst}
\end{equation}
A finite section of the operator matrix $_2\Delta_{W, (2,2,2,2)}^{(2,2,2,2)}$ is shown in the right frame of Figure~\ref{fig:Laplacian_Biharmonic_sparsity}. 
Following the steps of Algorithm~\ref{alg:Spectral methods}, we obtain the solution corresponding to the RHS function $f(x,y) = 10^4 \sin(30\pi x) \cos(30\pi y)$ on the left in Figure~\ref{fig:biharmonic_EX}.
            
As in the previous sections, we estimate the error on a uniform grid for two known solutions to the biharmonic equation, namely $u=s_1 := W^{(2,2,2,2)}(x,y) \,e^{-100(x^2+y^2)}$ and $u=s_2 := s_1\cos\left(300(x+y)^2  \right)$.  We find that the infinity-norm absolute and relative errors on the grid for $s_1$ are $\approx 4.3\cdot 10^{-21}$ and $4.3\cdot 10^{-14}$, respectively and for $s_2$, they are $\approx 1.8\cdot 10^{-19}$ and $1.8\cdot 10^{-12}$.  Figure~\ref{fig:biharmonic_EX} shows the norms of the coefficients of $s_1$ (black), $s_2$ (purple) and for other solutions for which we do not have exact solutions available.  

\begin{figure}[h]
  \centering
  
   \begin{minipage}[b]{0.28\textwidth}
    \centering
    \includegraphics[width=\textwidth]{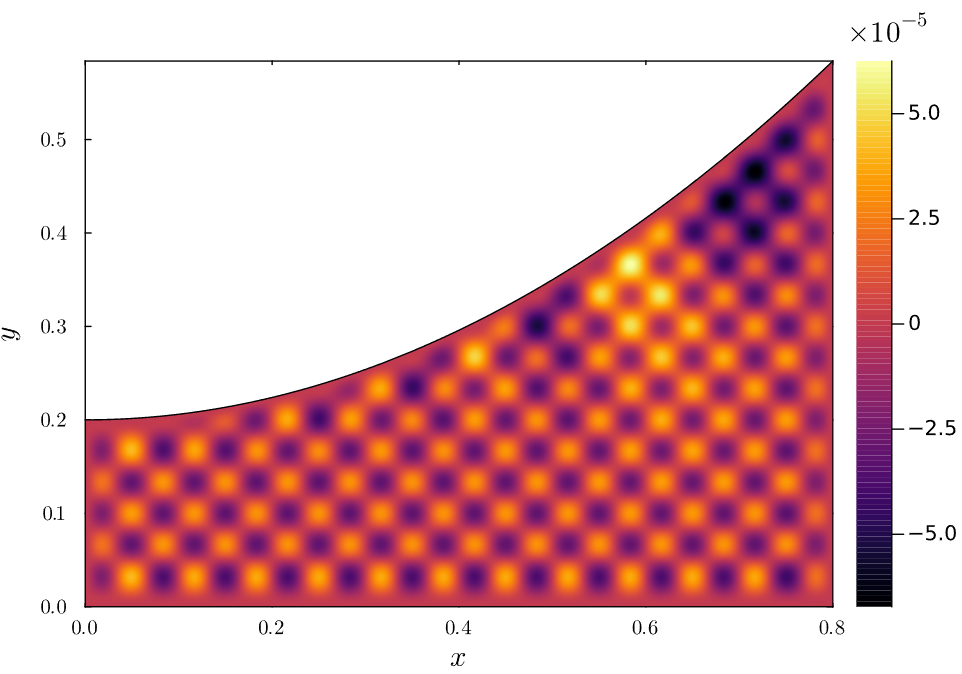}
  \end{minipage}
  \hspace{0.005\textwidth}
  \begin{minipage}[b]{0.7\textwidth}
    \centering
	\makebox[\linewidth][c]{\parbox[c][1.8em][c]{\linewidth}{\centering\scriptsize Decay of expansion coefficients of solutions to the biharmonic equation}}
    \includegraphics[width=\linewidth]{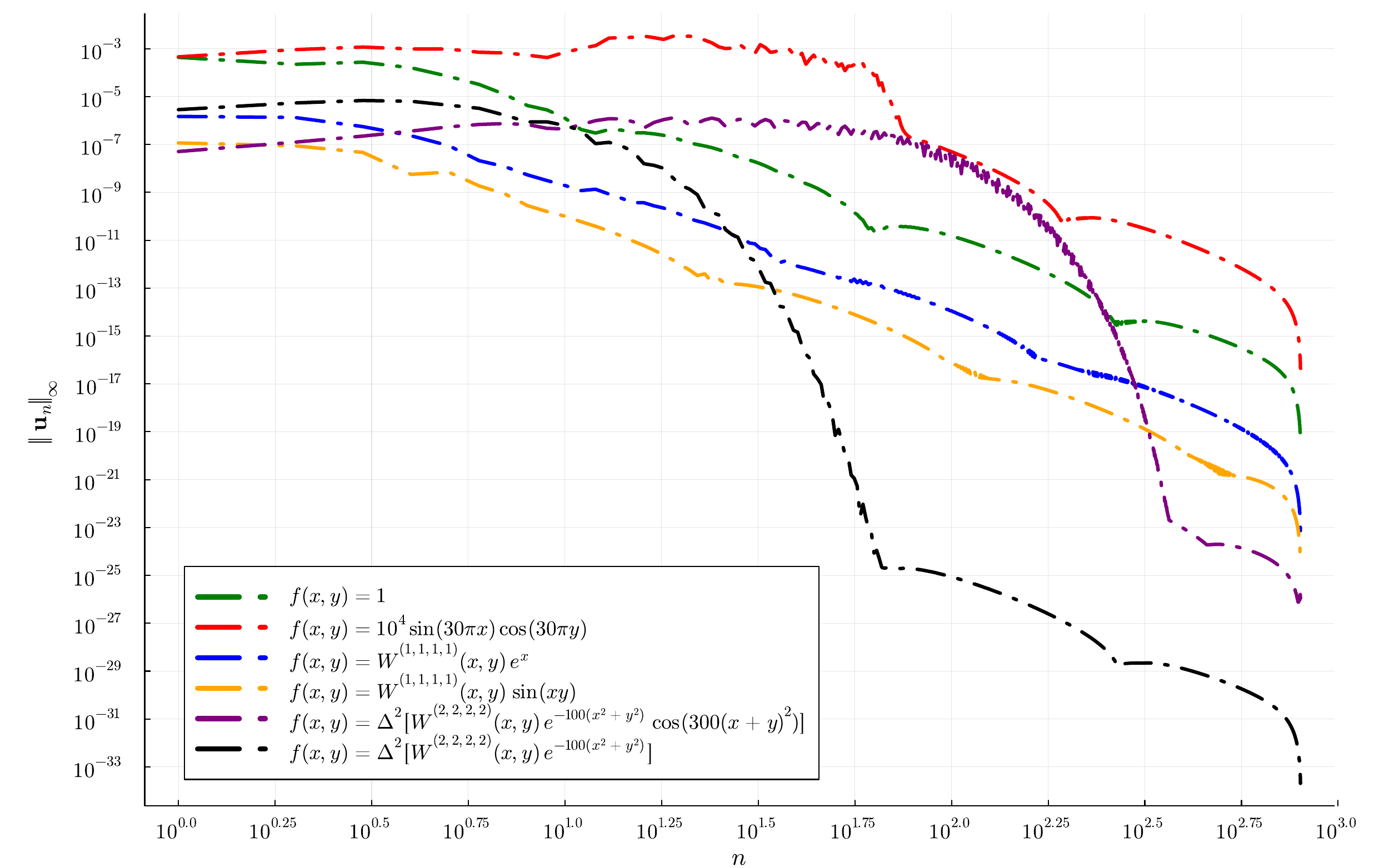}
  \end{minipage}

        \caption{Left: The approximate solution to \(\Delta^2 u = f\), subject to zero Dirichlet and Neumann conditions, with \(f(x,y) = 10^4 \sin(30\pi x) \cos(30\pi y)\). Right: the norms of each block of calculated coefficients of the approximate solution to the biharmonic equation for various RHS functions. }
   \label{fig:biharmonic_EX}
\end{figure}

\section{Conclusions and future work}

We describe fast algorithms using generalised Koornwinder polynomials,  
which allow us to develop a sparse spectral method with high accuracy and rapid convergence for solving general linear PDEs on complicated 2D geometries bounded by algebraic curves.  To the best of our knowledge, this is the only known global spectral method to obtain sparse matrices for PDEs with polynomial coefficients on such geometries.
 This work serves as a stepping stone to constructing sparse spectral methods on 3D geometries bounded by algebraic surfaces, such as star domains and surfaces of revolution associated with the generalised Koornwinder domains discussed in this paper.
Another extension is to piece together semiclassical elements and develop a sparse spectral element method, for which some foundational work has already been done in \cite{VandenHeuvel}.
These methods have potential applications in numerical weather
prediction \cite{weather}, acoustic and elastic wave propagation \cite{komatitsch2000simulation} and medical imaging \cite{marty2024transcranial}.

  \bibliographystyle{abbrvurl}
\bibliography{references}

\appendix 
\section{Generalised Koornwinder polynomials on smooth domains} \label{Generalised_Koornwinder_polynomials_smooth_domains}

It can be shown that for a generalised Koornwinder domain (\ref{eq:Omega}) to be smooth, $\rho$ must take the form
\begin{equation*}
\rho = \sqrt{(x-\alpha)(\beta - x)\phi(x)}, \qquad x \in [\alpha, \beta],
\end{equation*}
where $\phi$ is a polynomial satisfying $\phi(x) > 0$ for $x \in [\alpha, \beta]$, and thus the domain is
\begin{equation*}
\Omega = \left\lbrace (x,y) \in \mathbb{R}^2 \: :\: y^2 \leq \delta^2\rho^2 = \delta^2(x-\alpha)(\beta - x)\phi(x) \right\rbrace, \qquad \delta > 0.
\end{equation*}
For any other generalised Koornwinder domain,
   $\partial \Omega$ has corners at $(x, y) = (\alpha,\gamma\rho(\alpha))$, $(\alpha,\delta\rho(\alpha))$ or $(\beta,\gamma\rho(\beta))$, $(\beta,\delta\rho(\beta))$.

We require that the generalised Koornwinder polynomials on $\Omega$ be orthogonal with respect to the weight
\begin{eqnarray*}
W^{(a)}(x,y) = (\delta^2\rho^2 - y^2)^a = \rho^{2a}\left(\delta^2 - \frac{y^2}{\rho^2}\right)^a = w_R^{(0,0,2a)}(x)w_P^{(a,a)}\left(\frac{y}{\rho}\right), \\
  \end{eqnarray*}
where $a > -1$, the univariate weights $ w_R^{(0,0,2a)}$ and $w_P^{(a,a)}$ are defined in (\ref{eq:wRwPdef}) and they are supported on $[\alpha, \beta]$ and $[-\delta, \delta]$, respectively.   It now follows from Theorem~\ref{thm:genKoornwinder} that the one-parameter family of generalised Koornwinder polynomials
\begin{equation}\label{KoornwinderFomulaDegenerate}
	H_{n,k}^{(a)}(x,y) := R_{n-k}^{(0,0,2a+2k+1)}(x) \, \rho(x)^k \, \widetilde{P}_k^{(a,a)}\left(\frac{y}{\rho(x)}\right), \quad (x,y) \in \Omega, \quad 0\leq k \leq n, n \geq 0, 
\end{equation}
form an orthonormal basis with respect to the inner product
\begin{equation*}
	\langle f, \, g \rangle_{W^{(a)}} := \iint_\Omega f(x,y) \, g(x,y) \,W^{(a)}(x,y).
\end{equation*}
Note that the one-parameter family of Koornwinder polynomials (\ref{KoornwinderFomulaDegenerate}) corresponds to the four-parameter family $H_{n,k}^{(a,b,c,d)}$, defined in  (\ref{KoornwinderFomula}), with $a = 0 =b$ and $c = d := a$,  i.e., $H_{n,k}^{(a)} := H_{n,k}^{(0,0,a,a)}$.  Similarly, for the general weight $W^{(a,b,c,d)}$ defined in (\ref{eq:Wdef}) with $\gamma := -\delta$, we have that $W^{(a)}(x,y):= W^{(0,0,a,a)}(x,y)$.  Hence, on a smooth domain, the four-parameter family of generalised Koornwinder polynomials reduces to a one-parameter family of polynomials.

                                  The quasi-matrix of generalised Koornwinder polynomials on $\Omega$ is denoted by $\mathbf{H}^{(a)}(x,y)$ and the weighted generalised Koornwinder polynomials are defined by
\begin{align*}
	\mathbf{W}^{(a)}(x,y) := W^{(a)}(x,y) \, \mathbf{H}^{(a)}(x,y).
\end{align*}
 The differentiation matrices $D_{x}^{(a)},  D_y^{(a)},  W_{x}^{(a)}$ and $W_y^{(a)}$ are defined by
	\begin{align*}
		{\partial  \over \partial x}\mathbf{H}^{(a)}(x,y) &= \mathbf{H}^{(a+1)}(x,y) \: D_{x}^{(a)}, \quad
		{\partial  \over \partial y}\mathbf{H}^{(a)}(x,y) = \mathbf{H}^{(a+1)}(x,y) \: D_y^{(a)}, \\
		{\partial \over \partial x}\mathbf{W}^{(a)}(x,y) &= \mathbf{W}^{(a-1)}(x,y) \: W_{x}^{(a)},  \quad 
		{\partial \over \partial y}\mathbf{W}^{(a)}(x,y) = \mathbf{W}^{(a-1)}(x,y) \: W_y^{(a)}.
	\end{align*}
	Using these differentiation matrices, we construct the following matrix Laplacian, which we shall use in our numerical experiments
 \begin{equation}
\Delta \mathbf{W}^{(1)}(x,y) = \mathbf{H}^{(1)}(x,y)\Delta_{W,(1)}^{(1)}, \qquad \Delta_{W,(1)}^{(1)} = D_x^{(0)}\,W_x^{(1)} 
+ D_y^{(0)}\,W_y^{(1)}   \label{Laplacian_Operator_degenerate}
\end{equation}

\begin{proposition}\label{prop:smoothoperator}
	The above differentiation matrices and matrix Laplacian are sparse, with the following bandwidths: 
	\begin{itemize}
		\item $D_{x}^{(a)}$ has block-bandwidths $(\deg \rho^2 - 3,\deg \rho^2 - 1)$ and each block has bandwidths $(0,2)$.
		\item $D_y^{(a)}$ has block-bandwidths $(-1,1)$ and each block has bandwidths $(-1,1)$.
		\item $W_{x}^{(a)}$ has block-bandwidths $(\deg \rho^2 - 1,\deg \rho^2 - 3)$ and each block has bandwidths $(2,0)$.
		\item $W_y^{(a)}$ has block-bandwidths $(1,-1)$ and each block has bandwidths $(1,-1)$.
		\item $\Delta_{W,(1)}^{(1)}$ has block-bandwidths $(2\deg\rho^2 - 4, 2\deg\rho^2 - 4)$ and each block has bandwidths $(2,2)$.
	\end{itemize}
\end{proposition}
The entries of these matrices can be derived using the same approach used to derive the entries of the differentiation matrices in (\ref{eq:PartialXdef})--(\ref{eq:Partialdef}).

\section{Proof of Proposition~\ref{bandwidths_diff}} \label{AppendixA}
\begin{proof}

    Integrating by parts and using the semiclassical Pearson equation (\ref{PearsonEq})--(\ref{eq:sigtauR}) (but with $c+d+1$ replaced with $c+d+2k+1$), we have that	
\begin{eqnarray}
& & \left(D_\eta(w^{(a,b,c+d+2k+1)}_{R})\right)_{[m,n]} \nonumber \\
& =&  \left\langle R_{n}^{(a,b,c+d+2k+1)}{}',  R_{m}^{(a+1,b+1,c+d+2k+2+\eta)}     \right\rangle_{w_R^{(a+1,b+1,c+d+2k+2+\eta)}} \label{eq:scdmip1} \\
& = & \int_{\alpha}^{\beta}   R_{n}^{(a,b,c+d+2k+1)}{}'\, R_{m}^{(a+1,b+1,c+d+2k+2+\eta)} w^{(a+1,b+1,c+d+2k+2+\eta)}_{R} \,\mathrm{d}x \nonumber \\
& =& -\int_{\alpha}^{\beta}   R_{n}^{(a,b,c+d+2k+1)} \frac{\mathrm{d}}{\mathrm{d}x}\left( R_{m}^{(a+1,b+1,c+d+2k+2+\eta)} w^{(a+1,b+1,c+d+2k+2+\eta)}_{R}\right) \,\mathrm{d}x \nonumber \\
& =& -\int_{\alpha}^{\beta}   R_{n}^{(a,b,c+d+2k+1)} \frac{\mathrm{d}}{\mathrm{d}x}\left( R_{m}^{(a+1,b+1,c+d+2k+2+\eta)} \sigma_{R} w^{(a,b,c+d+2k+1)}_{R}\right) \,\mathrm{d}x\nonumber \\
& =& -\int_{\alpha}^{\beta}   R_{n}^{(a,b,c+d+2k+1)}  R_{m}^{(a+1,b+1,c+d+2k+2+\eta)}{}' \sigma_{R} w^{(a,b,c+d+2k+1)}_{R}\,\mathrm{d}x \nonumber \\
&&  - \int_{\alpha}^{\beta}   R_{n}^{(a,b,c+d+2k+1)}  R_{m}^{(a+1,b+1,c+d+2k+2+\eta)} \tau_{R} w^{(a,b,c+d+2k+1)}_{R} \,\mathrm{d}x \nonumber \\
& = & -\left\langle R_{n}^{(a,b,c+d+2k+1)}, p(x) \right\rangle_{w_R^{(a,b,c+d+2k+1)}}  \label{eq:scdmip2}
\end{eqnarray}
where
\begin{equation*}
p(x) = R_{m}^{(a+1,b+1,c+d+2k+2+\eta)}{}' \sigma_{R} +  R_{m}^{(a+1,b+1,c+d+2k+2+\eta)} \tau_{R}.
\end{equation*}
By orthogonality, (\ref{eq:scdmip1}) can only be nonzero if $m \leq \deg R_{n}^{(a,b,c+d+2k+1)}{}' =  n -1$ and (\ref{eq:scdmip2}) can only be nonzero if $n \leq \deg p(x) = m + 1 + d_{\eta}$.  Hence, $\left(D_\eta(w^{(a,b,c+d+2k+1)}_{R})\right)_{[m,n]}$ can only be nonzero if  $n- 1 - d_{\eta}  \leq m \leq n - 1$ and therefore $D_\eta(w^{(a,b,c+d+2k+1)}_{R})$ has bandwidths $(-1, 1 + d_\eta)$.
\end{proof}

\section{Proof of Propositions~\ref{Diff_matrix} and~\ref{weighted_Diff_matrix}} \label{AppendixB}
\begin{proof}
In case (i), it follows from (\ref{eq:case1koorntomono}) that $H_{n,k}^{(a,b,c,d)}$ is a degree-$k$ polynomial in $y$ and it has the leading-order monomial expansion
\begin{equation*}
H_{n,k}^{(a,b,c,d)} = \widetilde{c}_0x^{n+k(d_0-1)} + \widetilde{c}_1x^{n+(k-1)(d_0-1)-1}y + \widetilde{c}_2x^{n+(k-2)(d_0-1)-2}y^2 + \cdots + c_kx^{n-k}y^k + \cdots,
\end{equation*}
hence, by (\ref{eq:case1monotokoorn}), $\partial_x H_{n,k}^{(a,b,c,d)}$ has the following expansion in the $\left\lbrace H_{m,\ell}^{(a+1,b+1,c+1,d+1)} \right\rbrace$-basis
 \begin{equation}
 {\partial  \over \partial x} H_{n,k}^{(a,b,c,d)}(x,y) = \sum_{\ell = 0}^{k} \sum_{m=\ell}^{n+(k-\ell)(d_0-1)-1} c_{m,\ell}^{\,x, (n,k)} \,H_{m,\ell}^{(a+1,b+1,c+1,d+1)}(x,y).  
 \label{eq:case1partialxexp}
\end{equation}
         In case (ii), a similar argument using (\ref{eq:case2koorntomono})--(\ref{eq:case2monotokoorn}) implies  that
\begin{equation}
{\partial  \over \partial x} H_{n,k}^{(a,b,c,d)}(x,y) = \sum_{{\substack{\ell = 0\\ \mathrm{{mod}(\ell,2)\,=\,\mathrm{mod}(k,2)}}}}^{k} \sum_{m=\ell}^{n+(k-\ell)(d_1-2)/2-1} c_{m,\ell}^{\,x, (n,k)} \,H_{m,\ell}^{(a+1,b+1,c+1,d+1)}(x,y). 
 \label{eq:case2partialxexp}
\end{equation}
          	Applying the change of variable $t = y/\rho$ and using (\ref{ClassicalRaising})--(\ref{ClassicalDerivative}), the expansion coefficients in (\ref{eq:case1partialxexp}) and (\ref{eq:case2partialxexp}) are given by
	\begin{align}
		c_{m,\ell}^{\,x, (n,k)} &= \left\langle \,{\partial  \over \partial x} H_{n,k}^{(a,b,c,d)} , \, H_{m,\ell}^{(a+1,b+1,c+1,d+1)} \,\right\rangle_{W^{(a+1,b+1,c+1,d+1)}} \nonumber\\
	   &= \Big( \int_\alpha^\beta \: R_{n-k}^{(a,b,c+d+2k+1)\:\prime}(x) \: \rho(x)^{k+\ell+1} \:R_{m-\ell}^{(a+1,b+1,c+d+2\ell+3)}(x) \: w^{(a+1,b+1,c+d+2)}_R(x) \: \mathrm{d}x \Big) \nonumber \\
	   &\quad \quad \quad \quad \quad \quad \cdot \: \Big( \int_{\gamma}^{\delta} \: \widetilde{P}_k^{(d,c)}(t) \: \widetilde{P}_\ell^{(d+1,c+1)}(t) \: w^{(d+1,c+1)}_P(t) \: \mathrm{d}t \Big) \nonumber \\
	   & + \Big( \int_\alpha^\beta \: R_{n-k}^{(a,b,c+d+2k+1)}(x) \: \rho(x)^{k+\ell}\:\rho^\prime(x) \:R_{m-\ell}^{(a+1,b+1,c+d+2\ell+3)}(x) \: w^{(a+1,b+1,c+d+2)}_R(x) \: \mathrm{d}x \Big)  \nonumber \\
	   &\quad \quad \quad \quad \quad \quad \cdot \: \Big( \int_{\gamma}^{\delta} \: \left[k\widetilde{P}_k^{(d,c)}(t) - t\, \widetilde{P}_k^{(d,c)\:\prime}(t) \right] \: \widetilde{P}_\ell^{(d+1,c+1)}(t) \: w^{(d+1,c+1)}_P(t) \: \mathrm{d}t \Big) \label{eq:2ndterm} \\
    	   &= \langle \,R_{n-k}^{(a,b,c+d+2k+1)\:\prime}, \, \rho^{k+\ell+1} \:R_{m-\ell}^{(a+1,b+1,c+d+2\ell+3)} \, \rangle_{w^{(a+1,b+1,c+d+2)}_R} \nonumber \\ 
	   &\quad \quad  \cdot \: \langle r^{(d+1,c+1)}_{k}\,\widetilde{P}^{(d+1,c+1)}_{k} - t^{(d+1,c+1)}_{k-1}\,\widetilde{P}^{(d+1,c+1)}_{k-1} -s^{(d+1,c+1)}_{k-2}\,\widetilde{P}^{(d+1,c+1)}_{k-2},\, \widetilde{P}_\ell^{(d+1,c+1)} \rangle_{w^{(d+1,c+1)}_P} \label{eq:jacip1}\\
	   &+ k\, \langle \,R_{n-k}^{(a,b,c+d+2k+1)}, \, \rho^{k+\ell}\:\rho^\prime \:R_{m-\ell}^{(a+1,b+1,c+d+2\ell+3)} \, \rangle_{w^{(a+1,b+1,c+d+2)}_R} \nonumber \\ 
	   &\quad \quad  \cdot \: \langle r^{(d+1,c+1)}_{k}\,\widetilde{P}^{(d+1,c+1)}_{k} - t^{(d+1,c+1)}_{k-1}\,\widetilde{P}^{(d+1,c+1)}_{k-1} -s^{(d+1,c+1)}_{k-2}\,\widetilde{P}^{(d+1,c+1)}_{k-2},\, \widetilde{P}_\ell^{(d+1,c+1)} \rangle_{w^{(d+1,c+1)}_P} \label{eq:jacip2}\\
	   &- d_{k-1}^{(d+1,c+1)}\, \langle \,R_{n-k}^{(a,b,c+d+2k+1)}, \, \rho^{k+\ell}\:\rho^\prime \:R_{m-\ell}^{(a+1,b+1,c+d+2\ell+3)} \, \rangle_{w^{(a+1,b+1,c+d+2)}_R} \nonumber \\ 
	   &\quad \quad  \cdot \: \langle \delta_{k-1}^{(d+1,c+1)}\,\widetilde{P}_{k}^{(d+1,c+1)} + \gamma_{k-1}^{(d+1,c+1)} \,\widetilde{P}_{k-1}^{(d, c)} + \delta_{k-2}^{(d+1,c+1)}\,\widetilde{P}_{k-2}^{(d+1,c+1)},\, \widetilde{P}_\ell^{(d+1,c+1)} \rangle_{w^{(d+1,c+1)}_P}.  \label{eq:jacip3}
   \end{align}	
The inner products (\ref{eq:jacip1})--(\ref{eq:jacip3}) (and thus also $c_{m,\ell}^{\,x, (n,k)}$) can be  nonzero only for $\ell = k, k-1, k-2$.  In case (ii), since $k$ and $\ell$ must have the same parity, see (\ref{eq:case2partialxexp}),  $c_{m,\ell}^{\,x, (n,k)} = 0$ if $\ell = k - 1$.
   
   When $\ell=k$, the integral (\ref{eq:2ndterm}) is zero since $k\widetilde{P}_k^{(d,c)}(t) - t\, \widetilde{P}_k^{(d,c)\:\prime}(t)$ is a polynomial of degree at most $k-1$.
     Thus, (\ref{eq:jacip1})--(\ref{eq:jacip3}) simplifies to
   \begin{align}
	c_{m,k}^{\,x, (n,k)} &= r_k^{(d+1,c+1)}\,\langle \,R_{n-k}^{(a,b,c+d+2k+1)\:\prime}, \, R_{m-k}^{(a+1,b+1,c+d+2k+3)} \, \rangle_{w^{(a+1,b+1,c+d+2k+3)}_R}. \nonumber 
     \end{align}
Using (\ref{DMdef}) and (\ref{SemiclassicalRaisingfomula}), it follows that this inner product are entries of the following matrix
\begin{eqnarray*}
& & \left\langle \mathbf{R}^{(a+1,b+1,c+d+2k+3)}, \frac{\mathrm{d}}{\mathrm{d}x} \mathbf{R}^{(a,b,c+d+2k+1)}  \right\rangle_{w^{(a+1,b+1,c+d+2k+3)}_R} \\
&&  = \left\langle \mathbf{R}^{(a+1,b+1,c+d+2k+3)}, \mathbf{R}^{(a+1,b+1,c+d+2k+2+\eta)}\,D_\eta(w^{(a,b,c+d+2k+1)}_R) \right\rangle_{w^{(a+1,b+1,c+d+2k+3)}_R}  \\
&& = \left\langle \mathbf{R}^{(a+1,b+1,c+d+2k+3)}, \mathbf{R}^{(a+1,b+1,c+d+2k+3)}\mathcal{T}_{(a+1,b+1,c+d+2k+2+\eta)}^{(a+1,b+1,c+d+2k+3)}\,D_\eta(w^{(a,b,c+d+2k+1)}_R) \right\rangle_{w^{(a+1,b+1,c+d+2k+3)}_R} \\
&& = \mathcal{T}_{(a+1,b+1,c+d+2k+2+\eta)}^{(a+1,b+1,c+d+2k+3)}\,D_\eta(w^{(a,b,c+d+2k+1)}_R).
\end{eqnarray*} 
Hence, 
   \begin{align}
	c_{m,k}^{\,x, (n,k)} &= \Big( r_k^{(d+1,c+1)}\,\mathcal{T}_{(a+1,b+1,c+d+2k+2+\eta)}^{(a+1,b+1,c+d+2k+3)}\,{D_{\eta}(w^{(a,b,c+d+2k+1)}_{R})}  \Big)_{[m-k,n-k]}.  \label{eq:partialxcmk}
   \end{align}   

When $\ell=k-2$, (\ref{eq:jacip1})--(\ref{eq:jacip3})  reduce to
   \begin{align} 
	&c_{m,k-2}^{\,x, (n,k)} = -\, s_{k-2}^{(d+1, c+1)}\,\langle \,R_{n-k}^{(a,b,c+d+2k+1)\:\prime}, \, \rho^2\,R_{m-k+2}^{(a+1,b+1,c+d+2k-1)} \, \rangle_{w^{(a+1,b+1,c+d+2k-1)}_R} \nonumber\\
				  &\quad - (k\,s_{k-2}^{(d+1, c+1)}+\delta_{k-2}^{(d+1,c+1)}\,d_{k-1}^{(d+1, c+1)})\,\langle \,R_{n-k}^{(a,b,c+d+2k+1)}, \, \rho\rho' \,R_{m-k+2}^{(a+1,b+1,c+d+2k-1)} \, \rangle_{w^{(a+1,b+1,c+d+2k-1)}_R} \label{eq:km2ip2}
      \end{align}
The first inner product is computed by first lowering (via the inverse of a raising matrix) and then differentiating  the OPs $\left\lbrace R_{n}^{(a,b,c+d+2k+1)} \right\rbrace$: using (\ref{eq:semiclassjacdef}), (\ref{DMdef}) and (\ref{SemiclassicalRaisingfomula}), we have   
          \begin{eqnarray*}
& & \left\langle \rho^2 \mathbf{R}^{(a+1,b+1,c+d+2k-1)}, \frac{\mathrm{d}}{\mathrm{d}x} \mathbf{R}^{(a,b,c+d+2k+1)}  \right\rangle_{w^{(a+1,b+1,c+d+2k-1)}_R} \\
&&  = \left\langle \mathbf{R}^{(a+1,b+1,c+d+2k-1)}\rho^2 \left(J(w^{(a+1,b+1,c+d+2k-1)}_{R})\right), \right. \\
&&  \qquad\left. \frac{\mathrm{d}}{\mathrm{d}x} \mathbf{R}^{(a,b,c+d+2k-2-\eta)}\,\left(\mathcal{T}_{(a,b, c+d+2k-2-\eta)}^{(a,b,c+d+2k+1)}\right)^{-1}\ \right\rangle_{w^{(a+1,b+1,c+d+2k-1)}_R}   \\
&&  = \left\langle \mathbf{R}^{(a+1,b+1,c+d+2k-1)}\rho^2 \left(J(w^{(a+1,b+1,c+d+2k-1)}_{R})\right), \right. \\
&&  \qquad\left. \mathbf{R}^{(a+1,b+1,c+d+2k-1)}\,D_\eta(w^{(a,b,c+d+2k-2-\eta)}_R)\left(\mathcal{T}_{(a,b, c+d+2k-2-\eta)}^{(a,b,c+d+2k+1)}\right)^{-1}\ \right\rangle_{w^{(a+1,b+1,c+d+2k-1)}_R} \\
&& = \rho^2 \left(J(w^{(a+1,b+1,c+d+2k-1)}_{R})\right) D_\eta(w^{(a,b,c+d+2k-2-\eta)}_R)\left(\mathcal{T}_{(a,b, c+d+2k-2-\eta)}^{(a,b,c+d+2k+1)}\right)^{-1}.
\end{eqnarray*} 
As we shall show, we only need to compute a banded portion of this matrix and thus we use Lemma~\ref{thm:MatrixProduct} to compute its required entries in linear complexity.  

The second inner product  (\ref{eq:km2ip2}) is computed by first lowering the third parameter and then raising the first and second parameters of the OPs $\left\lbrace R_{n}^{(a,b,c+d+2k+1)} \right\rbrace$:
\begin{eqnarray*}
& & \left\langle \rho\rho' \mathbf{R}^{(a+1,b+1,c+d+2k-1)}, \mathbf{R}^{(a,b,c+d+2k+1)}  \right\rangle_{w^{(a+1,b+1,c+d+2k-1)}_R} \\
&& = (\rho\rho') \left(J(w^{(a+1,b+1,c+d+2k-1)}_{R})\right) \mathcal{T}_{(a,b, c+d+2k+1)}^{(a+1,b+1,c+d+2k+1)}\left(\mathcal{T}_{(a,b, c+d+2k-1)}^{(a,b,c+d+2k+1)}\right)^{-1},
\end{eqnarray*} 
where $(\rho\rho')(\cdot) \equiv \rho(\cdot)\rho'(\cdot)$, and again the entries of this matrix are computed in linear complexity using Lemma~\ref{thm:MatrixProduct}.

Hence, (\ref{eq:km2ip2}) can be expressed as
  \begin{align}
	&c_{m,k-2}^{\,x, (n,k)} = 
				   \Big(-s_{k-2}^{(d+1,c+1)}\, \rho^2 \left(J(w^{(a+1,b+1,c+d+2k-1)}_{R})\right) \,D_{\eta}(w^{(a,b,c+d+2k-2-\eta)}_{R})\,\left(\mathcal{T}_{(a,b, c+d+2k-2-\eta)}^{(a,b,c+d+2k+1)}\right)^{-1} \nonumber \\
				  &\quad - (k\,s_{k-2}^{(d+1, c+1)}+\delta_{k-2}^{(d+1,c+1)}\,d_{k-1}^{(d+1, c+1)})\,(\rho \rho')\left(J(w^{(a+1,b+1,c+d+2k-1)}_{R})\right) \,\mathcal{T}_{(a,b,c+d+2k-1)}^{(a+1,b+1,c+d+2k-1)}\nonumber\\
				  &\qquad \qquad \cdot  \,\left(\mathcal{T}_{(a,b, c+d+2k-1)}^{(a,b,c+d+2k+1)}\right)^{-1}\Big)_{[m-k+2,n-k]}.  \label{remark}
   \end{align}

When $\ell=k-1$, $c_{m,k-1}^{\,x, (n,k)} = 0$ in case (ii), hence we only need to consider case (i), for which (\ref{eq:jacip1})--(\ref{eq:jacip3}) reduce to
   \begin{align}
		&c_{m,k-1}^{\,x, (n,k)} = -t_{k-1}^{(d+1,c+1)}\,\langle \,R_{n-k}^{(a,b,c+d+2k+1)\:\prime}, \, \rho\,R_{m-k+1}^{(a+1,b+1,c+d+2k+1)} \, \rangle_{w^{(a+1,b+1,c+d+2k+1)}_R}\nonumber \\
		&\quad - (k\,t_{k-1}^{(d+1, c+1)}+\gamma_{k-1}^{(d+1,c+1)}\,d_{k-1}^{(d+1, c+1)})\,\langle \,R_{n-k}^{(a,b,c+d+2k+1)}, \, \rho' \,R_{m-k+1}^{(a+1,b+1,c+d+2k+1)} \, \rangle_{w^{(a+1,b+1,c+d+2k+1)}_R}. \nonumber
     \end{align}
Similar to the derivation of (\ref{remark}), we can express this in terms of entries of univariate operator matrices as follows,
 \begin{align}
		&c_{m,k-1}^{\,x, (n,k)} =  \Big(-t_{k-1}^{(d+1,c+1)}\,\rho \left(J(w^{(a+1,b+1,c+d+2k+1)}_{R})\right) \,D_{0}(w^{(a,b,c+d+2k)}_{R})\,\left(\mathcal{T}_{(a,b,c+d+2k)}^{(a,b,c+d+2k+1)}\right)^{-1} \nonumber  \\
		&\quad - (k\,t_{k-1}^{(d+1, c+1)}+\gamma_{k-1}^{(d+1,c+1)}\,d_{k-1}^{(d+1, c+1)})\,\rho' \left(J(w^{(a+1,b+1,c+d+2k+1)}_{R})\right) \,\mathcal{T}_{(a,b,c+d+2k+1)}^{(a+1,b+1,c+d+2k+1)}\Big)_{[m-k+1,n-k]},  \label{eq:partialxcmkm1}
   \end{align}
where the first term is computed using Lemma~\ref{thm:MatrixProduct}. 
 
To determine the values of $m$ for which the coefficients $c_{m,\ell}^{\,x, (n,k)}$ are nonzero and thereby determine the block-bandwidths of $D_x^{(a,b,c,d)}$, we integrate by parts and arrive at
   \begin{align}
	&c_{m,\ell}^{\,x, (n,k)} = \left\langle \,{\partial  \over \partial x} H_{n,k}^{(a,b,c,d)} , \, H_{m,\ell}^{(a+1,b+1,c+1,d+1)} \,\right\rangle_{W^{(a+1,b+1,c+1,d+1)}} \label{DerivativeXCoeff1}\\
				  &= 
 \iint_{\Omega}				  {\partial  \over \partial x}H_{n,k}^{(a,b,c,d)}\: H_{m,\ell}^{(a+1,b+1,c+1,d+1)}\:W^{(a+1,b+1,c+1,d+1)}\: \mathrm{d}x   \: \mathrm{d}y \nonumber\\
				  &= -\iint_{\Omega} H_{n,k}^{(a,b,c,d)}\: [W^{(1,1,1,1)}\:{\partial  \over \partial x}H_{m,\ell}^{(a+1,b+1,c+1,d+1)}]\:W^{(a,b,c,d)}\: \mathrm{d}x   \: \mathrm{d}y \nonumber\\
				  &\quad +(a+1)\iint_{\Omega} H_{n,k}^{(a,b,c,d)}\: [W^{(0,1,1,1)}\:H_{m,\ell}^{(a+1,b+1,c+1,d+1)}]\:W^{(a,b,c,d)}\: \mathrm{d}x   \: \mathrm{d}y \nonumber\\
				  &\quad -(b+1)\iint_{\Omega} H_{n,k}^{(a,b,c,d)}\: [W^{(1,0,1,1)}\:H_{m,\ell}^{(a+1,b+1,c+1,d+1)}]\:W^{(a,b,c,d)}\: \mathrm{d}x   \: \mathrm{d}y \nonumber\\
				  &\quad +\gamma(c+1)\iint_{\Omega} H_{n,k}^{(a,b,c,d)}\:(\delta \rho - y)  \rho'\:W^{(1,1,0,0)}\:H_{m,\ell}^{(a+1,b+1,c+1,d+1)}]\:W^{(a,b,c,d)}\: \mathrm{d}x   \: \mathrm{d}y \nonumber\\
				  &\quad -\delta(d+1)\iint_{\Omega} H_{n,k}^{(a,b,c,d)}\: (y-\gamma \rho)\rho'\:W^{(1,1,0,0)}\:H_{m,\ell}^{(a+1,b+1,c+1,d+1)}]\:W^{(a,b,c,d)}\: \mathrm{d}x   \: \mathrm{d}y \nonumber\\
				  &=
  -\left\langle\, H_{n,k}^{(a,b,c,d)},p(x,y) \right\rangle_{W^{(a,b,c,d)}}, \label{DerivativeXCoeff2}
   \end{align}
where
\begin{eqnarray*}
 p(x,y) &=& W^{(1,1,1,1)}\:{\partial  \over \partial x}H_{m,\ell}^{(a+1,b+1,c+1,d+1)} - \left[(a+1)\:W^{(0,1,1,1)} - (b+1)\:W^{(1,0,1,1)} \right]\:H_{m,\ell}^{(a+1,b+1,c+1,d+1)}   \\
& & - \left[ \left(\gamma(c+1)(\delta \rho - y) -\delta(d+1)  (y-\gamma \rho)  \right)\rho'\:W^{(1,1,0,0)}\right]\:H_{m,\ell}^{(a+1,b+1,c+1,d+1)}.
\end{eqnarray*}

In case (i), it follows from (\ref{eq:case1koorntomono}) and (\ref{eq:Wdef}) that $p(x, y)$ is a degree-$(\ell+2)$ polynomial in $y$ and it has the leading-order monomial expansion
\begin{eqnarray*}
 p(x,y) &=& \widetilde{c}_0x^{m+\ell(d_0-1)+2d_0 + 1} + \widetilde{c}_1x^{m+\ell(d_0-1)+d_0 + 1}y + \widetilde{c}_2x^{m+\ell(d_0-1)+ 1}y^2 + \cdots + \\
& &  \widetilde{c}_\ell x^{m+2d_0+1-\ell}y^\ell + \widetilde{c}_{\ell+1} x^{m+d_0+1-\ell}y^{\ell+1} + \widetilde{c}_{\ell+2} x^{m+1-\ell}y^{\ell+2} + \cdots,
\end{eqnarray*}
and hence by (\ref{eq:case1monotokoorn}), $p(x, y)$ has an expansion in the $\left\lbrace H^{(a,b,c,d)}_{n,k} \right\rbrace$-basis given by
\begin{equation}
p(x,y) = \sum_{j = 0}^{\ell + 2} \:\sum_{i = j}^{m + (\ell - j)(d_0 -1)+2d_0 +1}  c_{i,j}H_{i,j}^{(a,b,c,d)}. \label{eq:case1pexp}
\end{equation}

Recall that in case (i), the coefficients (\ref{DerivativeXCoeff1}) can only be nonzero for $\ell = k-2, k-1, k$.  Setting $\ell = k - 2$ in (\ref{eq:case1pexp}) it follows that the inner product (\ref{DerivativeXCoeff2}) can only be nonzero if $j = k$ in (\ref{eq:case1pexp}) and $m + (\ell - j)(d_0 -1)+2d_0 +1 =  m + 3 \geq n$;  setting $\ell = k - 2$ in (\ref{eq:case1partialxexp}), it follows that $m \leq n + 2d_0 - 3$.  Thus, $c_{m,k-2}^{\,x, (n,k)}$, whose expression is given in (\ref{remark}),  can only be nonzero for $n - 3 \leq m \leq n + 2d_0 - 3$ and hence the summation limits of the second summation in Proposition~\ref{Diff_matrix}.  Setting $\ell = k - 1$,  a similar argument shows that   $c_{m,k-1}^{\,x, (n,k)}$ can only be nonzero for $n - d_0 -2 \leq m \leq n + d_0 - 2$ and setting $\ell = k$, it follows that $c_{m,k}^{\,x, (n,k)}$ can only be nonzero for $n - 2d_0 -1 \leq m \leq n - 1$. 
                      
In case (ii), with $p(x, y)$ defined as above but with $c = d$ and $\gamma = -\delta$, one can deduce from the leading-order monomial expansion of $p(x, y)$ and (\ref{eq:case2monotokoorn}) that it has the following expansion
   \begin{equation}
	p(x,y) = \sum_{\substack{j = 0 \\ \mathrm{mod}(j,2)\,=\,\mathrm{mod}(\ell,2)}}^{\ell+2} \: \sum_{i=j}^{m+(\ell-j)(d_1-2)/2+d_1+1} c_{i,j} \,H_{i,j}^{(a,b,c,d)}(x,y).  \label{eq:case2pexp}
\end{equation}

Recall that in case (ii), the coefficients (\ref{DerivativeXCoeff1}) can only be nonzero for $\ell = k-2, k$.  Setting $\ell = k - 2$ in (\ref{eq:case2pexp}) it follows that the inner product (\ref{DerivativeXCoeff2}) can only be nonzero if $j = k$ in (\ref{eq:case2pexp}) and $m + (\ell - j)(d_1 -2)/2+d_1 +1 =  m + 3 \geq n$;  setting $\ell = k - 2$ in (\ref{eq:case2partialxexp}), it follows that $m \leq n + d_1 - 3$ and thus $c_{m,k-2}^{\,x, (n,k)}$  can only be nonzero for $n - 3 \leq m \leq n + d_1 - 3$.  By a similar argument, $c_{m,k}^{\,x, (n,k)}$  can only be nonzero for $n - d_1 -1 \leq m \leq n -1$.  This completes the proof of Proposition~\ref{Diff_matrix}.

To prove Proposition~\ref{weighted_Diff_matrix}, 
 we note that the weighted derivative has the following expansion 
    \begin{equation*}
		{\partial  \over \partial x}[W^{(a,b,c,d)}(x,y)\,H_{n,k}^{(a,b,c,d)}(x,y)] = \sum_{m \geq 0}\sum_{j = 0}^{m} \bar{c}_{m,j}^{\,x, (n,k)} \,W^{(a-1,b-1,c-1,d-1)}(x,y)\,H_{m,j}^{(a-1,b-1,c-1,d-1)}(x,y).
   \end{equation*}
Using the fact that $\left\lbrace W^{(a-1,b-1,c-1,d-1)} H_{m,j}^{(a-1,b-1,c-1,d-1)}  \right\rbrace$ are orthonormal with respect to $\langle \cdot, \cdot \rangle_{W^{(1-a,1-b,1-c,1-d)}}$ and integrating by parts, the expansion coefficients are given by
   \begin{align}
	   \bar{c}_{m,j}^{\,x, (n,k)} &= \left\langle \,{\partial  \over \partial x}(W^{(a,b,c,d)}\,H_{n,k}^{(a,b,c,d)}) , \, W^{(a-1,b-1,c-1,d-1)}\,H_{m,j}^{(a-1,b-1,c-1,d-1)} \,\right\rangle_{W^{(1-a,1-b,1-c,1-d)}} \nonumber\\
	   	   &=  \iint_{\Omega}\: {\partial  \over \partial x}[W^{(a,b,c,d)}(x,y)\,H_{n,k}^{(a,b,c,d)}(x,y)] \: H_{m,j}^{(a-1,b-1,c-1,d-1)}(x,y) \: \mathrm{d}x \: \mathrm{d}y \nonumber \\
	   &=  - \iint_{\Omega} \: {\partial  \over \partial x}H_{m,j}^{(a-1,b-1,c-1,d-1)}(x,y) \: [H_{n,k}^{(a,b,c,d)}(x,y)\:W^{(a,b,c,d)}(x,y)]  \: \mathrm{d}y \: \mathrm{d}x \nonumber \\
	   &= - \left\langle \,{\partial  \over \partial x}H_{m,j}^{(a-1,b-1,c-1,d-1)}, \, H_{n,k}^{(a,b,c,d)} \, \right\rangle_{W^{(a,b,c,d)}}. \nonumber 
      \end{align}	
 The following expressions for the nonzero expansion coefficients and the limits of summation in Proposition~\ref{weighted_Diff_matrix} can be derived in a similar manner as for Proposition~\ref{Diff_matrix}
	\begin{align}
		 \bar{c}_{m,k}^{\,x, (n,k)} =& \Big( - r_k^{(d,c)}\,\mathcal{T}_{(a-1,b-1,c+d+2k+\eta)}^{(a-1,b-1,c+d+2k+1)}\,{D_{\eta}(w^{(a-1,b-1,c+d+2k-1)}_{R})}  \Big)_{[n-k,m-k]},  
		\label{eq:partialxwcmk}\\
 \bar{c}_{m,k+2}^{\,x, (n,k)} = &  \Big(s_{k}^{(d,c)}\,\rho^2 \left(J(w^{(a,b,c+d+2k+1)}_{R})\right) \,D_{\eta}(w^{(a-1,b-1,c+d+2k-\eta)}_{R})\,\left(\mathcal{T}_{(a-1,b-1,c+d+2k-\eta)}^{(a-1,b-1,c+d+2k+3)}\right)^{-1} \nonumber \\
			& - [(k+2)s_{k}^{(d, c)}+\delta_{k}^{(d,c)}d_{k+1}^{(d, c)}](\rho \rho')\left(J(w^{(a,b,c+d+2k+1)}_{R})\right)\cdot \nonumber\\
&			 \mathcal{T}_{(a-1,b-1,c+d+2k+1)}^{(a,b,c+d+2k+1)} \left(\mathcal{T}_{(a-1,b-1,c+d+2k+1)}^{(a-1,b-1,c+d+2k+3)}\right)^{-1}\Big)_{[n-k, m-k-2]} \nonumber \\
\bar{c}_{m,k+1}^{\,x, (n,k)} = & \begin{cases}
			\textnormal{case (i):} &\Big(t_{k}^{(d,c)}\rho \left(J(w^{(a,b,c+d+2k+1)}_{R})\right) D_{0}(w^{(a-1,b-1,c+d+2k)}_{R})\left(\mathcal{T}_{(a-1,b-1,c+d+2k)}^{(a-1,b-1,c+d+2k+1)}\right)^{-1}  \\
			&+ [(k+1)\,t_{k}^{(d, c)}+\gamma_{k}^{(d,c)}\,d_{k}^{(d, c)}]\,\rho' \left(J(w^{(a,b,c+d+2k+1)}_{R})\right) \\
			&\cdot \,\mathcal{T}_{(a-1,b-1,c+d+2k+1)}^{(a,b,c+d+2k+1)}\Big)_{[n-k, m-k-1]},  
			\\
			\textnormal{case (ii):} & 0
		\end{cases}.
		\label{eq:partialxwcmkp1}
	\end{align}
\end{proof}

\section{Proposition~\ref{structure_conversion_mat} proof outline} \label{AppendixC}
\begin{proof}
Proposition~\ref{structure_conversion_mat}  follows from the fact that the generalised Koornwinder polynomials satisfy
\begin{eqnarray}
H_{n,k}^{(a,b,c,d)} &=& \sum_{m=n-2}^{n}  q_{m,k}^{(n,k)} \, H_{m, k}^{(a+1,b+1,c,d)},\label{eq:koornconvexp}\\
H_{n,k}^{(a,b,c,d)} &=& \sum_{m=n-\deg\rho^2}^{n}  r_{m,k}^{(n,k)} \, H_{m, k}^{(a,b,c+1,d+1)} + \sum_{m=n-2}^{n+\deg\rho^2-2}  r_{m,k-2}^{(n,k)} \, H_{m, k-2}^{(a,b,c+1,d+1)}
							\nonumber \\
						& &	 + \sum_{m=n-d_0-1}^{n+d_0-1}  r_{m,k-1}^{(n,k)} \, H_{m, k-1}^{(a,b,c+1,d+1)},\label{eq:koornconvexp2}					
\end{eqnarray}
and the weighted generalised Koornwinder polynomials satisfy
\begin{eqnarray}
&& W^{(a,b,c,d)}\,H_{n,k}^{(a,b,c,d)} = \sum_{m=n}^{n+2} \bar{q}_{m,k}^{(n,k)} \,W^{(a-1,b-1,c,d)}\,H_{m,k}^{(a-1,b-1,c,d)},\label{eq:wkoornexp}\\
&& W^{(a,b,c,d)}H_{n,k}^{(a,b,c,d)} = \sum_{m=n}^{n+\deg\rho^2}  \bar{r}_{m,k}^{(n,k)}  W^{(a,b,c-1,d-1)}H_{m, k}^{(a,b,c-1,d-1)} 
	 + \nonumber \\
	&& \sum_{m=n-\deg\rho^2+2}^{n+2}  \bar{r}_{m,k+2}^{(n,k)}  W^{(a,b,c-1,d-1)}H_{m, k+2}^{(a,b,c-1,d-1)}
 + \sum_{m=n-d_0+1}^{n+d_0+1}  \bar{r}_{m,k+1}^{(n,k)} W^{(a,b,c-1,d-1)}H_{m, k+1}^{(a,b,c-1,d-1)}.\quad\qquad \label{eq:wkoornconvexp}
\end{eqnarray}
The expansion coefficients in (\ref{eq:koornconvexp})--(\ref{eq:wkoornconvexp}) are given by
\begin{eqnarray}
q_{m,k}^{(n,k)} &=& \left( \mathcal{T}_{(a,b,c+d+2k+1)}^{(a+1,b+1,c+d+2k+1)} \right)_{[m-k,n-k]},  \label{eq:expcoeffs1} \\
r_{m,k}^{(n,k)} &=& r_{k}^{(d+1,c+1)}\left(\mathcal{T}_{(a,b,c+d+2k+1)}^{(a,b,c+d+2k+3)} \right)_{[m-k,n-k]}, \label{eq:expcoeffsr}\\
r_{m,k-2}^{(n,k)} &=& -s_{k-2}^{(d+1,c+1)}  \left(\mathcal{T}_{(a,b,c+d+2k-1)}^{(a,b,c+d+2k+1)} \right)_{[n-k,m-k+2]} , \\
r_{m,k-1}^{(n,k)} &=& \begin{cases}
			 -t_{k-1}^{(d+1,c+1)} \times \left( \rho \left(J(w_R^{(a,b,c+d+2k+1)})\right) \right)_{[n-k,m-k+1]}, & \text{in case (i)} \\
			 0,  & \text{in case (ii)}
		\end{cases}, \label{eq:expcoeffsr2} \\
\bar{q}_{m,k}^{(n,k)} &=& \left( \mathcal{T}_{(a-1,b-1,c+d+2k+1)}^{(a,b,c+d+2k+1)} \right)_{[n-k,m-k]}	\label{eq:wexpcoeffs}\\
\bar{r}_{m,k}^{(n,k)} &=& r_{k}^{(d,c)} \left(\mathcal{T}_{(a,b,c+d+2k-1)}^{(a,b,c+d+2k+1)} \right)_{[n-k,m-k]}, \label{eq:rbar}\\
\bar{r}_{m,k+2}^{(n,k)} &=& -s_{k}^{(d,c)}  \left(\mathcal{T}_{(a,b,c+d+2k+1)}^{(a,b,c+d+2k+3)} \right)_{[m-k-2,n-k]} , \\
\bar{r}_{m,k+1}^{(n,k)} &=& \begin{cases}
			 -t_{k}^{(d,c)} \times \left( \rho \left(J(w_R^{(a,b,c+d+2k+1)})\right) \right)_{[n-k,m-k-1]}, & \text{in case (i)},\\
		 0, & \text{in case (ii)}
		\end{cases}.  \label{eq:rbar2}
\end{eqnarray}
The expansion (\ref{eq:koornconvexp}) and coefficients (\ref{eq:expcoeffs1}) are derived by calculating the inner products
\begin{equation*}
H_{n,k}^{(a,b,c,d)} = \sum_{m \geq 0}\sum_{\ell = 0}^{m} q_{m,\ell}^{(n,k)}H^{(a+1,b+1,c,d)}_{m,\ell},\qquad q_{m,\ell}^{(n,k)} = \left\langle H_{n,k}^{(a,b,c,d)}, H^{(a+1,b+1,c,d)}_{m,\ell} \right\rangle_{W^{(a+1,b+1,c,d)}},
\end{equation*}
in a manner similar to the proof of Proposition~\ref{M_y} and likewise, (\ref{eq:koornconvexp2}) and (\ref{eq:expcoeffsr})--(\ref{eq:expcoeffsr2}) can be thus derived.

The expressions (\ref{eq:wkoornexp}) and (\ref{eq:wexpcoeffs}) follow from setting
\begin{equation*}
W^{(a,b,c,d)}H_{n,k}^{(a,b,c,d)} = \sum_{m \geq 0}\sum_{\ell = 0}^{m} \bar{q}_{m,\ell}^{(n,k)}W^{(a-1,b-1,c,d)}H^{(a-1,b-1,c,d)}_{m,\ell},
\end{equation*}
and calculating the inner products
\begin{equation*}
\bar{q}_{m,\ell}^{(n,k)} = \left\langle W^{(a,b,c,d)}H_{n,k}^{(a,b,c,d)}, W^{(a-1,b-1,c,d)}H^{(a-1,b-1,c,d)}_{m,\ell} \right\rangle_{W^{(1-a,1-b,-c,-d)}},
\end{equation*}
and likewise, (\ref{eq:wkoornconvexp}) and (\ref{eq:rbar})--(\ref{eq:rbar2}) can be similarly derived.

\end{proof}

\end{document}